\documentclass[english]{article}
\usepackage[T1]{fontenc}
\usepackage[latin9]{inputenc}
\usepackage{color}
\usepackage{babel}
\usepackage{mathrsfs}
\usepackage{amsmath}
\usepackage{amsthm}
\usepackage{amssymb}
\usepackage{stmaryrd}
\usepackage[unicode=true,pdfusetitle,
 bookmarks=true,bookmarksnumbered=false,bookmarksopen=false,
 breaklinks=false,pdfborder={0 0 1},backref=false,colorlinks=true]
 {hyperref}

\makeatletter
\theoremstyle{plain}
\newtheorem{thm}{\protect\theoremname}
\theoremstyle{definition}
\newtheorem{defn}[thm]{\protect\definitionname}
\theoremstyle{definition}
\newtheorem{example}[thm]{\protect\examplename}
\theoremstyle{plain}
\newtheorem{prop}[thm]{\protect\propositionname}
\theoremstyle{remark}
\newtheorem{rem}[thm]{\protect\remarkname}
\theoremstyle{plain}
\newtheorem{lem}[thm]{\protect\lemmaname}
\theoremstyle{plain}
\newtheorem{cor}[thm]{\protect\corollaryname}
\theoremstyle{plain}
\newtheorem{assumption}[thm]{\protect\assumptionname}

\usepackage{url}
\DeclareMathSymbol{\shortminussymb}{\mathbin}{AMSa}{"39}

\makeatother

\providecommand{\assumptionname}{Assumption}
\providecommand{\corollaryname}{Corollary}
\providecommand{\definitionname}{Definition}
\providecommand{\examplename}{Example}
\providecommand{\lemmaname}{Lemma}
\providecommand{\propositionname}{Proposition}
\providecommand{\remarkname}{Remark}
\providecommand{\theoremname}{Theorem}

\begin{document}
\title{Weak Poincaré Inequalities for Markov chains: theory and applications}
\author{Christophe Andrieu, Anthony Lee, Sam Power, Andi Q. Wang		\\
School of Mathematics, University of Bristol and University of Warwick}
\maketitle
\begin{abstract}
We investigate the application of Weak Poincaré Inequalities (WPI)
to Markov chains to study their rates of convergence and to derive
complexity bounds. At a theoretical level we investigate the necessity
of the existence of WPIs to ensure $\mathrm{L}^{2}$-convergence,
in particular by establishing equivalence with the Resolvent Uniform
Positivity-Improving (RUPI) condition and providing a counterexample.
From a more practical perspective, we extend the celebrated Cheeger's
inequalities to the subgeometric setting, and further apply these
techniques to study random-walk Metropolis algorithms for heavy-tailed
target distributions and to obtain lower bounds on pseudo-marginal
algorithms.
\end{abstract}
\global\long\def\dif{\mathrm{d}}%
\global\long\def\var{\mathrm{var}}%
\global\long\def\R{\mathbb{R}}%
\global\long\def\X{\mathcal{X}}%
\global\long\def\calE{\mathcal{E}}%
\global\long\def\E{\mathsf{E}}%
\global\long\def\Ebb{\mathbb{E}}%

\global\long\def\ELL{\mathrm{L}^{2}}%
\global\long\def\osc{\mathrm{osc}}%
\global\long\def\Id{\mathrm{Id}}%
\global\long\def\essup{\mathrm{ess\,sup}}%
 
\global\long\def\shortminus{\mathrm{\shortminussymb}}%
\global\long\def\muess{\mathrm{ess_{\mu}}}%

\global\long\def\GapR{\mathrm{Gap_{{\rm R}}}}%

\newpage{}

\tableofcontents{}

\section{Introduction}

This paper is concerned with the characterization of the convergence
to equilibrium of (discrete) time homogeneous Markov chains with transition
kernel $P$ to their invariant distributions. Such a distribution
is assumed throughout to exist and be unique, and will be denoted
by $\mu$. Numerous techniques and metrics are possible for this and
we focus here on functional analytic tools to characterize a type
of $\ELL$-convergence to equilibrium using functional inequalities,
particularly Poincaré Inequalities. More precisely for a $\mu$-invariant
kernel $P$ we aim to characterize scenarios where
\[
\left\Vert P^{n}f-\mu\left(f\right)\right\Vert _{2}^{2}\leqslant\gamma\left(n\right)\cdot\Psi\left(f-\mu\left(f\right)\right)\,,
\]
for a vanishing sequence $\gamma=\left\{ \gamma\left(n\right)\colon n\in\mathbb{N}\right\} $
and a type of functional $\Psi$ e.g. $\Psi\left(\cdot\right)=\left\Vert \cdot\right\Vert _{2}^{2}$
or $\Psi\left(\cdot\right)=\left\Vert \cdot\right\Vert _{\mathrm{osc}}^{2}$
-- notation definitions are detailed in Appendix~\ref{subsec:Notation}.
In this paper we will refer to such types of convergence as $\left(\Psi,\gamma\right)$-convergence
(Definition~\ref{def:Psi-gamma-CV}). The use of Strong Poincaré
Inequalities (SPIs) for the analysis of Markov chains is well-developed
and understood, allowing for the characterization of $\ELL-$geometric
convergence, that is $\big(\left\Vert \cdot\right\Vert _{2}^{2},\left\{ \rho^{n}:\rho\in\left[0,1\right),n\in\mathbb{N}\right\} \big)$-convergence,
or providing bounds on the asymptotic variance of ergodic averages. 

However it has been realised only recently that such tools can be
particularly powerful to establish results beyond the reach of, for
example, drift/minorization approaches \cite{meyn1993markov,douc2018markov}.
For instance, (a) in \cite{andrieu2022explicit}, leveraging isoperimetric
inequalities and building on \cite{dwivedi2019logconcave,chen2020fast},
quantitative SPIs have been established, providing the first rigorous
proof that the spectral gap of the random walk Metropolis on $\mathbb{R}^{d}$
vanishes as $d^{-1}$ for a broad class of probability distributions,
(b) in \cite{andrieu2022comparison_journal}, in the scenario where
$\ELL$-geometric convergence does not hold, comparison tools based
on Weak Poincaré Inequalities (WPIs) have been proposed and developed,
allowing for a fine and natural characterization of pseudo-marginal
algorithms in relation to their marginal counterparts.

The use and theory of WPIs in the Markov chain context are largely
underdeveloped in comparison to the Markov process scenario. This
paper aims to fill this gap by developing practical tools to establish
such functional inequalities, while also exploring the equivalence
between $\left(\Psi,\gamma\right)$-convergence and the existence
of WPIs. More specifically, 

In Section~\ref{sec:Weak-Poincar=0000E9-inequalities} we provide
a precise overview of the basic properties of WPIs for Markov chains
and their use in deducing $\left(\Psi,\gamma\right)$-convergence,
complete with various novel notions of comparison and optimality leading,
for example, to new techniques to lower bound convergence rates. In
this same section, we also discuss equivalence with convergence in
other standard metrics and provide a counterexample.

Section~\ref{sec:Existence-of-WPIs} is concerned with the existence
of WPIs. Our first result is a Weak Cheeger Inequality (WCI) for Markov
chains which provides a theoretically and practically useful tool
to establish WPIs, linking the probabilistic and functional analytic
worlds. We further investigate the necessity of the existence of a
non-trivial WPI for a Markov chain to be $\left(\Psi,\gamma\right)$-convergent.
We show equivalence of WPIs with the more probabilistic Resolvent
Uniform Positivity-Improving (RUPI) condition of \cite{gong2006spectral}.
In turn we show that RUPI is itself implied by $\mu$-irreducibility
and a positive holding probability condition. In addition, we provide
a counterexample of an $\big(\left\Vert \cdot\right\Vert _{\mathrm{osc}}^{2},\left\{ \rho^{n}:\rho\in\left[0,1\right),n\in\mathbb{N}\right\} \big)$-convergent
Markov chain which fails to satisfy a WPI. We finish Section~\ref{sec:Existence-of-WPIs}
by showing that WCIs can be obtained from isoperimetric properties
of the target distribution $\mu$ and a close coupling condition of
the Markov kernel $P$, opening the way to quantitative WPIs and estimates
of $\ELL$ rates of convergence.

In Section~\ref{sec:Application-to-Pseudo-marginal} we study pseudo-marginal
MCMC methods, and give lower bounds on their convergence rates, complementing
the previous upper bounds derived in \cite{andrieu2022comparison_journal}.
In Section~\ref{sec:Applications-to-RWM} we derive specific WPIs
for the Random Walk Metropolis (RWM) MCMC algorithm, as applied to
simulation from heavy-tailed probability measures on Euclidean spaces
and deduce complexity bounds. As a consequence, we obtain both positive
and negative results on their convergence to equilibrium, providing
an understanding of the cost incurred when sampling from such distributions
with this type of algorithm. 

In Section~\ref{sec:Discussion} we give some concluding discussion
regarding related literature. Some of the results in this manuscript
appeared in a preliminary form in our technical report \cite{andrieu2022poincare_tech}.

\paragraph*{Notation}

A list of notation can be found in Appendix~\ref{subsec:Notation}.

\section{Weak Poincaré inequalities\label{sec:Weak-Poincar=0000E9-inequalities}}

\subsection{Basic definitions}

We first review the basic definitions needed to define weak Poincaré
inequalities; see also \cite{andrieu2022poincare_tech}.

Throughout this section, for $\mu$ a probability measure defined
on $\left(\E,\mathscr{E}\right)$, we fix a $\mu$-invariant Markov
kernel $P\colon\E\times\mathscr{E}\rightarrow\left[0,1\right]$.
\begin{defn}[Sieve]
\label{def:Phi_fn} We call a functional $\Psi:\ELL\bigl(\mu\bigr)\to\left[0,\infty\right]$
a \textit{sieve functional}, or \textit{sieve}, if 
\begin{enumerate}
\item for any $f\in\ELL\left(\mu\right)$, $c>0$, it holds that $\Psi\left(c\cdot f\right)=c^{2}\cdot\Psi\left(f\right)$, 
\item $\mathfrak{a}:=\sup_{f\in\ELL_{0}\left(\mu\right)\backslash\left\{ 0\right\} }\left\Vert f\right\Vert _{2}^{2}/\Psi\left(f\right)<\infty$.
\end{enumerate}
We will furthermore describe $\Psi$ as $P$-non-expansive if $\Psi\left(Pf\right)\leqslant\Psi\left(f\right)$
for $f\in\ELL_{0}\left(\mu\right)$.
\end{defn}

Note that the second property implies that if $\Psi\left(f\right)=0$
for some $f\in\ELL_{0}\left(\mu\right)$ then $f=0$, which, as we
shall see, is a requirement for a weak Poincaré inequality to hold.
\begin{example}
\label{exa:osc2}Typically, we will work with the sieve $\Psi:\ELL\left(\mu\right)\to\left[0,\infty\right]$
given by $\Psi:=\left\Vert \cdot\right\Vert _{{\rm osc}}^{2}$, which
satisfies the requirements of Definition~\ref{def:Phi_fn} with $\mathfrak{a}\leqslant\frac{1}{4}$.
\end{example}

In Section~\ref{subsec:bounded_to_p} we will consider other choices
of functional $\Psi$.

\begin{defn}[Weak Poincaré Inequalities]
\label{def:WPI}For a sieve $\Psi$, we say that a $\mu$-invariant
kernel $T$ satisfies 
\begin{enumerate}
\item a $\left(\Psi,\alpha\right)$\textit{--weak Poincaré inequality},
abbreviated $\left(\Psi,\alpha\right)$--WPI, if for a decreasing
function $\alpha:\left(0,\infty\right)\to\left[0,\infty\right)$,
\begin{equation}
\left\Vert f\right\Vert _{2}^{2}\leqslant\alpha(r)\cdot\calE\left(T,f\right)+r\cdot\Psi\left(f\right),\quad\forall r>0,f\in\ELL_{0}\left(\mu\right),\label{eq:WPI}
\end{equation}
\item a $\left(\Psi,\beta\right)$\textit{\emph{--WPI}} if for $\beta:\left(0,\infty\right)\to\left[0,\infty\right)$
a decreasing function with $\beta\left(s\right)\to0$ as $s\to\infty$,
\begin{equation}
\left\Vert f\right\Vert _{2}^{2}\leqslant s\cdot\calE\left(T,f\right)+\beta\left(s\right)\cdot\Psi\left(f\right),\quad\forall s>0,f\in\ELL_{0}\left(\mu\right),\label{eq:beta-WPI}
\end{equation}
\item a \textit{$\left(\Psi,K^{*}\right)$}\textit{\emph{--WPI}} if for
a nonnegative, convex, strictly increasing function $K^{*}:\left[0,\mathfrak{a}\right)\to\left[0,\infty\right)$
satisfying $K^{*}\left(0\right)=0$,
\[
\frac{\calE\left(T,f\right)}{\Psi\left(f\right)}\geqslant K^{*}\left(\frac{\left\Vert f\right\Vert _{2}^{2}}{\Psi\left(f\right)}\right),\quad f\in\ELL_{0}\left(\mu\right),0<\Psi\left(f\right)<\infty\,.
\]
\end{enumerate}
\end{defn}

The condition $\mathfrak{a}:=\sup_{f\in\ELL_{0}\left(\mu\right)\backslash\left\{ 0\right\} }\left\Vert f\right\Vert _{2}^{2}/\Psi\left(f\right)<\infty$
implies that 
\begin{enumerate}
\item if a $\left(\Psi,\alpha\right)$\textit{\emph{--}}WPI holds for some
$\alpha$, then it necessarily holds that $\alpha\left(r\right)>0$
for $r\in\left(0,\mathfrak{a}\right)$,
\item if a $\left(\Psi,\alpha\right)$\textit{\emph{--}}WPI holds for some
$\alpha$, then there also holds a $\left(\Psi,\alpha^{'}\right)$\textit{\emph{--}}WPI
with $\alpha^{'}\left(r\right):=\alpha\left(r\right)\cdot\mathbf{1}_{\left[0,\mathfrak{a}\right)}\left(r\right)$,
and
\item if a $\left(\Psi,\beta\right)$\textit{\emph{--}}WPI holds for some
$\beta$, then there also holds a $\left(\Psi,\beta^{'}\right)$\textit{\emph{--}}WPI
with $\beta^{'}\left(s\right):=\min\left\{ \beta\left(s\right),\mathfrak{a}\right\} $.
\end{enumerate}
The subsequent Proposition~\ref{prop:a-b-WPI-corres} establishes
that any one WPI holding implies that all of them hold, and hence
we can refer to a $\Psi$--WPI when the parametrization is unimportant.
The proposition also provides a recipe to move between parametrizations.
As we shall see, the $K^{*}$ formulation leads to quantitative characterizations
of convergence to equilibrium, but the $\alpha$ or $\beta$ parametrizations
are often more convenient to establish in practice.

For any decreasing function $f\colon\mathbb{R}_{+}\rightarrow\mathbb{R}$
we let $f^{\shortminus}:\R\to\left[0,\infty\right]$ be its generalised
inverse, given by $f^{\shortminus}\left(x\right):=\inf\left\{ y>0\colon f\left(y\right)\leqslant x\right\} $
for $x\in\R$. 
\begin{prop}
\label{prop:a-b-WPI-corres}Let $T$ be a Markov kernel on $\big(\mathsf{E},\mathscr{E}\big)$,
$\Psi$ be a sieve, and $\mathfrak{a}:=\sup_{f\in\ELL_{0}\left(\mu\right)\backslash\left\{ 0\right\} }\left\Vert f\right\Vert _{2}^{2}/\Psi\left(f\right)<\infty$.
\begin{enumerate}
\item \label{enu:alphavsbeta-alpha}If a $\left(\Psi,\alpha\right)$\textup{--WPI}
holds with $\alpha\left(r\right)=0$ for $r\geqslant\mathfrak{a}$,
then a $\left(\Psi,\beta\right)$\textup{--WPI, with} $\beta:=\alpha^{\shortminus}$
on $\left(0,\infty\right)$, holds. 
\item \label{enu:alphavsbeta-beta}If a $\left(\Psi,\beta\right)$\textup{--WPI}
holds with $\beta\leqslant\mathfrak{a}$, then a $\left(\Psi,\alpha\right)$\textup{--WPI}
holds, with $\alpha:=\beta^{\shortminus}$ on $\left(0,\infty\right)$.
\item \label{enu:alpha-alpha-minus-identity}Further, if $\alpha$ (resp.
$\beta$) is right continuous then $\left(\alpha^{\shortminus}\right)^{\shortminus}=\alpha$
(resp. $\left(\beta^{\shortminus}\right)^{\shortminus}=\beta$); the
two parametrisations are thus equivalent.
\item \label{enu:beta-implied-Kstar}If a $\left(\Psi,\beta\right)$--WPI
holds then a $\left(\Psi,K^{*}\right)$\textup{--WPI} holds with
\begin{equation}
K^{*}(v):=\sup_{u\geqslant0}\left\{ u\cdot v-K\left(u\right)\right\} ,\qquad v\geqslant0,\label{eq:Kstar-from-beta}
\end{equation}
the convex conjugate of $K\colon\left[0,\infty\right)\rightarrow\left[0,\infty\right)$
given by $K\left(u\right):=u\,\beta\left(1/u\right)$ for $u>0$ and
$K\left(0\right):=0$.
\item \label{enu:K-imply-beta}If a $\left(\Psi,K^{*}\right)$--WPI holds
then a $\left(\Psi,\beta\right)$\textup{--WPI} holds with $\beta\left(s\right)=s\cdot K\left(s^{-1}\right)$
for $s>0$ where
\[
K\left(u\right):=\sup_{0\leqslant v\leqslant\mathfrak{a}}\left\{ u\cdot v-K^{*}\left(v\right)\right\} ,\qquad u\geqslant0,
\]
is the convex conjugate of $K^{*}$.
\end{enumerate}
\end{prop}

\begin{proof}
The proof of \ref{enu:alphavsbeta-alpha}-\ref{enu:alpha-alpha-minus-identity}
can be found in \cite[Proposition 5]{andrieu2022poincare_tech}. The
proof of \ref{enu:beta-implied-Kstar} can be found in \cite[Lemma 1, Supp. material]{andrieu2022comparison_journal}. 

So we focus now on \ref{enu:K-imply-beta}. For $f\in\ELL_{0}\left(\mu\right)$
such that $0<\Psi\left(f\right)<\infty$ and any $u>0$, we have that
\begin{align*}
\frac{\calE\left(T,f\right)}{\Psi\left(f\right)} & \geqslant K^{*}\left(\frac{\left\Vert f\right\Vert _{2}^{2}}{\Psi\left(f\right)}\right)\geqslant u\cdot\frac{\left\Vert f\right\Vert _{2}^{2}}{\Psi\left(f\right)}-K\left(u\right),
\end{align*}
and by rearranging, 
\[
\left\Vert f\right\Vert _{2}^{2}\leqslant u^{-1}\cdot\calE\left(T,f\right)+u^{-1}\cdot K\left(u\right)\cdot\Psi\left(f\right)\,.
\]
For $u>0$, the function
\[
u\mapsto u^{-1}\cdot K\left(u\right)=\sup_{0<v\leqslant\mathfrak{a}}\left\{ v-u^{-1}\cdot K^{*}\left(v\right)\right\} 
\]
is strictly increasing since for $0<u_{0}<u_{1}$,
\[
\sup_{0<v\leqslant\mathfrak{a}}\left\{ v-u_{0}^{-1}\cdot K^{*}\left(v\right)\right\} <\sup_{0<v\leqslant\mathfrak{a}}\left\{ v-u_{1}^{-1}\cdot K^{*}\left(v\right)\right\} \,.
\]
We show that $\lim_{u\downarrow0}u^{-1}\cdot K\left(u\right)=0$,
which implies that $\beta\left(s\right)\to0$ as $s\to\infty$. For
$u>0$,
\begin{align*}
u^{-1}\cdot K\left(u\right) & \leqslant\sup\left\{ v\colon0<v\leqslant\mathfrak{a},v^{-1}\cdot K^{*}\left(v\right)\leqslant u\right\} =:\kappa\left(u\right)\,.
\end{align*}
Since $v\mapsto v^{-1}K^{*}\left(v\right)$ is increasing, it follows
that $u\mapsto\kappa\left(u\right)$ is increasing. Now suppose that
$\ell:=\lim_{v\downarrow0}v^{-1}\cdot K^{*}\left(v\right)>0$. Then
for $0<u<\ell$ we have directly that $u^{-1}\cdot K\left(u\right)=\sup_{0<v\leqslant\mathfrak{a}}\left\{ v-u^{-1}\cdot K^{*}\left(v\right)\right\} =\sup_{0<v\leqslant\mathfrak{a}}v\cdot\left\{ 1-u^{-1}\cdot v^{-1}\cdot K^{*}\left(v\right)\right\} $
must be 0.  Otherwise if $\ell=0$, then for any $\bar{v}\in\left(0,\mathfrak{a}\right]$
and any $0\leqslant u<\bar{v}^{-1}K^{*}\left(\bar{v}\right)$ we have
$\kappa\left(u\right)<\bar{v}$. 

The final statement follows from the definition of the convex conjugate
$K$ of $K^{*}$.
\end{proof}
\begin{defn}
In the situation where a $\left(\Psi,\alpha\right)$--WPI (resp.
$\left(\Psi,\beta\right)$--WPI) holds for a right-continuous function
$\alpha$ (resp. $\beta$), we refer to it as a $\left(\Psi,\alpha,\beta\right)$--WPI
where $\beta=\alpha^{\shortminus}$ (resp. $\alpha=\beta^{\shortminus}$).
\end{defn}

\begin{rem}
\label{rem:nonrev_WPI}As remarked in Appendix~\ref{subsec:Notation},
it is straightforward to see that for all $p\geqslant1$, there holds
the equality $\mathcal{E}_{p}\left(T,f\right)=\mathcal{E}_{p}\left(S,f\right)$,
defined in (\ref{eq:calE_p}), where $S:=\frac{1}{2}\left(T+T^{*}\right)$
is the additive reversibilization of $T$. Hence, WPIs are generally
only properties of the reversible part of $T$, a property used in
some of our later arguments.
\end{rem}

In practice, we are interested in bounding the convergence rate to
equilibrium of a given $\mu$-invariant Markov kernel $P$, as measured
in $\mathrm{L}^{2}\left(\mu\right)$. 
\begin{defn}[$\big(\Psi,\gamma\big)$--convergence]
\label{def:Psi-gamma-CV} Given a $\mu$-invariant Markov kernel,
$P$ and a sieve $\Psi$, we say that $P$ is $\left(\Psi,\gamma\right)$-convergent
or simply $\Psi$-convergent if for all $n\in\mathbb{N}$ and $f\in\ELL_{0}\left(\mu\right)$
such that $0<\Psi\left(f\right)<\infty$,
\begin{equation}
\left\Vert P^{n}f\right\Vert _{2}^{2}\leqslant\gamma\left(n\right)\cdot\Psi\left(f\right),\label{eq:P-Phi-gamma-convergent}
\end{equation}
holds for a sequence $\left\{ \gamma\left(n\right):n\in\mathbb{N}\right\} \subseteq\left(0,\infty\right)$
such that $\lim_{n\rightarrow\infty}\gamma\left(n\right)=0$.
\end{defn}

$\big(\Psi,\gamma\big)$-convergence can be obtained in the framework
of Definition~\ref{def:WPI} by taking $T=P^{*}P$ due to the following
convergence result.
\begin{thm}
\label{thm:WPI_F_bd} Let $P$ be a $\mu$-invariant Markov kernel
on $\left(\mathsf{E},\mathscr{E}\right)$ and assume that $T:=P^{*}P$
satisfies a $\left(\Psi,K^{*}\right)-$WPI, where $\Psi$ is a $P$-non-expansive
sieve. Define $F\colon\left(0,\mathfrak{a}\right]\rightarrow\mathbb{R}$
to be the strictly decreasing, convex and invertible function
\[
F\left(x\right):=\int_{x}^{\mathfrak{a}}\frac{{\rm d}v}{K^{*}\left(v\right)}\,.
\]
Then
\begin{enumerate}
\item \label{enu:thm-WPI-F-bd-complexity} for any $\varepsilon\in\left(0,\mathfrak{a}\right]$,
with
\[
n\left(\varepsilon;\Psi\right):=\inf\left\{ n\in\mathbb{N}\colon\left\Vert P^{n}f\right\Vert _{2}^{2}\leqslant\varepsilon\cdot\Psi\left(f\right)\,\forall f\in\ELL_{0}\left(\mu\right)\right\} \,,
\]
it holds that
\[
n\left(\varepsilon;\Psi\right)\leqslant F\left(\varepsilon\right)\,.
\]
Therefore for $n>F\left(\varepsilon\right)$, $\left\Vert P^{n}f\right\Vert _{2}^{2}\leqslant\varepsilon\cdot\Psi\left(f\right)\,\forall f\in\ELL_{0}\left(\mu\right)$.
\item \label{enu:thm-WPI-F-bd-rate} $P$ is $\big(\Psi,\gamma\big)$-convergent
with $\gamma\left(n\right):=F^{-1}\left(n\right)$.
\end{enumerate}
\end{thm}

\begin{proof}
For \ref{enu:thm-WPI-F-bd-complexity}, from the proof of \cite[Proof of Theorem 8, Supplementary material]{andrieu2022comparison_journal}
we have for any $n\geqslant1$ and $\Psi\left(f\right)>0$,
\begin{equation}
n\leqslant F\left(\frac{\left\Vert P^{n}f\right\Vert _{2}^{2}}{\Psi\left(f\right)}\right)-F\left(\frac{\left\Vert f\right\Vert _{2}^{2}}{\Psi\left(f\right)}\right)\,.\label{eq:WPI-implies-n-leq-diffF}
\end{equation}
Let $\varepsilon>0$. Since $n\mapsto\left\Vert P^{n}f\right\Vert _{2}^{2}/\Psi\left(f\right)$
is decreasing and $F$ is strictly decreasing, it then follows that
$\left\Vert P^{n}f\right\Vert _{2}^{2}/\Psi\left(f\right)>\varepsilon$
implies $n\leqslant F\left(\varepsilon\right)-F\left(\left\Vert f\right\Vert _{2}^{2}/\Psi\left(f\right)\right)$
and as a result for $n>F\left(\varepsilon\right)-F\big(\left\Vert f\right\Vert _{2}^{2}/\Psi\left(f\right)\big)$
we must have that $\left\Vert P^{n}f\right\Vert _{2}^{2}/\Psi\left(f\right)\leqslant\varepsilon$.
From the first statement we conclude that $n\left(\varepsilon;\Psi\right)\leqslant F\left(\varepsilon\right)$
noting that there is no loss in this bound since
\[
\inf_{\Psi\left(f\right)\neq0}F\left(\frac{\left\Vert f\right\Vert _{2}^{2}}{\Psi\left(f\right)}\right)=F\left(\sup_{\Psi\left(f\right)\neq0}\frac{\left\Vert f\right\Vert _{2}^{2}}{\Psi\left(f\right)}\right)=0\,.
\]
 For \ref{enu:thm-WPI-F-bd-rate} simply invert (\ref{eq:WPI-implies-n-leq-diffF}). 
\end{proof}
\begin{example}[{\cite[Lemma~14]{andrieu2022comparison_journal}}]
\label{exa:frombetatogamma-poly}If for some $c_{0},c_{1}>0$ $\beta\left(s\right)=c_{0}s^{-c_{1}}$
then $K^{*}(v)=C\left(c_{0},c_{1}\right)v^{1+c_{1}^{-1}}$ for some
$C\left(c_{0},c_{1}\right)>0$ and
\[
F^{-1}\left(n\right)\leqslant c_{0}\cdot\left(1+c_{1}\right)^{1+c_{1}}\cdot n^{-c_{1}}.
\]
\end{example}

\begin{rem}
In the case where $P$ is reversible and $\Psi=\left\Vert \cdot\right\Vert _{\mathrm{osc}}^{2}$,
the conclusions of Theorem~\ref{thm:WPI_F_bd} allow one to make
statements about the convergence of the probability measures $\nu P^{n}$
to equilibrium for suitably regular initial $\nu$; we return to this
point in Section~\ref{subsec:bounded_to_p} and Section~\ref{sec:Applications-to-RWM}
for our main application.
\end{rem}

\subsection{Optimal $\alpha,\beta$}

It is clear from the definition of a Poincaré inequality that if a
kernel satisfies a $\left(\Psi,\alpha\right)$-weak Poincaré inequality,
then it also satisfies a $\left(\Psi,\alpha'\right)$-weak Poincaré
inequality, for any decreasing $\alpha'$ which satisfies $\alpha'\geqslant\alpha$
pointwise. It is thus natural to introduce a notion of a \textit{minimal}
or optimal weak Poincaré inequality. 

By studying these extremal cases, we will be able to assume without
a loss of generality that our $\left(\alpha,\beta,K^{*}\right)$ satisfy
some natural structural properties, thus facilitating our subsequent
technical developments. Moreover, given the task of comparing the
convergence behaviour of a pair of Markov chains $\left(P_{1},P_{2}\right)$,
matters become substantially more transparent by comparing optimal
inequalities. 
\begin{defn}
\label{def:alpha-beta-star}For a $\mu$-invariant Markov kernel $T$
and sieve $\Psi$ define,
\begin{enumerate}
\item for any $r>0$,
\[
\alpha^{\star}\left(r;\Psi\right):=\sup\left\{ \frac{\left\Vert g\right\Vert _{2}^{2}}{\mathcal{E}\left(T,g\right)}\cdot\left(1-\frac{r}{\left\Vert g\right\Vert _{2}^{2}}\right)\colon g\in\ELL_{0}\left(\mu\right),\Psi\left(g\right)=1\right\} \vee0,
\]
noting that if $r\geqslant\mathfrak{a}$, then $\alpha^{\star}\left(r;\Psi\right)=0$;
\item for any $s>0$, 
\[
\beta^{\star}\left(s;\Psi\right):=\sup\left\{ \left\Vert g\right\Vert _{2}^{2}-s\cdot\mathcal{E}\left(T,g\right)\colon g\in\ELL_{0}\left(\mu\right),\Psi\left(g\right)=1\right\} \vee0,
\]
noting that for $s>0$ $\beta^{\star}\left(s;\Psi\right)\leqslant\mathfrak{a}$.
\end{enumerate}
When $\Psi=\left\Vert \cdot\right\Vert _{\mathrm{osc}}^{2}$ we shall
plainly write $\alpha^{\star}\left(\cdot\right):=\alpha^{\star}\left(\cdot;\Psi\right)$
and $\beta^{\star}\left(\cdot\right):=\beta^{\star}\left(\cdot;\Psi\right)$.
\end{defn}

An important point is that $\alpha^{\star}$ and $\beta^{\star}$
do not necessarily define valid WPIs. Indeed as shown in Proposition~\ref{prop:preparatory-counterexample}
the Markov chain of Example~\ref{exa:irred-pstarkpk} is such that
$P^{*}P$ does not satisfy a WPI, although it can be shown to be $\|\cdot\|_{{\rm osc}}$-convergent;
in fact inspection of the proof of Proposition~\ref{prop:preparatory-counterexample}
allows one to deduce that in this specific example $\inf_{s>0}\beta^{\star}\left(s\right)>0$.
\begin{thm}
\label{thm:alpha-beta-star-WPI-continuous-convex} Suppose that the
$\mu$-invariant kernel $T$ possesses some $\left(\Psi,\alpha\right)$--
or $\left(\Psi,\beta\right)$--WPI for some functions $\alpha$ or
$\beta$. Then
\begin{enumerate}
\item $\alpha^{\star}\left(\cdot;\Psi\right)\leqslant\alpha\left(\cdot\right)$
defines a $\left(\Psi,\alpha^{\star}\right)\shortminus$WPI and $\beta^{\star}\left(\cdot;\Psi\right)\leqslant\beta\left(\cdot\right)$
defines a $\left(\Psi,\beta^{\star}\right)$--WPI; 
\item the functions $\alpha^{\star}\left(\cdot;\Psi\right):\left(0,\mathfrak{a}\right]\to\left[0,\infty\right)$
and $\beta^{\star}\left(\cdot;\Psi\right):\left[0,\infty\right)\to\left[0,\mathfrak{a}\right]$
are convex and continuous;
\item $\beta^{\star}$ is strictly decreasing to $0$ and $\alpha^{\star}=\left(\beta^{\star}\right)^{-1}$
is the inverse function, which is well-defined on $\left(0,\mathfrak{a}\right]$
and strictly decreasing.
\end{enumerate}
\end{thm}

\begin{proof}
We consider the $\beta$ formulation, and drop explicit reference
to the fixed $\Psi$ under consideration; the $\alpha$ formulation
is analogous. By assumption, we know that $T$ possesses a $\left(\Psi,\beta\right)$\textit{\emph{--}}WPI,
for some function $\beta$ as in Definition~\ref{def:WPI} (\textit{c.f.}
\cite[Proposition~5]{andrieu2022poincare_tech}). By definition of
$\beta^{\star}$, we have that $0\leqslant\beta^{\star}\leqslant\beta$
pointwise and so $\beta^{\star}\left(s\right)\to0$ as $s\to\infty$.
Since the pointwise supremum of affine functions (of $s$) is convex,
we obtain convexity and continuity of $\beta^{\star}$, from the fact
that it is the composition of a nondecreasing convex continuous function,
$s\mapsto\max\left\{ 0,s\right\} $, with a convex function. We observe
that $\beta^{\star}\left(0\right)=\mathfrak{a}$. Now, let $s_{0}:=\inf\left\{ s>0:\beta^{\star}\left(s\right)=0\right\} $,
which may be infinite. Since $\beta^{\star}$ is convex and continuous,
it is strictly decreasing on $\left(0,s_{0}\right)$. It follows that
$\beta^{\star}$ is invertible on $\left(0,s_{0}\right)$ with inverse
$\left(\beta^{\star}\right)^{-1}:\left(0,\mathfrak{a}\right]\to\left[0,\infty\right)$
that is also convex and strictly decreasing.

Now we show that $\alpha^{\star}=\left(\beta^{\star}\right)^{-1}$.
For $r\in\left(0,\mathfrak{a}\right]$, let $s:=\left(\beta^{\star}\right)^{-1}\left(r\right)$.
For any $f\in\ELL_{0}\left(\mu\right)$ with $\Psi\left(f\right)=1$
we have
\[
\left\Vert f\right\Vert _{2}^{2}-s\cdot\mathcal{E}\left(T,f\right)\leqslant\beta^{\star}\left(s\right)=r,
\]
and this implies 
\[
\alpha^{\star}\left(r\right)=\sup_{f:\Psi\left(f\right)=1}\left\{ \frac{\left\Vert f\right\Vert _{2}^{2}}{\mathcal{E}\left(T,f\right)}-\frac{r}{\mathcal{E}\left(T,f\right)}\right\} \leqslant s.
\]
Assume for the sake of contradiction that $\alpha^{\star}\left(r\right)=t<s$.
For any $f\in\ELL_{0}\left(\mu\right)$ with $\Psi\left(f\right)=1$
we have
\[
\frac{\left\Vert f\right\Vert _{2}^{2}}{\mathcal{E}\left(T,f\right)}-\frac{r}{\mathcal{E}\left(T,f\right)}\leqslant t,
\]
and so
\[
\beta^{\star}\left(t\right)=\sup_{f:\Psi\left(f\right)=1}\left\{ \left\Vert f\right\Vert _{2}^{2}-t\cdot\mathcal{E}\left(T,f\right)\right\} \leqslant r=\beta^{\star}\left(s\right),
\]
which is a contradiction since $\beta^{\star}$ is decreasing, and
we conclude.
\end{proof}

\subsection{Ordering of $\alpha$'s, $\beta$'s and $\gamma$'s and Peskun--Tierney
ordering \label{subsec:optimal-choice-Ordering-of-rates}}
\begin{thm}
\label{thm:order-alpha-beta-gamma}Let $P_{1}$ and $P_{2}$ be $\mu$-invariant
Markov kernels such that for a sieve $\Psi$, $P_{1}^{*}P_{1}$ satisfies
a $\left(\Psi,\alpha_{1},\beta_{1}\right)$--WPI and $P_{2}^{*}P_{2}$
a $\left(\Psi,\alpha_{2},\beta_{2}\right)$--WPI respectively. Let
$\gamma_{i}=F_{i}^{-1}$ be the respective convergence bound functions
as defined in Theorem~\ref{thm:WPI_F_bd}. Then we have
\begin{enumerate}
\item $\alpha_{2}\left(\cdot;\Psi\right)\geqslant\alpha_{1}\left(\cdot;\Psi\right)$
if and only if $\beta_{2}\left(\cdot;\Psi\right)\geqslant\beta_{1}\left(\cdot;\Psi\right)$;
\item $\beta_{2}\left(\cdot;\Psi\right)\geqslant\beta_{1}\left(\cdot;\Psi\right)$
implies $\gamma_{2}\left(\cdot;\Psi\right)\geqslant\gamma_{1}\left(\cdot;\Psi\right)$.
\end{enumerate}
\end{thm}

\begin{proof}
First statement: we drop $\Psi$ for notational simplicity. For the
direction $\left(\implies\right):$ for any $s>0$ we have $\left\{ r>0:\alpha_{2}\left(r\right)\leqslant s\right\} \subset\left\{ r>0:\alpha_{1}\left(r\right)\leqslant s\right\} $
and hence $\beta_{2}=\alpha_{2}^{\shortminus}\geqslant\alpha_{1}^{\shortminus}=\beta_{1}$;
$\left(\Longleftarrow\right)$ follows along the same lines. For the
second statement: from their definitions, $K_{1}\leqslant K_{2}$
and hence $K_{1}^{*}\geqslant K_{2}^{*}$. As a result, $F_{1}\leqslant F_{2}$
and consequently $\gamma_{1}:=F_{1}^{-1}\leqslant F_{2}^{-1}=:\gamma_{2}$.
\end{proof}
We know from \cite{tierney1998note} that for $P_{1},P_{2}$ $\mu$-reversible,
then $\mathcal{E}\left(P_{1},g\right)\geqslant\mathcal{E}\left(P_{2},g\right)$
for any $g\in\ELL\left(\mu\right)$ implies ${\rm {\rm var}}\left(P_{1},f\right)\leqslant{\rm {\rm var}}\left(P_{2},f\right)$
for $f\in\ELL\left(\mu\right)$ and ${\rm Gap}_{\mathrm{R}}\left(P_{1}\right)\geqslant{\rm Gap}_{\mathrm{R}}\left(P_{2}\right)$.
The latter is useful when ${\rm Gap}\left(P_{1}\right)={\rm Gap}_{\mathrm{R}}\left(P_{1}\right)$
and ${\rm Gap}\left(P_{2}\right)={\rm Gap}_{\mathrm{R}}\left(P_{2}\right)$
since this implies $P_{1}$ has a faster rate of convergence. The
following generalizes the latter statement to the subgeometric setup.
\begin{thm}
Let $P_{1},P_{2}$ be $\mu$-invariant Markov kernels such that for
a sieve $\Psi$, 
\begin{enumerate}
\item $P_{1}^{*}P_{1}$ (resp. $P_{2}^{*}P_{2}$) satisfies a $\left(\Psi,\alpha_{1},\beta_{1}\right)$--WPI
(resp. a $\left(\Psi,\alpha_{2},\beta_{2}\right)$--WPI), 
\item $\mathcal{E}\left(P_{1}^{*}P_{1},g\right)\geqslant\mathcal{E}\left(P_{2}^{*}P_{2},g\right)$
for any $g\in\mathrm{L}^{2}\left(\mu\right)$ such that $\Psi\left(g\right)<\infty$.
\end{enumerate}
Then with $\alpha_{i}^{\star}\left(\cdot;\Psi\right)$ and $\beta_{i}^{\star}\left(\cdot;\Psi\right)$
for $i=1,2$ defined as in Definition~\ref{def:alpha-beta-star},
a $\left(\Psi,\alpha_{i}^{\star},\beta_{i}^{\star}\right)$--WPI
holds for $i=1,2$, there are the orderings $\alpha_{1}^{\star}\left(\cdot;\Psi\right)\leqslant\alpha_{2}^{\star}\left(\cdot;\Psi\right)$,
$\beta_{1}^{\star}\left(\cdot;\Psi\right)\leqslant\beta_{2}^{\star}\left(\cdot;\Psi\right)$,
and the corresponding respective convergence bounds in Theorem~\ref{thm:WPI_F_bd}
satisfy $\gamma_{1}^{\star}\leqslant\gamma_{2}^{\star}$.
\end{thm}

\begin{proof}
From the ordering of Dirichlet forms, we have for any $g\in\ELL\left(\mu\right)$
that
\[
\left\Vert g\right\Vert _{2}^{2}-s\cdot\mathcal{E}\left(P_{1}^{*}P_{1},g\right)\leqslant\left\Vert g\right\Vert _{2}^{2}-s\cdot\mathcal{E}\left(P_{2}^{*}P_{2},g\right),
\]
from Definition~\ref{def:alpha-beta-star} we deduce $\beta_{1}^{\star}\left(\cdot;\Psi\right)\leqslant\beta_{2}^{\star}\left(\cdot;\Psi\right)$;
a similar argument gives that $\alpha_{1}^{\star}\left(\cdot;\Psi\right)\leqslant\alpha_{2}^{\star}\left(\cdot;\Psi\right)$.
From Theorem~\ref{thm:order-alpha-beta-gamma}, we then conclude
that $\gamma_{1}^{\star}\leqslant\gamma_{2}^{\star}$.
\end{proof}

\subsection{Lower bounds on convergence rates \label{subsec:optimal-choice-Lower-bounds-on}}

This section is concerned with tools to establish lower bounds on
$\beta$ involved in a $\left(\Psi,\beta\right)$\textit{\emph{--}}WPI
from which one can deduce lower bounds on the convergence rate $\gamma$
in Theorem~\ref{thm:WPI_F_bd}. This is illustrated with examples.

A practical difficulty with the result of Theorem~\ref{thm:WPI_F_bd}
and the results presented so far is that they require manipulating
$\mathcal{E}\left(P^{*}P,f\right)$ which is most often not tractable,
in contrast with $\mathcal{E}\left(P,f\right)$. This is addressed
extensively in \cite[Section 2.2.1 and Theorem 42]{andrieu2022comparison_journal}.
We first show that a lower bound on $\beta_{1}^{\star}$ in a $\left(\Psi,\beta_{1}^{\star}\right)$\textit{\emph{--}}WPI
for $P$ can imply a lower bound on $\beta_{2}^{\star}$ in a $\left(\Psi,\beta_{2}^{\star}\right)$\textit{\emph{--}}WPI
for $P^{*}P$. We first recall the following result. 
\begin{lem}[{\cite[Remark~3.1]{diaconis1996nash}}]
\label{lem:PP-dirichlet-form-ub}Let $P$ be $\mu$-invariant. Then
\[
\mathcal{E}\left(P^{*}P,f\right)\leqslant2\cdot\mathcal{E}\left(P,f\right),\qquad f\in\ELL_{0}\left(\mu\right).
\]
\end{lem}

\begin{rem}
If $P$ is $\mu$-reversible, then one can obtain $\mathcal{E}\left(P^{2},f\right)\leqslant\left(1+\lambda_{\star}\right)\cdot\mathcal{E}\left(P,f\right)$
by using the spectral theorem, where $\lambda_{\star}=\sup\sigma_{0}\left(P\right)$.
However, since the focus here is on WPIs, the case $\lambda_{\star}<1$
is less relevant. We note also that a converse of sorts may be obtained
when $P$, and therefore $P^{*}$, satisfies $P\big(x,\left\{ x\right\} \big)\geqslant\varepsilon$
on a $\mu$-full set; see Lemma~\ref{lem:P-to-PP-WPI}.
\end{rem}

\begin{lem}
\label{lem:PP-beta-star-from-P-beta-star} Let $P$ be $\mu$-invariant,
and assume it satisfies a $\left(\Psi,\beta_{1}^{\star}\right)$\emph{--}WPI,
where $\beta_{1}^{\star}$ is pointwise minimal. Assume $P^{*}P$
satisfies a $\left(\Psi,\beta_{2}^{\star}\right)$\emph{--}WPI where
$\beta_{2}^{\star}$ is pointwise minimal. Then $\beta_{2}^{\star}\left(s\right)\geqslant\beta_{1}^{\star}\left(2\cdot s\right)$.
\end{lem}

\begin{proof}
Let $\mathcal{F}=\left\{ f\in\ELL_{0}\left(\mu\right):\Psi\left(f\right)=1\right\} $.
By Lemma~\ref{lem:PP-dirichlet-form-ub} we have that $\mathcal{E}\left(P^{*}P,f\right)\leqslant2\cdot\mathcal{E}\left(P,f\right)$.
We may write 
\[
\beta_{1}^{\star}\left(s\right)=0\vee\sup_{f\in\mathcal{F}}\left\{ \left\Vert f\right\Vert _{2}^{2}-s\cdot\mathcal{E}\left(P,f\right)\right\} .
\]
We then have that
\begin{align*}
\beta_{2}^{\star}\left(s\right) & =0\vee\sup_{f\in\mathcal{F}}\left\{ \left\Vert f\right\Vert _{2}^{2}-s\cdot\mathcal{E}\left(P^{*}P,f\right)\right\} \\
 & \geqslant0\vee\sup_{f\in\mathcal{F}}\left\{ \left\Vert f\right\Vert _{2}^{2}-2\cdot s\cdot\mathcal{E}\left(P,f\right)\right\} \\
 & =\beta_{1}^{\star}\left(2\cdot s\right),
\end{align*}
and we conclude.
\end{proof}
In the case where $P$ is $\mu$-reversible, we can then deduce from
a $\left(\Psi,\beta_{1}\right)$--WPI for $P$ a lower bound on a
separable rate of convergence for $\left\Vert P^{n}f\right\Vert _{2}$. 
\begin{prop}
\label{prop:rev-L2-conv-rate-lower-bound} Assume $P$ is $\mu$-reversible,
satisfies a $\left(\Psi,\beta\right)$--WPI and the pointwise minimal
$\beta^{\star}$ satisfies $\beta^{\star}\left(s\right)\in\Omega\left(s^{-p}\right)$
for some $p>0$. Then it cannot hold that with $q>p$, $\left\Vert P^{n}f\right\Vert _{2}^{2}\in\mathcal{O}\left(n^{-q}\right)$
for all $f\in\ELL_{0}\left(\mu\right)$ with $\Psi\left(f\right)<\infty$.
\end{prop}

\begin{proof}
If $\beta^{\star}\left(s\right)\in\Omega\left(s^{-p}\right)$ then
we may deduce by Lemma~\ref{lem:PP-beta-star-from-P-beta-star} that
if $P^{2}$ satisfies (\ref{eq:beta-WPI}), its pointwise minimal
$\beta_{2}^{\star}$ also satisfies $\beta_{2}^{\star}\left(s\right)\in\Omega\left(s^{-p}\right)$.
Now assume for the sake of contradiction that $\left\Vert P^{n}f\right\Vert _{2}^{2}\in\mathcal{O}\left(n^{-q}\right)$
for all $f\in\mathrm{L}_{0}^{2}\left(\mu\right)$ such that $\Psi\left(f\right)<\infty$.
Then by \cite[Proposition 24 and Remark 25]{andrieu2022comparison_journal},
we deduce that a WPI for $P^{2}$ holds with $\beta_{2}\left(s\right)\in\mathcal{O}\left(s^{-q}\right)$,
which contradicts $\beta_{2}^{\star}\left(s\right)\in\Omega\left(s^{-p}\right)$
being pointwise minimal.
\end{proof}
In principle, noting that $\alpha^{\star}$ and $\beta^{\star}$ are
pointwise minimal functions, any function $f\in\ELL_{0}\left(\mu\right)$
with $\Psi\left(f\right)=1$ may be used to construct a lower bound
on them. For example, for any such function, 
\[
s\mapsto\beta^{\star}\bigl(s\bigr)\geqslant\left\Vert f\right\Vert _{2}^{2}-s\cdot\mathcal{E}\left(T,f\right),\qquad s>0.
\]
In practice, to produce an informative lower bound for the whole function
$\beta^{\star}$, one will need to identify an appropriate subset
of $\ELL_{0}\left(\mu\right)$ and consider the supremum of the lower
bound above over such a subset. Indicator functions of measurable
sets are always in $\ELL_{0}\left(\mu\right)$, have finite oscillation,
and can provide tractability. For example in such a scenario, for
$A\in\mathscr{E}$, $\mathcal{E}\left(P,{\bf 1}_{A}\right)$ has a
natural probabilistic interpretation (see Lemma~\ref{lem:dirichlet-form-indicator}). 

The following result establishes a lower bound on $\beta^{\star}$
for Markov kernels $P$ that can exhibit sticky behaviour in regions
of the state space. \cite[Theorem~5.1]{roberts1996geometric} showed
that for a $\mu$-invariant Markov kernel $P$ with $\mu$ not concentrated
at a single point, then ${\rm ess}_{\mu}\sup_{x}P\left(x,\left\{ x\right\} \right)=1$
implies that $P$ cannot converge geometrically. In \cite[Theorem 1]{lee2014variance}
conductance is used to prove the same when $P$ is $\mu$-reversible,
and the following provides a quantitative refinement. A similar result
can be obtained using a Weak Cheeger inequality established in Section~\ref{subsec:Cheeger-inequalities},
Theorem~\ref{thm:positive-CP-implies-optim-WPI}-b).
\begin{thm}
\label{thm:lower-bound-beta-star}Let $P$ be $\mu$-invariant satisfying
a $\left(\Psi,\beta\right)$--WPI, where $\Psi\left(\mathbf{1}_{A}\right)\leqslant1$
for any set $A\in\mathscr{E}$. For any $\varepsilon>0$, define the
set $A_{\varepsilon}:=\left\{ x\in\mathsf{E}\colon P\left(x,\left\{ x\right\} \right)\geqslant1-\varepsilon\right\} $.
Then
\begin{enumerate}
\item for any $s>0$, 
\begin{equation}
\beta\left(s;\Psi\right)\geqslant\sup_{\varepsilon\in\left(0,1\right)}\left\{ \mu\left(A_{\varepsilon}\right)\cdot\left(1-s\cdot\varepsilon-\mu\left(A_{\varepsilon}\right)\right)\right\} ,\label{eq:lower-bound-beta}
\end{equation}
this is in particular true for $\beta^{\star}$ given in Definition~\ref{def:alpha-beta-star}
and $P$ satisfies a $\left(\Psi,\beta^{\star}\right)$\textup{-WPI},
\item if in addition there exist $A\in\mathscr{E}$ such that $\mu\left(A\right)>0$
and $P\left(x,\left\{ x\right\} \right)<1$ for $x\in A$ then there
exists $\underline{s}>0$ and $C>0$ such that for any $s\geqslant\underline{s},$
\[
\beta\left(s;\Psi\right)\geqslant C\cdot\mu\left(A_{\varepsilon\left(s\right)}\right)\,,
\]
with $\varepsilon\left(s\right):=C/s$.
\end{enumerate}
\end{thm}

\begin{proof}
For any $A\in\mathscr{E}$, from Lemma~\ref{lem:dirichlet-form-indicator},
we have $\mathcal{E}\left(P,\mathbf{1}_{A}\right)=\mu\otimes P\left(A\times A^{\complement}\right)$
and ${\rm var}\big(\mathbf{1}_{A}\big)=\mu\otimes\mu\left(A\times A^{\complement}\right)$.
Since $\Psi\left(\mathbf{1}_{A_{\varepsilon}}\right)\leqslant1$,
for any $\varepsilon>0$ we have that
\begin{align*}
\beta^{\star}\left(s\right) & :=\sup\left\{ \left\Vert f\right\Vert _{2}^{2}-s\cdot\mathcal{E}\left(P,f\right)\colon f\in\ELL_{0}\left(\mu\right),\Psi\left(f\right)\leqslant1\right\} \\
 & \geqslant{\rm var}_{\mu}\left(\mathbf{1}_{A_{\varepsilon}}\right)-s\cdot\int\mu\left(\mathrm{d}x\right)\cdot P\left(x,\mathrm{d}y\right)\cdot\mathbf{1}_{A_{\varepsilon}}\left(x\right)\cdot\mathbf{1}_{A_{\varepsilon}^{\complement}}\left(y\right)\\
 & \geqslant{\rm var}_{\mu}\left(\mathbf{1}_{A_{\varepsilon}}\right)-s\cdot\int\mu\left(\mathrm{d}x\right)\cdot P\big(x,\left\{ x\right\} ^{\complement}\big)\cdot\mathbf{1}_{A_{\varepsilon}}\left(x\right)\\
 & \geqslant\mu\left(A_{\varepsilon}\right)\cdot\mu\big(A_{\varepsilon}^{\complement}\big)-s\cdot\mu\left(A_{\varepsilon}\right)\cdot\varepsilon\\
 & =\mu\left(A_{\varepsilon}\right)\cdot\left(1-s\cdot\varepsilon-\mu\left(A_{\varepsilon}\right)\right).
\end{align*}
Together with Theorem~\ref{thm:alpha-beta-star-WPI-continuous-convex}
the first statement follows. 

For the second statement, note first that there exists $\bar{\varepsilon}\in\left(0,1\right)$
such that $\bar{\mu}\left(\bar{\varepsilon}\right):=\sup\left\{ \mu\big(A_{\varepsilon}\big):\varepsilon\in\left[0,\bar{\varepsilon}\right]\right\} <1$.
Indeed, assume $\bar{\mu}\left(\bar{\varepsilon}\right)=1$ for all
$\bar{\varepsilon}\in\left(0,1\right)$. We can thus find a sequence
$\bar{\varepsilon}_{n}\to0$ such that $\bar{\mu}\left(\bar{\varepsilon}_{n}\right)=1$
for all $n$, and hence we can find $\varepsilon_{n}\in\left[0,\bar{\varepsilon}_{n}\right]$
such that $\mu\big(A_{\varepsilon_{n}}\big)\geqslant1-1/n$ for each
$n$. By passing to a subsequence if necessary, we can assume that
$\varepsilon_{n}$ is decreasing in $n$, so that the $A_{\varepsilon_{n}}$
form a decreasing sequence of sets (with respect to inclusion). We
can thus define $A_{\infty}:=\cap_{n\in\mathbb{N}}A_{\varepsilon_{n}}$
, and by the decreasing property, it follows that 
\[
\mu\left(A_{\infty}\right)=\lim_{n\rightarrow\infty}\mu\left(A_{\varepsilon_{n}}\right)=1.
\]
However, we have
\[
A_{\infty}=\left\{ x\in\mathsf{E}:P\left(x,\left\{ x\right\} \right)\geqslant1-\varepsilon_{n},\forall n\in\mathbb{N}\right\} =\left\{ x\in\mathsf{E}:P\left(x,\left\{ x\right\} \right)=1\right\} ,
\]
from which it follows that $P\left(x,\left\{ x\right\} \right)=1$
$\mu$-a.e., leading to a contradiction. 

From the first statement, and with $\left(\bar{\varepsilon},\bar{\mu}\right)\in\left(0,1\right)$
as above, it then holds for $s\in\left(\frac{1-\bar{\mu}}{2\cdot\bar{\varepsilon}},\infty\right)$
that
\begin{align*}
\beta\left(s;\Psi\right) & \geqslant\sup_{\varepsilon\in\left(0,\varepsilon_{0}\right)}\left\{ \mu\big(A_{\varepsilon}\big)\cdot\left(1-s\cdot\varepsilon-\mu\big(A_{\varepsilon}\big)\right)\right\} \\
 & \geqslant\sup_{\varepsilon\in\left(0,\varepsilon_{0}\right)}\left\{ \mu\big(A_{\varepsilon}\big)\cdot\left(1-\bar{\mu}-s\cdot\varepsilon\right)\right\} \\
 & \geqslant\frac{1-\bar{\mu}}{2}\cdot\mu\big(A_{\varepsilon\left(s\right)}\big)\,,
\end{align*}
where $\varepsilon\left(s\right):=\frac{1-\bar{\mu}}{2\cdot s}$.
\end{proof}
\begin{example}
\label{exa:beta-rej-lower-bound-rate}Let $P$ be $\mu$-invariant
and satisfy a $\left(\Psi,\beta\right)$--WPI. If $\mu\big(A_{\varepsilon}\big)\geqslant c\cdot\varepsilon^{a}$
for $\varepsilon\in\left[0,\bar{\varepsilon}\right]$ for some $\bar{\varepsilon}>0$,
then there exists $C>0$ such that for large enough $s>0$
\[
\beta\left(s;\Psi\right)\geqslant C\cdot s^{-a}\,.
\]
In particular we deduce that $\beta^{\star}\left(s;\Psi\right)\in\Omega\left(s^{-a}\right)$,
and by Lemma~\ref{lem:PP-beta-star-from-P-beta-star}, we conclude
that if $P^{*}P$ satisfies a $\left(\Psi,\beta_{2}\right)$--WPI,
then $\beta_{2}^{\star}\in\Omega\left(s^{-a}\right)$ also. If $P$
is additionally reversible, then application of Proposition~\ref{prop:rev-L2-conv-rate-lower-bound}
shows that $P$ cannot be $\left(\Psi,\gamma\right)$-convergent with
$\gamma\left(n\right)\sim n^{-b}$ for any $b>a$.

\end{example}

\begin{example}
The $\mu$-reversible independent Metropolis--Hastings (IMH) Markov
kernel $P$ with proposal probability measure $\nu$ is
\[
P\left(x,A\right)=\int\nu\left({\rm d}y\right)\cdot\left\{ 1\wedge\frac{w\left(y\right)}{w\left(x\right)}\right\} {\bf 1}_{A}\left(y\right)+\int\nu\left({\rm d}y\right)\cdot\left\{ 1-1\wedge\frac{w\left(y\right)}{w\left(x\right)}\right\} ,\quad A\in\mathscr{E},
\]
where $w={\rm d}\mu/{\rm d}\nu$. One can deduce that 
\begin{align*}
P\left(x,\left\{ x\right\} \right) & \geqslant\int\nu\left({\rm d}y\right)\cdot\left\{ 1-1\wedge\frac{w\left(y\right)}{w\left(x\right)}\right\} \\
 & =1-\int\mu\left({\rm d}y\right)\cdot\left\{ w\left(x\right)^{-1}\wedge w\left(y\right)^{-1}\right\} \\
 & \geqslant1-w\left(x\right)^{-1}.
\end{align*}
and as a result for $\varepsilon>0$, one sees that 
\[
B_{\varepsilon}:=\left\{ x\in\mathsf{E}:w\left(x\right)^{-1}\leqslant\varepsilon\right\} \subset A_{\varepsilon}=\left\{ x\in\mathsf{E}\colon P\left(x,\left\{ x\right\} \right)\geqslant1-\varepsilon\right\} .
\]
From Theorem~\ref{thm:lower-bound-beta-star} we deduce the existence
of $C>0$ such that for any $s\geqslant\underline{s}$,
\begin{align*}
\beta\left(s;\Psi\right) & \geqslant C\cdot\mu\big(A_{\varepsilon\left(s\right)}\big)\geqslant C\cdot\mu\big(w\geqslant C^{-1}\cdot s\big).
\end{align*}
Upon noting that the IMH operator is positive, this then implies a
lower bound on the fastest rate of convergence possible.
\end{example}

\subsection{Other types of convergence from $\big(\left\Vert \cdot\right\Vert _{\mathrm{osc}}^{2},\gamma\big)$
convergence\label{subsec:bounded_to_p}}

In practice it can sometime be difficult to establish that a candidate
sieve $\Psi$ found through calculations is indeed a sieve. In contrast,
the specific choices $\Psi=\left\Vert \cdot\right\Vert _{\infty}^{2}$
or $\Psi=\left\Vert \cdot\right\Vert _{\mathrm{osc}}^{2}$ often simplify
calculations greatly. This appears at first sight to be at the expense
of generality in terms of the class of functions or metrics for which
convergence can be established. We present here a series of results
on the implications of $\left(\left\Vert \cdot\right\Vert _{{\rm osc}}^{2},\gamma\right)$-convergence
in other metrics. 

The following, which generalizes the result of \cite[Lemma~5.1]{cattiaux2012central}
(see the discussion below), shows that $\left(\left\Vert \cdot\right\Vert _{{\rm osc}}^{2},\gamma\right)$-convergence
automatically implies convergence in any Orlicz norm dominating the
$\mathrm{L^{2}}$ norm.
\begin{defn}[Orlicz norm]
 For $\mu$ a probability distribution on $\left(\mathsf{E},\mathscr{E}\right)$
and a convex, lower semicontinuous function $N:\left[0,\infty\right)\to\left[0,\infty\right)$
satisfying $N\left(0\right)=0$, the $N$-Orlicz norm $\left\Vert \cdot\right\Vert _{N}$
(or Luxemburg norm) is defined as
\[
\left\Vert f\right\Vert _{N}:=\inf\left\{ t>0:\mu\left(N\left(\frac{\left|f\right|}{t}\right)\right)\leqslant1\right\} ,\forall f\colon\mathsf{E}\rightarrow\mathbb{R}.
\]
\end{defn}

Note that whenever $N_{1}\geqslant N_{2}$, then $\left\Vert \cdot\right\Vert _{N_{1}}\geqslant\left\Vert \cdot\right\Vert _{N_{2}}$,
and that the choice $N(x)=x^{2}$ leads to the norm $\left\Vert \cdot\right\Vert _{2}$.
The following allows one to establish rates of convergence for Orlicz
norms for which $N\left(x\right)$ grows faster than $x^{2}$ in a
suitable sense. 
\begin{prop}
\label{prop:cattiaux-et-al-orlicz}Let $P$ be a $\mu$-invariant
Markov kernel, assumed to be $\left(\left\Vert \cdot\right\Vert _{{\rm osc}}^{2},\gamma\right)$-convergent,
and let $N:\left[0,\infty\right)\to\left[0,\infty\right)$ be a convex,
lower semicontinuous function satisfying

\begin{equation}
N\left(0\right)=0,\quad x\mapsto x^{-2}\cdot N\left(x\right)\,\text{is increasing}.\label{eq:N_monotone}
\end{equation}
Then $N$ defines an Orlicz norm $\left\Vert \cdot\right\Vert _{N}^{2}$
and $P$ is also $\left(\left\Vert \cdot\right\Vert _{N}^{2},\gamma_{N}\right)$-convergent,
with 
\[
\gamma_{N}\left(n\right)\leqslant2^{4}\cdot\gamma\left(n\right)\cdot N^{-1}\left(\gamma\left(n\right)^{-1}\right)^{2},\quad n\in\mathbb{N}.
\]
\end{prop}

\begin{proof}
That $N$ defines an Orlicz norm is straightforward. Let $f\colon\mathsf{E}\rightarrow\mathbb{R}$
satisfying $\left\Vert f\right\Vert _{N}=1$, which implies $\mu\left(N\left(\left|f\right|\right)\right)\leqslant1$
since $N$ is non-decreasing. Let $H>0$, and decompose $f=f\cdot\mathbf{1}\left[\left|f\right|\leqslant H\right]+f\cdot\mathbf{1}\left[\left|f\right|>H\right]=f_{\downarrow}+f_{\uparrow}$.
Write that
\begin{align*}
\left\Vert P^{n}f-\mu\left(f\right)\right\Vert _{2} & \leqslant\left\Vert P^{n}\left(f-f_{\downarrow}\right)\right\Vert _{2}+\left\Vert P^{n}f_{\downarrow}-\mu\left(f_{\downarrow}\right)\right\Vert _{2}+\left\Vert \mu\left(f_{\downarrow}-f\right)\right\Vert _{2}\\
 & =\left\Vert P^{n}f_{\uparrow}\right\Vert _{2}+\left\Vert P^{n}f_{\downarrow}-\mu\left(f_{\downarrow}\right)\right\Vert _{2}+\left|\mu\left(f_{\uparrow}\right)\right|\\
 & \leqslant\left\Vert f_{\uparrow}\right\Vert _{2}+\left\Vert P^{n}f_{\downarrow}-\mu\left(f_{\downarrow}\right)\right\Vert _{2}+\left\Vert f_{\uparrow}\right\Vert _{2}.
\end{align*}
Since $f_{\downarrow}$ is bounded, $\left(\left\Vert \cdot\right\Vert _{{\rm osc}}^{2},\gamma\right)$-convergence
yields that $\left\Vert P^{n}f_{\downarrow}-\mu\left(f_{\downarrow}\right)\right\Vert _{2}\leqslant2\cdot H\cdot\left(\gamma\left(n\right)\right)^{1/2}$.
For the other term, using (\ref{eq:N_monotone}) and the fact that
$\mu\left(N\left(\left|f\right|\right)\right)\leqslant1$ to write
\begin{align*}
\left\Vert f_{\uparrow}\right\Vert _{2}^{2} & =\mu\left(f^{2}\cdot\mathbf{1}\left[\left|f\right|>H\right]\right)\\
 & =\mu\left(\frac{f^{2}}{N\left(\left|f\right|\right)}\cdot N\left(\left|f\right|\right)\cdot\mathbf{1}\left[\left|f\right|>H\right]\right)\\
 & \leqslant\frac{H^{2}}{N\left(H\right)}\cdot\mu\left(N\left(\left|f\right|\right)\cdot\mathbf{1}\left[\left|f\right|>H\right]\right)\\
 & \leqslant\frac{H^{2}}{N\left(H\right)},
\end{align*}
which yields the bound $\left\Vert P^{n}f-\mu\left(f\right)\right\Vert _{2}\leqslant2\cdot H\cdot\gamma\left(n\right)^{1/2}+2\cdot H\cdot N\left(H\right)^{-1/2}$.
Taking $H=N^{-1}\left(\gamma\left(n\right)^{-1}\right)$ gives that
$\left\Vert P^{n}f-\mu\left(f\right)\right\Vert _{2}\leqslant2^{2}\cdot\gamma\left(n\right)^{1/2}\cdot N^{-1}\left(\gamma\left(n\right)^{-1}\right)$,
from which the result follows.
\end{proof}
For examples of the above, consider
\begin{itemize}
\item $N\left(x\right)=x^{p}$, $p>2$, so that $\left\Vert f\right\Vert _{N}=\left\Vert f\right\Vert _{\mathrm{L}^{p}}$
and $\gamma_{N}\left(n\right)\leqslant2^{4}\cdot\gamma\left(n\right)^{1-2/p}$.
\item $N\left(x\right)=\exp\left(x^{r}\right)-1$, $r\geqslant1$, so that
\[
\gamma_{N}\left(n\right)\leqslant2^{4}\cdot\gamma\left(n\right)\cdot\left(\log\left(1+\gamma\left(n\right)^{-1}\right)\right)^{2/r}\lesssim\gamma\left(n\right)\cdot\left(\log\left(\frac{1}{\gamma\left(n\right)}\right)\right)^{2/r}\,\cdot
\]
\end{itemize}
An alternative strategy to handle broader classes of functions consist
of establishing that a $\big(\left\Vert \cdot\right\Vert _{{\rm osc}}^{2},\alpha\big)$\textit{\emph{--}}WPI
or $\big(\left\Vert \cdot\right\Vert _{{\rm osc}}^{2},\beta\big)$\textit{\emph{--}}WPI
implies a $\big(\left\Vert \cdot\right\Vert _{N}^{2},\beta_{N}\big)$\textit{\emph{--}}WPI
for $N$ defining an Orlicz norm, see \cite[Theorem~37]{zitt2008annealing},
\cite[Sections 8-10]{bobkov2009distributions}, and \cite[Proposition~35, Theorems~36,~40]{andrieu2022comparison_journal}.
We do not know whether these approaches are comparable in general,
but have observed that in each case, one recovers similar rates in
the polynomial scenario for $\mathrm{L}^{p}$ norms. We note however
that we have found the approach given in \cite{andrieu2022comparison_journal}
more difficult to use in practice.

Since the bound for $N\left(x\right)=x^{2}$ is not decreasing with
$n$, the above result does not provide an $\ELL$ convergence rate
for all $\ELL$ functions. This should not be surprising as this would
otherwise imply that $P^{k}$ is a strict contraction for some $k\in\mathbb{N}$,
therefore implying geometric convergence. Note however the result
of \cite{cattiaux2010poincare}, which can be adapted to reversible
Markov chains and establishes that $\big(\left\Vert \cdot\right\Vert _{2}^{2},\rho^{n}\big)$-convergence
can be deduced from $\big(\left\Vert \cdot\right\Vert _{{\rm osc}}^{2},\rho^{n}\big)$-convergence
for $\rho\in\left(0,1\right)$. 

The following provides a useful equivalence to $\big(\left\Vert \cdot\right\Vert _{{\rm osc}}^{2},\gamma\big)$-convergence,
particularly relevant in the reversible scenario, and shows that $\big(\left\Vert \cdot\right\Vert ,\gamma\big)$-convergence
implies total variation convergence for a broad class of initial distributions:
the latter improves on \cite[Remark~2, Supp. Material]{andrieu2022comparison_journal}.
\begin{prop}
\label{thm:convergence-duality}Let $P$ be a $\mu$-invariant Markov
kernel. Then for all $n\in\mathbb{N}$, it holds that
\[
\gamma_{\star}^{1/2}\left(n\right):=\sup_{f:\left\Vert f\right\Vert _{{\rm osc}}\leqslant1}\left\Vert P^{n}f-\mu\left(f\right)\right\Vert _{2}=\frac{1}{2}\sup_{f:\left\Vert f\right\Vert _{2}\leqslant1}\left\Vert \left(P^{*}\right)^{n}f-\mu\left(f\right)\right\Vert _{1}.
\]
Further, for any $\nu\ll\mu$
\[
\left\Vert \nu P^{n}-\mu\right\Vert _{\mathrm{TV}}\leqslant\gamma_{\star}^{1/2}\left(n\right)\cdot\chi^{2}\left(\nu,\mu\right)^{1/2}.
\]
Conversely, if for some $\gamma\colon\mathbb{N}\rightarrow\left[0,\infty\right)$
for all probabilities $\nu\ll\mu$ such that $\frac{{\rm d}\nu}{{\rm d}\mu}\in\ELL\left(\mu\right)$
\[
\left\Vert \nu P^{n}-\mu\right\Vert _{\mathrm{TV}}\leqslant\gamma^{1/2}\left(n\right)\cdot\chi^{2}\left(\nu,\mu\right)^{1/2}
\]
then $\gamma_{\star}^{1/2}\left(n\right)\leqslant\sqrt{2}\cdot\gamma^{1/2}\left(n\right)$.
\end{prop}

\begin{rem}
The the first statement of Proposition~\ref{thm:convergence-duality}
can be adapted to replace the $\left\Vert \cdot\right\Vert _{1}$
norm with a general Orlicz norms, but we do not detail this here for
brevity. This type of result mirrors the theory on quantitative convergence
obtained via drift/minorization techniques \cite[Theorem 17.2.5]{douc2018markov}.

We note briefly that similar duality-type results are also available
for general norm-based sieves. In particular, given an estimate of
the form $\left\Vert P^{n}f\right\Vert _{2}^{2}\leqslant\gamma\left(n\right)\cdot\left\Vert f\right\Vert _{\circ}^{2}$
for all suitable $f$, an analogous argument establishes that $\left\Vert \left(P^{*}\right)^{n}g\right\Vert _{\bullet}^{2}\leqslant\gamma\left(n\right)\cdot\left\Vert g\right\Vert _{2}^{2}$,
whenever $\left\{ \left\Vert \cdot\right\Vert _{\circ},\left\Vert \cdot\right\Vert _{\bullet}\right\} $
are a dual pair of norms defined on $\ELL_{0}\left(\mu\right)$, i.e.
$\left|\left\langle f,g\right\rangle \right|\leqslant\left\Vert f\right\Vert _{\circ}\cdot\left\Vert g\right\Vert _{\bullet}$
for all suitable $f,g$. 
\end{rem}

\begin{proof}[Proof of Proposition~\ref{thm:convergence-duality}]
 Since for $f\in\ELL_{0}\left(\mu\right)$,
\[
\left\Vert f\right\Vert _{2}=\sup\left\{ \left\langle f,g\right\rangle :g\in\ELL_{0}\left(\mu\right),\left\Vert g\right\Vert _{2}\leqslant1\right\} ,
\]
we have
\begin{align*}
\sup\big\{\left\Vert P^{n}f\right\Vert _{2} & :\left\Vert f\right\Vert _{\mathrm{osc}}\leqslant1\big\}\\
=\sup & \left\{ \left\langle P^{n}f,g\right\rangle :f\in\ELL_{0}\left(\mu\right),\left\Vert f\right\Vert _{\mathrm{osc}}\leqslant1,g\in\ELL_{0}\left(\mu\right),\left\Vert g\right\Vert _{2}\leqslant1\right\} \\
=\sup & \left\{ \left\langle f,\left(P^{*}\right)^{n}g\right\rangle :f\in\ELL_{0}\left(\mu\right),\left\Vert f\right\Vert _{\mathrm{osc}}\leqslant1,g\in\ELL_{0}\left(\mu\right),\left\Vert g\right\Vert _{2}\leqslant1\right\} .
\end{align*}
Now, for any $h\in{\rm L}^{1}\bigl(\mu\bigr)$, $h\neq0$, satisfying
$\mu\bigl(h\bigr)=0$, 
\begin{align*}
\sup\left\{ \left\langle f,h\right\rangle :\left\Vert f\right\Vert _{\mathrm{osc}}\leqslant1\right\}  & =\frac{1}{2}\cdot\sup\left\{ \left\langle f,h\right\rangle :\left\Vert f\right\Vert _{\infty}\leqslant1\right\} \\
 & =\frac{1}{2}\cdot\mu\left(\left|h\right|\right).
\end{align*}
The second equality follows from the fact that for $\|f\|_{\infty}<\infty$,
$\left\langle f,h\right\rangle \leqslant\left\langle \left|f\right|,\left|h\right|\right\rangle =\left\langle \left|f\right|\cdot\mathrm{sgn}\left(h\right),h\right\rangle $,
that is, equality holds for $f=g\cdot\mathrm{sgn}\left(h\right)$
with $0\leqslant g\leqslant1$ and the supremum is reached for $g=1$.
Similarly, for the first equality, for $\left\Vert f\right\Vert _{\mathrm{osc}}\leqslant1$
the condition 
\[
\mu\left(h\right)=0\iff\mu\left(h\cdot\mathbf{1}\left[h>0\right]\right)=-\mu\left(h\cdot\mathbf{1}\left[h<0\right]\right)=\frac{1}{2}\mu\left(\left|h\right|\right)
\]
implies that the supremum is reached, for example, for $g=1/2$. The
case $h=0$ is straightforward. Consequently 
\begin{align*}
\sup\left\{ \left\Vert P^{n}f\right\Vert _{2}:f\in\ELL_{0}\left(\mu\right),\left\Vert f\right\Vert _{\mathrm{osc}}\leqslant1\right\} \\
=\frac{1}{2}\cdot\sup & \left\{ \mu\left(\left|\left(P^{*}\right)^{n}g\right|\right):g\in\ELL_{0}\left(\mu\right),\left\Vert g\right\Vert _{2}\leqslant1\right\} ,
\end{align*}
and the first statement follows. For any $g\in\ELL_{0}\left(\mu\right)$
and $n\in\mathbb{N}$, we have
\[
\frac{1}{2}\cdot\mu\left(\left|\left(P^{*}\right)^{n}g\right|\right)\leqslant\gamma_{\star}^{1/2}\left(n\right)\cdot\left\Vert g\right\Vert _{2}\,,
\]
 and for $\nu$ a probability measure such that $\nu\ll\mu$ and $g:=\frac{\mathrm{d}\nu}{\mathrm{d}\mu}-1\in\ELL_{0}\left(\mu\right)$,
we obtain that 
\begin{align*}
\frac{1}{2}\cdot\mu\left(\left|\left(P^{*}\right)^{n}\left(\frac{\mathrm{d}\nu}{\mathrm{d}\mu}-1\right)\right|\right) & =\frac{1}{2}\cdot\mu\left(\left|\frac{\mathrm{d}\left(\nu P^{n}\right)}{\mathrm{d}\mu}-1\right|\right)\\
 & =\left\Vert \nu P^{n}-\mu\right\Vert _{\mathrm{TV}}\,,
\end{align*}
and the second statement follows. Assume now that for any $n\in\mathbb{N}$
and probability $\nu$ obtained as $\nu=g\cdot\mu/\mu\left(g\right)$
for $g\in\ELL\left(\mu\right)$ non-negative, it holds that 
\[
\left\Vert \nu P^{n}-\mu\right\Vert _{\mathrm{TV}}\leqslant\gamma^{1/2}\left(n\right)\cdot\chi^{2}\left(\nu,\mu\right)^{1/2}.
\]
Noting that ${\rm d}\nu/{\rm d}\mu=g/\mu\left(g\right)$, we can conversely
deduce that
\[
\frac{1}{2}\cdot\mu\left(\left|\left(P^{*}\right)^{n}g-\mu\left(g\right)\right|\right)\leqslant\gamma^{1/2}\left(n\right)\cdot\left\Vert g-\mu\left(g\right)\right\Vert _{2}.
\]
For $g\in\ELL\left(\mu\right)$, taking the decomposition into positive
and negative parts $g=g_{+}-g_{-}$ shows that
\begin{align*}
 & \qquad\frac{1}{4}\cdot\mu\left(\left|\left(P^{*}\right)^{n}\left(g-\mu\left(g\right)\right)\right|\right)^{2}\\
 & \leqslant\frac{1}{4}\cdot\mu\left(\left|\left(P^{*}\right)^{n}\left(g_{+}-\mu\left(g_{+}\right)\right)\right|+\left|\left(P^{*}\right)^{n}\left(g_{-}-\mu\left(g_{-}\right)\right)\right|\right)^{2}\\
 & \leqslant\frac{1}{2}\cdot\mu\left(\left|\left(P^{*}\right)^{n}\left(g_{+}-\mu\left(g_{+}\right)\right)\right|\right)^{2}+\frac{1}{2}\cdot\mu\left(\left|\left(P^{*}\right)^{n}\left(g_{-}-\mu\left(g_{-}\right)\right)\right|\right)^{2}\\
 & \leqslant2\cdot\gamma\left(n\right)\cdot\left\{ \left\Vert g_{+}-\mu\left(g_{+}\right)\right\Vert _{2}^{2}+\left\Vert g_{-}-\mu\left(g_{-}\right)\right\Vert _{2}^{2}\right\} \\
 & =2\cdot\gamma\left(n\right)\cdot\left\Vert g-\mu\left(g\right)\right\Vert _{2}^{2},
\end{align*}
from which we may conclude the last part.
\end{proof}
The first part of Proposition~\ref{thm:convergence-duality} seems
relevant in the reversible scenario only. However, as mentioned in
\cite{rockner2001weak}, the following non quantitative result holds
even without reversibility. 
\begin{prop}
The following are equivalent:

\begin{equation}
\lim_{n\to\infty}\sup_{f:\left\Vert f\right\Vert _{2}\leqslant1}\left\Vert P^{n}f-\mu\left(f\right)\right\Vert _{1}=0,\label{eq:l1_conv-disc}
\end{equation}
 and
\begin{equation}
\lim_{n\to\infty}\sup_{f:\left\Vert f\right\Vert _{\infty}\leqslant1}\left\Vert P^{n}f-\mu\left(f\right)\right\Vert _{2}=0.\label{eq:l2_conv-disc}
\end{equation}
\end{prop}

\begin{proof}
We start with (\ref{eq:l1_conv-disc})$\Rightarrow$ (\ref{eq:l2_conv-disc}).
So consider $f$ with $\left\Vert f\right\Vert _{\infty}\leqslant1$.
\begin{align*}
\left\Vert P^{n}f-\mu\left(f\right)\right\Vert _{2}^{2} & =\int\left|P^{n}f-\mu\left(f\right)\right|\cdot\left|P^{n}f-\mu\left(f\right)\right|\,\dif\mu\\
 & \leqslant2\cdot\int\left|P^{n}f-\mu\left(f\right)\right|\,\dif\mu,
\end{align*}
and this final expression converges uniformly over $f$ to 0 by (\ref{eq:l1_conv-disc}),
since $\left\{ f:\left\Vert f\right\Vert _{\infty}\leqslant1\right\} \subseteq\left\{ f:\mu\left(f^{2}\right)\leqslant1\right\} $.
We now consider the converse, (\ref{eq:l2_conv-disc})$\Rightarrow$
(\ref{eq:l1_conv-disc}). Without loss of generality, we may consider
$f\in\mathcal{F}:=\left\{ g\in{\rm L}_{0}^{2}\left(\mu\right):\left\Vert g\right\Vert _{2}\leqslant1\right\} $.
Let $\epsilon>0$ be arbitrary; we will show that for $n$ large enough,
$\sup_{f:\left\Vert f\right\Vert _{2}\le1}\int\left|P^{n}f\right|\,\dif\mu\leqslant\epsilon$.
Take $K=4/\epsilon$ and $N$ large enough such that
\[
\sup_{g:\left\Vert g\right\Vert _{\infty}\leqslant K}\left\Vert P^{N}g-\mu\left(g\right)\right\Vert _{2}\leqslant\frac{\epsilon}{2},
\]
which is valid due to (\ref{eq:l2_conv-disc}). Decomposing an arbitrary
$f\in\mathcal{F}$ as $f=f\cdot{\bf 1}_{A}+f\cdot{\bf 1}_{A^{\complement}}$
for $A\in\mathscr{E}$, we have
\[
\int\left|P^{N}f\right|\,\dif\mu\leqslant\int\left|P^{N}\left(f\cdot{\bf 1}_{A}\right)\right|\,\dif\mu+\int\left|P^{N}\left(f\cdot{\bf 1}_{A^{\complement}}\right)\right|\,\dif\mu,
\]
by Minkowski's inequality. Now by Jensen's inequality, $\mu$-invariance
of $P^{N}$, and the Cauchy--Schwarz inequality, we see that
\[
\int\left|P^{N}\left(f\cdot{\bf 1}_{A^{\complement}}\right)\right|\,\dif\mu\leqslant\int\left|f\cdot{\bf 1}_{A^{\complement}}\right|\,\dif\mu\leqslant\left\Vert f\right\Vert _{2}\cdot\mu\left(A^{\complement}\right)^{1/2}\leqslant\mu\left(A^{\complement}\right)^{1/2}.
\]
Taking $A=\left\{ x\in\mathsf{E}:\left|f\left(x\right)\right|\leqslant K\right\} $,
we obtain by Markov's inequality that
\[
\mu\left(A^{\complement}\right)=\mu\left({\bf 1}_{\left|f\right|^{2}>K^{2}}\right)\leqslant K^{-2}.
\]
From 
\[
\left|\mu\left(f\cdot{\bf 1}_{A^{\complement}}\right)\right|\leqslant\mu\left(\left|f\cdot{\bf 1}_{A^{\complement}}\right|\right)\leqslant K^{-1},
\]
and $\mu\left(f\right)=0$, we also obtain that $\left|\mu\left(f\cdot{\bf 1}_{A}\right)\right|\leqslant K^{-1}$.
Finally, we deduce that
\begin{align*}
\int\left|P^{N}f\right|\,\dif\mu & \leqslant\int\left|P^{N}\left(f\cdot{\bf 1}_{A}\right)\right|\,\dif\mu+\int\left|P^{N}\left(f\cdot{\bf 1}_{A^{\complement}}\right)\right|\,\dif\mu.\\
 & \leqslant\int\left|P^{N}\left(f\cdot{\bf 1}_{A}\right)-\mu\left(f\cdot{\bf 1}_{A}\right)\right|\,\dif\mu+\left|\mu\left(f\cdot{\bf 1}_{A}\right)\right|+K^{-1}\\
 & \leqslant\frac{\epsilon}{2}+2\cdot K^{-1}\\
 & \leqslant\epsilon.
\end{align*}
Since $f\in\mathcal{F}$ was arbitrary, the result follows. 
\end{proof}

\subsection{${\rm L}^{2}$ and ${\rm L}^{1}$ subgeometric convergence rates
can be different and CLT}

Proposition~\ref{thm:convergence-duality} yields that given a bound
on the decay of the total variation distance of the form $\left\Vert \nu P^{n}-\mu\right\Vert _{\mathrm{TV}}\leqslant\gamma^{1/2}\left(n\right)\cdot\chi^{2}\left(\nu,\mu\right)^{1/2}$,
one can deduce an $\mathrm{L}^{2}$ bound on the convergence of bounded
functions as $\left\Vert P^{n}f-\mu\left(f\right)\right\Vert _{2}\leqslant\sqrt{2}\cdot\gamma^{1/2}\left(n\right)\cdot\left\Vert f\right\Vert _{\mathrm{osc}}$.
Note that this transition can be lossy, and that in general, TV and
$\mathrm{L}^{2}$ rates of convergence need not match in the subgeometric
setting. Indeed, in Example~\ref{exa:poly}, we will exhibit a Markov
chain whose (subgeometric) rate of convergence in TV is a strict improvement
on what is implied by its $\mathrm{L}^{2}$ convergence rate. This
may be contrasted with the convergence rates of reversible, geometrically
ergodic Markov chains, which are always the same by \cite[Theorem~3]{roberts2001geometric}.

A secondary point is that standard techniques to derive TV convergence
bounds generally do not lead to tight and convenient estimates of
the type given above. Indeed, consider the scenario where $P$ is
$\mu\shortminus$irreducible and aperiodic, therefore implying that
for $\mu\shortminus$almost all $x\in\mathsf{E}$, $\lim_{n\rightarrow\infty}\left\Vert \delta_{x}P^{n}-\mu\right\Vert _{\mathrm{TV}}=0$.
Recalling that for a signed measure $\nu$ with $\nu\left(1\right)=0$,
we have $\left\Vert \nu\right\Vert _{\mathrm{TV}}=\sup\left\{ \nu\left(f\right):\left\Vert f\right\Vert _{\mathrm{osc}}\leqslant1\right\} $,
we see that
\begin{align*}
\sup\left\{ \left(P^{n}f\right)\left(x\right)-\mu\left(f\right):\left\Vert f\right\Vert _{\mathrm{osc}}\leqslant1\right\}  & =\left\Vert \delta_{x}P^{n}-\mu\right\Vert _{\mathrm{TV}}\\
\implies\left\Vert P^{n}f-\mu\left(f\right)\right\Vert _{2} & \leqslant\mu\left(\left\Vert \delta_{x}P^{n}-\mu\right\Vert _{\mathrm{TV}}^{2}\right)^{1/2}\cdot\left\Vert f\right\Vert _{\mathrm{osc}}.
\end{align*}
Applying the dominated convergence theorem yields that $\mu\big(\left\Vert \delta_{x}P^{n}-\mu\right\Vert _{\mathrm{TV}}^{2}\big)\to0$,
and so we obtain that $\left\Vert P^{n}f-\mu\left(f\right)\right\Vert _{2}\to0$
for all bounded functions. However, converting this estimate into
an explicit quantitative control on the $\mathrm{L}^{2}$ rate of
convergence can prove challenging in practice. In particular, since
TV convergence bounds are often in the form $\left\Vert \delta_{x}P^{n}-\mu\right\Vert _{\mathrm{TV}}\leqslant\min\left\{ 1,V\left(x\right)\cdot\gamma\left(n\right)\right\} $
for choices of $V$, $\gamma$ resulting from a tradeoff between rate
of convergence and dependence on the initial condition. As a result,
capturing fast decay of $n\mapsto\mu\big(\left\Vert \delta_{x}P^{n}-\mu\right\Vert _{\mathrm{TV}}^{2}\big)$
is not particularly simple in general. For example, under the assumption
that $\mu\left(V^{2}\right)<\infty$, one can obtain that $\left\Vert P^{n}f-\mu\left(f\right)\right\Vert _{2}\leqslant\frac{1}{2}\cdot\left\Vert V\right\Vert _{2}\cdot\gamma\left(n\right)\cdot\left\Vert f\right\Vert _{\mathrm{osc}}$,
but experience shows that enforcing this integrability assumption
requires sacrificing rates of convergence, in particular in the subgeometric
scenario.

\begin{example}
\label{exa:poly}Let $\mathsf{E}=\left[1,\infty\right)$, $a>1$,
\[
P\left(x,\cdot\right)=w\left(x\right)\cdot\nu\left(\cdot\right)+\left\{ 1-w\left(x\right)\right\} \cdot\delta_{x}\left(\cdot\right),
\]
where $\nu$ is a probability measure with density
\[
\nu\left(x\right)=\frac{a-1}{x^{a}},\qquad x\in\mathsf{E},
\]
and $w:\mathsf{E}\to\left(0,1\right]$ given by $w\left(x\right)=x^{-b}$
with $b\in\left(0,a-1\right)$. Then $P$ is $\mu$-reversible, where
the probability measure $\mu$ satisfies
\[
\mu\left({\rm d}x\right)=\frac{\nu\left({\rm d}x\right)\cdot w\left(x\right)^{-1}}{\nu\left(w^{-1}\right)},
\]
and therefore has density 
\[
\mu\left(x\right)=\frac{a-b-1}{x^{a-b}},\qquad x\in\mathsf{E}.
\]
\end{example}

The following proposition establishes that for the Markov chain in
Example~\ref{exa:poly}, ${\rm L}^{2}$ and ${\rm L}^{1}$ subgeometric
convergence rates are different.
\begin{prop}
\label{prop:exa-discrepancy-L1-L2-subgeometric} In Example~\ref{exa:poly},
$P$ satisfies:
\begin{enumerate}
\item $n^{s}\cdot\left\Vert P^{n}\left(x,\cdot\right)-\mu\right\Vert _{{\rm TV}}\to0$
for any $s<\frac{a-b-1}{b}$.
\item $n^{s}\cdot\sup_{f\in\ELL_{0}(\mu):\left\Vert f\right\Vert _{{\rm osc}}=1}\left\Vert P^{n}f\right\Vert _{2}<\infty$
if and only if $s\leqslant\frac{a-b-1}{2b}$.
\end{enumerate}
\end{prop}

\begin{proof}
This follows from Lemmas~\ref{lem:jump-example-drift},~\ref{lem:eg-positive}
and~\ref{lem:eg-negative}.
\end{proof}
\begin{rem}
We end this section with some remarks concerning the existence of
a central limit theorem for trajectories of Markov chains. For $f\in\mathrm{L}_{0}^{2}\left(\mu\right)$
the standard Maxwell--Woodroofe condition \cite{maxwell2000central},
\begin{equation}
\sum_{n=1}^{\infty}n^{-3/2}\left\Vert {\textstyle \sum_{k=0}^{n-1}}P^{k}f\right\Vert _{2}<\infty\,,\label{eq:mw-clt-cond}
\end{equation}
is sufficient to ensure $n^{-1/2}\sum_{i=0}^{n-1}f\left(X_{i}\right)\overset{L}{\to}\mathcal{N}\left(0,\sigma^{2}\right)$
for some $\sigma^{2}>0$. $\left(\Psi,\gamma\right)$- convergence
can therefore provide a direct route to checking this condition in
practice. 

For example, suppose that $P$ is $\big(\left\Vert \cdot\right\Vert _{\mathrm{osc}}^{2},\gamma\big)$-convergent.
Then, by combining (\ref{eq:mw-clt-cond}), Proposition~\ref{prop:cattiaux-et-al-orlicz},
and a summation-by-parts, we can obtain that a central limit theorem
will hold for $f\in\ELL_{0}\left(\mu\right)$ with $\left\Vert f\right\Vert _{N}<\infty$
provided that
\begin{equation}
\sum_{k=1}^{\infty}k^{-1/2}\cdot\gamma_{N}(k)^{1/2}<\infty,\label{eq:sum_k_finite}
\end{equation}
where $\gamma_{N}(k)\lesssim\gamma(k)\cdot N^{-1}\big(\gamma(k)^{-1}\big)^{2}$.

We offer now a brief comparison with \cite[Theorem 21.4.9]{douc2018markov},
a parallel result inspired by drift and minorization techniques. The
starting point of \cite{douc2018markov} is a pair of Young functions
$\left(\phi,\psi\right)$: $\phi,\psi$ are continuous, strictly increasing
with $\phi^{-1}\left(x\right)\cdot\psi^{-1}\left(y\right)\leqslant x+y$,
and it is assumed that convergence to equilibrium in total variation
distance occurs at rate $1/\phi$, in the sense that for some constant
$C>0$, one can bound $\int\mu(\dif x)\cdot\left\Vert P^{n}\left(x,\cdot\right)-\mu\right\Vert _{{\rm TV}}\leqslant C/\phi\left(n\right)$.
The conclusion of \cite[Theorem 21.4.9]{douc2018markov} is that a
central limit theorem holds for functions $f$ satisfying $\left\Vert f\right\Vert _{\tilde{N}}<\infty$
with $\tilde{N}(x)=\psi(x^{2})$. 

In our setting, assuming $\left(\Psi,\gamma\right)$- convergence,
application of Theorem~\ref{thm:convergence-duality} would only
yield estimates of the form $\left\Vert \nu P^{n}-\mu\right\Vert _{{\rm TV}}^{2}\lesssim\phi(k)^{-1}\cdot\chi^{2}\left(\nu,\mu\right)$
and direct application of \cite[Theorem 21.4.9]{douc2018markov} seems
to require additional work.  However one can directly work with condition
(\ref{eq:sum_k_finite}). Setting $\phi(k)=\gamma(k)^{-1}$ and $N(x)=\psi(x^{2})$
a necessary condition for the summability of the series concerned
is that the summands decay faster than $k^{-1}$:
\[
k^{-1/2}\cdot\phi\left(k\right)^{-1/2}\cdot\psi^{-1}\left(\phi\left(k\right)\right)^{-1/2}\ll k^{-1}.
\]
In the polynomial case with rate $\gamma(k)\lesssim k^{-a}$ with
$a>1$ and with $N(x)=x^{p}$, this condition requires $p>\frac{2a}{a-1}=2+\frac{2}{a-1}$,
which precisely matches the threshold obtained in \cite[Corollary~21.4.5]{douc2018markov}.
More generally, this condition is equivalent to 
\[
\phi^{-1}\left(z\right)\cdot\psi^{-1}\left(z\right)\ll z,
\]
which closely mirrors the Young function condition of \cite[Theorem 21.4.9]{douc2018markov}.
We leave further investigations of this relation to future work.

\end{rem}

\begin{lem}
\label{lem:jump-example-drift}In Example~\ref{exa:poly}, $P$ satisfies
\[
n^{s}\cdot\left\Vert P^{n}\left(x,\cdot\right)-\mu\right\Vert _{{\rm TV}}\to0,
\]
for any $s<\frac{a-b-1}{b}$.
\end{lem}

\begin{proof}
We first identify a drift function $V$ such that
\[
PV\leqslant V-c\cdot V^{\alpha}+b\cdot{\bf 1}_{C},
\]
with $C$ a small set. Let $k\in\left(0,a-1\right)$, $\alpha=1-b/k$
and define $V\left(x\right)=x^{k}$. We observe that
\[
\nu\left(V\right)=\frac{a-1}{a-k-1}<\infty,
\]
since $k<a-1$. Take $C=\left[1,x_{0}\right]$, where
\[
x_{0}=\left\{ 2\cdot\frac{a-1}{a-k-1}\right\} ^{\frac{1}{\alpha\cdot k}}.
\]
It follows that for $x\in\mathsf{E}$,
\begin{align*}
PV\left(x\right) & =V\left(x\right)-w\left(x\right)\cdot V\left(x\right)+w\left(x\right)\cdot\nu\left(V\right)\\
 & \leqslant V\left(x\right)-w\left(x\right)\cdot V\left(x\right)+\nu\left(V\right)\\
 & =V\left(x\right)-x^{k-b}+\nu\left(V\right)\\
 & =V\left(x\right)-V\left(x\right)^{\alpha}+\nu\left(V\right)
\end{align*}
The choice of $x_{0}$ implies that for $x\in C^{\complement}$, it
holds that
\[
V\left(x\right)^{\alpha}\geqslant x_{0}^{k\alpha}=2\cdot\nu\left(V\right),
\]
and hence, we may verify that
\[
PV\left(x\right)\leqslant V\left(x\right)-\frac{1}{2}V\left(x\right)^{\alpha}+\nu\left(V\right)\cdot{\bf 1}_{C}\left(x\right).
\]
Since $P\left(x,\cdot\right)\geqslant w\left(x_{0}\right)\cdot\nu\left(\cdot\right)$
on $C$, we verify that $C$ is small. By \cite[Theorem~3.6]{jarner2002polynomial},
we may conclude that the rate of convergence in total variation is
given by 
\[
\frac{\alpha}{1-\alpha}=\frac{k-b}{b},
\]
which can be made arbitrarily close to $\left(a-b-1\right)/b$.
\end{proof}
\begin{lem}
\label{lem:compare-measures}Let $\nu\ll\mu$ be probability measures
and let
\[
w=\frac{{\rm d}\mu}{{\rm d}\nu},
\]
with the abuse of notation that we allow $w=\infty$ on regions $A$
where $\mu\left(A\right)>0$ and $\nu\left(A\right)=0$. Then for
any bounded function $f$, it holds that
\[
{\rm var}_{\mu}\left(f\right)\leqslant s\cdot{\rm var}_{\nu}\left(f\right)+\mu\left(w>s\right)\cdot\left\Vert f\right\Vert _{{\rm osc}}^{2}.
\]
\end{lem}

\begin{proof}
Observe that for any $c\in\mathbb{R}$,
\[
{\rm var}_{\mu}\left(f\right)=\inf_{t}\mu\left(\left(f-t\right)^{2}\right)\leqslant\mu\left(\left(f-c\right)^{2}\right).
\]
Take $c=\nu\left(f\right)$ and let $s>0$. Define $A\left(s\right)=\left\{ x:w\left(x\right)\leqslant s\right\} $.
It follows that
\begin{align*}
\mu\left(\left(f-c\right)^{2}\right) & =\mu\left({\bf 1}_{A\left(s\right)}\cdot\left(f-c\right)^{2}\right)+\mu\left({\bf 1}_{A\left(s\right)^{\complement}}\cdot\left(f-c\right)^{2}\right)\\
 & \leqslant s\cdot\nu\left({\bf 1}_{A\left(s\right)}\cdot\left(f-c\right)^{2}\right)+\mu\left(A\left(s\right)^{\complement}\right)\cdot\left\Vert f\right\Vert _{{\rm osc}}^{2}\\
 & \leqslant s\cdot\nu\left(\left(f-c\right)^{2}\right)+\mu\left(w>s\right)\cdot\left\Vert f\right\Vert _{{\rm osc}}^{2}\\
 & =s\cdot\mathrm{var}_{\nu}\left(f\right)+\mu\left(w>s\right)\cdot\left\Vert f\right\Vert _{{\rm osc}}^{2}
\end{align*}
where we have used the fact that since $c=\nu\left(f\right)$ and
$\nu\ll\mu$,
\[
\mathrm{ess_{\mu}}\inf f\leqslant\mathrm{ess_{\nu}}\inf f\leqslant c\leqslant\mathrm{ess_{\nu}}\sup f\leqslant\mathrm{ess_{\mu}}\sup f.
\]
\end{proof}
\begin{lem}
\label{lem:eg-positive}In Example~\ref{exa:poly}, $P$ satisfies
\[
\sup_{f\in\ELL_{0}\left(\pi\right):\left\Vert f\right\Vert _{{\rm osc}}=1}\left\Vert P^{n}f\right\Vert _{2}^{2}\leqslant C\cdot n^{-\frac{a-b-1}{b}},
\]
for some $C>0$.
\end{lem}

\begin{proof}
First observe that for $f\in\ELL_{0}(\mu)$,
\begin{align*}
\left\langle Pf,f\right\rangle  & =\int\mu\left({\rm d}x\right)\cdot P\left(x,\mathrm{d}y\right)\cdot f\left(x\right)\cdot f\left(y\right)\\
 & \geqslant\nu\left(w^{-1}\right)^{-1}\cdot\int\nu\left(\mathrm{d}x\right)\cdot w\left(x\right)^{-1}\cdot w\left(x\right)\cdot\nu\left(\mathrm{d}y\right)\cdot f\left(x\right)\cdot f\left(y\right)\\
 & =\nu\left(w^{-1}\right)^{-1}\cdot\nu\left(f\right)^{2}\\
 & \geqslant0
\end{align*}
and so $P$ is positive and hence $\mathcal{E}\left(P^{2},f\right)\geqslant\mathcal{E}\left(P,f\right)$.
Similarly, we find that
\begin{align*}
\mathcal{E}\left(P,f\right) & =\frac{1}{2}\cdot\int\mu\left({\rm d}x\right)\cdot P\left(x,\mathrm{d}y\right)\cdot\left(f\left(y\right)-f\left(x\right)\right)^{2}\\
 & =\nu\left(w^{-1}\right)^{-1}\cdot\frac{1}{2}\cdot\int\int\nu\left(\mathrm{d}x\right)\cdot\nu\left(\mathrm{d}y\right)\cdot\left(f\left(y\right)-f\left(x\right)\right)^{2}\\
 & =\nu\left(w^{-1}\right)^{-1}\cdot{\rm var}_{\nu}\left(f\right).
\end{align*}
Hence, we may obtain a WPI for $P$ by chaining \cite[Section~3]{andrieu2022comparison_journal}.
Applying Lemma~\ref{lem:compare-measures} and taking $s=t\cdot\nu\left(w^{-1}\right)$,
we have
\begin{align*}
{\rm var}_{\mu}\left(f\right) & \leqslant t\cdot{\rm var}_{\nu}\left(f\right)+\mu\left(\frac{{\rm d}\mu}{{\rm d}\nu}>t\right)\cdot\left\Vert f\right\Vert _{{\rm osc}}^{2}\\
 & =s\cdot\mathcal{E}\left(P,f\right)+\mu\left(\frac{{\rm d}\mu}{{\rm d}\nu}>s\cdot\nu\left(w^{-1}\right)^{-1}\right)\cdot\left\Vert f\right\Vert _{{\rm osc}}^{2}\\
 & =s\cdot\mathcal{E}\left(P,f\right)+\mathbb{P}_{\mu}\left(X^{b}>s\right)\cdot\left\Vert f\right\Vert _{{\rm osc}}^{2},
\end{align*}
so we obtain that $P$, and hence $P^{2}$ by positivity, satisfies
an $\left\Vert \cdot\right\Vert _{{\rm osc}}^{2}$\textit{\emph{--}}WPI
with
\[
\beta\left(s\right)=\mathbb{P}_{\mu}\left(X^{b}>s\right)=s^{-\frac{a-b-1}{b}}.
\]
We conclude by combining Theorem~\ref{thm:WPI_F_bd} and Example~\ref{exa:frombetatogamma-poly}.
\end{proof}
\begin{lem}
\label{lem:eg-negative}In Example~\ref{exa:poly}, if $q>\frac{a-b-1}{b}$,
then there exists $f\in\ELL_{0}\left(\mu\right)$ with $\left\Vert f\right\Vert _{{\rm osc}}=1$
such that 
\[
\left\Vert P^{n}f\right\Vert _{2}^{2}\not\in\mathcal{O}\left(n^{-q}\right).
\]
\end{lem}

\begin{proof}
Let
\begin{align*}
A_{\varepsilon} & =\left\{ x\in\E:P\left(x,\left\{ x\right\} \right)\geqslant1-\varepsilon\right\} \\
 & =\left\{ x\in\E:w\left(x\right)\leqslant\varepsilon\right\} \\
 & =\left[t_{\varepsilon},\infty\right)
\end{align*}
where $t_{\varepsilon}=\varepsilon^{-1/b}$. Direct calculation gives
that
\[
\mu\left(A_{\varepsilon}\right)=\frac{1}{t_{\varepsilon}^{a-b-1}}=\varepsilon^{\frac{a-b-1}{b}}.
\]
The result then follows from Example~\ref{exa:beta-rej-lower-bound-rate}.
\end{proof}

\section{Existence of WPIs\label{sec:Existence-of-WPIs}}

This section is concerned with the existence and practical establishment
of WPIs, and subsequent mixing time bounds, using arguments of a probabilistic
nature. More specifically, we extend in Section~\ref{subsec:Cheeger-inequalities}
the notion of Cheeger inequalities \cite{lawler1988bounds} to Markov
chains which do not admit a right spectral gap. Namely, we show that
establishing certain properties of Dirichlet forms on the \textit{subset}
of indicator functions $f=\mathbf{1}_{A}$ (with $A\in\mathscr{E}$
and $\mu(A)\leqslant1/2$) is equivalent to a WPI. This is then employed
in two separate settings. Then in Section~\ref{subsec:establish-WPI-WPIs-from-RUPI},
we establish necessary and sufficient conditions for the qualitative
\textit{existence} of WPIs for a Markov kernel in terms of its probabilistic
pathwise properties, namely irreducibility and aperiodicity. In Section~\ref{subsec:WPIs-CC-Isop}
we deduce \textit{quantitative} bounds on WPIs and mixing times from
\textit{isoperimetric properties} of the invariant distribution of
the Markov chain and a \textit{close coupling} condition. These latter
results will form the theoretical backbone of our later application
of studying the convergence of Random Walk Metropolis on heavy-tailed
targets in high dimensions; see Section~\ref{sec:Applications-to-RWM}.

\subsection{Weak Cheeger inequalities\label{subsec:Cheeger-inequalities}}

We first introduce a generalisation of the notion of conductance of
a Markov kernel \cite{lawler1988bounds}.
\begin{defn}[Weak conductance profile]
\label{def:JS-weak-conductance-1} The \textit{weak conductance profile}
(WCP) of a $\mu$-invariant Markov kernel $P$ is
\[
\Phi_{P}^{\mathrm{W}}\left(v\right):=\inf\left\{ \frac{\left(\mu\otimes P\right)\left(A\times A^{\complement}\right)}{\mu\left(A\right)}:v\leqslant\mu\left(A\right)\leqslant\frac{1}{2}\right\} ,\qquad v\in\left(0,\frac{1}{2}\right].
\]
We say that $P$ has a \textit{positive WCP} if $\Phi_{P}^{\mathrm{W}}\left(v\right)>0$
for all $v\in\left(0,\frac{1}{2}\right]$.
\end{defn}

The classical notion of conductance corresponds to $\lim_{v\rightarrow0}\Phi_{P}^{\mathrm{W}}\left(v\right)$,
and Lem\-ma~\ref{lem:dirichlet-form-indicator} outlines the link
between this quantity and Dirichlet forms, while the probabilistic
interpretation opens the door to practical uses. Note the distinction
with the `strong' conductance profile (defined in, say, Definition
3 of \cite{andrieu2022explicit}), wherein the infimum is taken over
sufficiently \textit{small} sets $A$, rather than sufficiently large
sets. The main results of this section are the following; the proofs
are at the end of the section.
\begin{thm}[Weak Cheeger inequalities]
\label{thm:positive-CP-implies-optim-WPI} Let $P$ be a $\mu$-invariant
Markov kernel.
\begin{enumerate}
\item If $P$ has a positive WCP $\Phi_{P}^{\mathrm{W}}$, then $P$ satisfies
a $K^{*}$--WPI, for some $K^{*}\colon\left[0,\frac{1}{4}\right]\rightarrow\left[0,\infty\right)$
satisfying for any $v\in\left(0,\frac{1}{4}\right]$,
\[
K^{*}\left(v\right)\geqslant2^{-2}\cdot v\cdot\Phi_{P}^{\mathrm{W}}\left(2^{-2}\cdot v\right)^{2}.
\]
\item If $P$ satisfies a $K^{*}$--WPI then $P$ has a positive WCP with
conductance profile $\Phi_{P}^{\mathrm{W}}$ satisfying for any $v\in\left(0,\frac{1}{4}\right]$,
\[
\frac{1}{2}\cdot v\cdot\Phi_{P}^{\mathrm{W}}\left(\frac{1-\sqrt{1-4v}}{2}\right)\geqslant K^{*}\big(v\big)\,.
\]
\end{enumerate}
\end{thm}

\begin{cor}
\label{cor:mixing-with-weak-conductance} Let $P$ be a $\mu$-reversible
positive Markov kernel with a positive WCP. In Theorem~\ref{thm:WPI_F_bd}
we have for $x\in\left(0,\frac{1}{4}\right]$,

\[
2\cdot\int_{x}^{1/4}\frac{{\rm d}v}{v\cdot\Phi_{P}^{\mathrm{W}}\left(\frac{1-\sqrt{1-4v}}{2}\right)}\leqslant F\left(x\right)\leqslant2^{2}\cdot\int_{x}^{1/4}\frac{{\rm d}v}{v\cdot\Phi_{P}^{\mathrm{W}}\left(2^{-2}\cdot v\right)^{2}}
\]
from which upper and lower bounds on the rate of convergence $\gamma\left(n\right)=F^{-1}\left(n\right)$
can be obtained. Further for any $\epsilon>0$, it holds that 
\[
\sup_{f\in\ELL_{0}(\mu)\setminus\{0\}}\left\Vert P^{n}f\right\Vert _{2}^{2}/\left\Vert f\right\Vert _{\mathrm{osc}}^{2}\leqslant\epsilon,
\]
 whenever $n\geqslant2^{2}\cdot\int_{2^{-2}\epsilon}^{2^{-4}}\frac{\mathrm{d}v}{v\cdot\Phi_{P}^{\mathrm{W}}\left(v\right)^{2}}$.

An expression for $K^{*}$ in the first statement of the theorem
is obtained in the proof, but may be of theoretical interest only.
In an earlier technical report we also established weak Cheeger inequalities
\cite[Theorem 38]{andrieu2022poincare_tech}, but by considering $\alpha$-WPIs
instead of $K^{*}$-WPIs, therefore requiring additional effort to
lower bound $K^{*}$ in terms of $\Phi_{P}^{\mathrm{W}}$ when interested
in mixing times or convergence rates. In addition to providing a more
direct route to applications, the results are sharper (see Appendix~\ref{sec:Compared-sharpness-of}).
We also point out that reversibility is not a requirement in Theorem~\ref{thm:positive-CP-implies-optim-WPI}
and that in fact it is not either in \cite[Theorem 38]{andrieu2022poincare_tech}.
\end{cor}

The proof of Theorem~\ref{thm:positive-CP-implies-optim-WPI} differs
from that of \cite[Theorem 38]{andrieu2022poincare_tech} and relies
on ``$\mathrm{L}^{1}\left(\mu\right)$-- type'' functional inequalities,
which we refer to as $\mathrm{L}^{1}$--WPI inequalities; \cite{bobkov2007large}
refers to these as Weak Cheeger inequalities but we here prefer to
reserve the term for the inequalities in Theorem~\ref{thm:positive-CP-implies-optim-WPI},
by analogy with the use of ``Cheeger inequalities'' in the (discrete-time)
Markov chain literature. The proofs of Theorem~\ref{thm:positive-CP-implies-optim-WPI}
and Corollary~\ref{cor:mixing-with-weak-conductance} can be found
at the end of the section and rely on two preparatory lemmas, Lemma~\ref{lem:WCP-equiv-WCI}
and Lemma~\ref{lem:WCI-implies-WPI}, which link WCP, $\mathrm{L}^{1}$--WPI
and WPIs. 

While various intermediate technical details of our proof are adapted
from \cite{bobkov2007large} to the setting of discrete-time Markov
chains, the overall objectives of the two contributions are fairly
distinct. In particular, \cite{bobkov2007large} specifically focuses
on studying isoperimetric inequalities for probability measures on
$\mathbb{R}^{d}$ and their functional forms, with the connection
to the convergence of Markov semigroups being of secondary interest.
By contrast, we are primarily interested in the application of isoperimetric
concepts to convergence analysis of Markov chains, and so focus more
attention on these aspects of the problem. We also note that related
functional inequalities of the form developed here, and their connections
to WPIs, have been discussed in earlier work; for instance Sections~4~and~5
of \cite{rockner2001weak} consider the links between isoperimetry
and WPIs in the specific scenarios of diffusion processes and continuous-time
Markov jump processes. 

Recall that for $f\in\mathrm{L}^{1}\left(\mu\right)$, we defined
$\var_{\mu}^{\left(1\right)}\left(f\right):=\inf_{c\in\R}\left\Vert f-c\right\Vert _{1}$
and $\mathcal{E}_{1}\left(P,f\right):=\frac{1}{2}\cdot\int\mu\left(\mathrm{d}x\right)\cdot P\left(x,\mathrm{d}y\right)\cdot\left|f\left(y\right)-f\left(x\right)\right|$. 
\begin{defn}[$\mathrm{L}^{1}$--WPI]
\label{def:WCI} We say that a $\mu$-invariant kernel $T$ satisfies
a $\mathrm{L}^{1}$\textit{--weak Poincaré inequality}, abbreviated
$\mathrm{L}^{1}$--WPI, if for a decreasing function $\alpha:\left(0,\infty\right)\to\left[0,\infty\right)$
such that $\alpha\left(r\right)>0$ for $r\in\left(0,\frac{1}{2}\right)$,
\begin{equation}
\var_{\mu}^{\left(1\right)}\left(f\right)\leqslant\alpha\left(r\right)\cdot\calE_{1}\left(T,f\right)+r\cdot\left\Vert f\right\Vert _{\mathrm{osc}},\quad\forall r>0,f\in\mathrm{L}^{1}\left(\mu\right).\label{eq:alpha-WCI}
\end{equation}
\end{defn}

\begin{rem}
Note that with $A\in\mathscr{E}$ such that $\mu\left(A\right)=\frac{1}{2}$
then letting $f=\mathbf{1}_{A}$ and remarking that $\var_{\mu}^{(1)}\left(\mathbf{1}_{A}\right)=\frac{1}{2}$
the existence of an $\mathrm{L}^{1}$--WPI implies
\[
\frac{1}{2}\leqslant\alpha\left(r\right)\cdot\calE_{1}\left(T,\mathbf{1}_{A}\right)+r\cdot1,
\]
from which we see that the condition $\alpha\left(r\right)>0$ for
$r\in\left(0,\frac{1}{2}\right)$ is a requirement. Further, since
$\var_{\mu}^{\left(1\right)}\left(f\right)\leqslant\frac{1}{2}\cdot\left\Vert f\right\Vert _{\mathrm{osc}}$
by considering $c=\frac{1}{2}\cdot\left(\muess\sup f-\muess\inf f\right)$,
we see that (\ref{eq:alpha-WCI}) is non-trivial for $r\in\left(0,\frac{1}{2}\right)$
only. 
\end{rem}

\begin{rem}
As for standard WPIs, it is possible to parametrize $\mathrm{L}^{1}$--WPIs
in different ways; however the $\alpha-$parametrization will be sufficient
for our purposes.
\end{rem}

The following is related to aspects of \cite[Lemma 8.3]{bobkov2007large}.
\begin{lem}
\label{lem:WCP-equiv-WCI}Let $P$ be a $\mu$-invariant Markov kernel.
Then, 
\begin{enumerate}
\item if $P$ has a positive WCP $\Phi_{P}^{\mathrm{W}}$, then it satisfies
an $\mathrm{L}^{1}$--WPI with 
\begin{equation}
\alpha_{1}\left(r\right):=\sup_{r\leqslant s\leqslant\frac{1}{2}}\left\{ \frac{s-r}{s\cdot\Phi_{P}^{\mathrm{W}}\left(s\right)}\right\} ,\qquad r\in\left(0,\frac{1}{2}\right);\label{eq:L1-WPI}
\end{equation}
\item if $P$ satisfies an $\mathrm{L}^{1}$--WPI for some $\alpha_{1}$
then it has a positive WCP and
\begin{equation}
\Phi_{P}^{\mathrm{W}}\left(v\right)\geqslant\sup_{0<r\leqslant v}\left\{ \frac{v-r}{\alpha_{1}\left(r\right)\cdot v}\right\} ,\qquad v\in\left(0,\frac{1}{2}\right].\label{eq:WCP}
\end{equation}
\end{enumerate}
Additionally, if $\alpha_{1}$ in (\ref{eq:WCP}) is the optimal $\mathrm{L}^{1}$--WPI
function given in (\ref{eq:L1-WPI}), then equality holds in (\ref{eq:WCP}).
\end{lem}

\begin{proof}
Since a $\mathrm{L}^{1}$--WPI depends only on $P$ through $\mathcal{E}_{1}\left(P,\cdot\right)$
and a WCP depends only on $P$ through $\mathcal{E}_{1}\left(P,\cdot\right)$
or $\mathcal{E}_{2}\left(P,\cdot\right)$, we may assume without loss
of generality that $P$ is $\mu$-reversible; otherwise we may treat
$S=\frac{1}{2}\left(P+P^{*}\right)$.

Suppose first that $P$ has a positive WCP, $\Phi_{P}^{\mathrm{W}}.$
Let $f\in\left[0,1\right]$ such that $\mu\left(\left\{ f=0\right\} \right)\geqslant\frac{1}{2}$
and $\essup f=1$. Since $P$ may be assumed reversible, we mirror
the approaches of \cite{lawler1988bounds} to deduce that
\begin{align*}
\mathcal{E}_{1}\left(P,f\right) & =\frac{1}{2}\cdot\int\mu\left(\mathrm{d}x\right)\cdot P\left(x,\mathrm{d}y\right)\cdot\left|f\left(y\right)-f\left(x\right)\right|\\
 & =\int\mu\left(\mathrm{d}x\right)\cdot P\left(x,\mathrm{d}y\right)\cdot\left\{ f\left(y\right)-f\left(x\right)\right\} \cdot{\bf 1}\left[f\left(x\right)<f\left(y\right)\right]\\
 & =\int\mu\left(\mathrm{d}x\right)\cdot P\left(x,\mathrm{d}y\right)\cdot\int{\bf 1}\left[f\left(x\right)\leqslant t\right]\cdot\mathbf{1}\left[f\left(y\right)>t\right]\,\mathrm{d}t\\
 & =\int_{0}^{1}\left(\mu\otimes P\right)\left(F_{t}^{\complement}\times F_{t}\right)\,\mathrm{d}t\\
 & =\int_{0}^{1}\left(\mu\otimes P\right)\left(F_{t}\times F_{t}^{\complement}\right)\,\mathrm{d}t,
\end{align*}
with $F_{t}:=\left\{ x\in\mathsf{E}\colon f\left(x\right)>t\right\} $.
Note also that $\mu\left(F_{t}\right)\leqslant\frac{1}{2}$ for $t>0$
since $\mu\left(\left\{ f=0\right\} \right)\geqslant\frac{1}{2}$.
Further, by the assumptions on $f$, for any $r>0$,
\[
\mu\left(f\right)-r\cdot\left\Vert f\right\Vert _{{\rm osc}}=\int_{0}^{1}\left\{ \mu\left(F_{t}\right)-r\right\} \,\mathrm{d}t.
\]
Thus by comparing integrands, we deduce that a sufficient condition
for
\[
\mu\left(f\right)-r\cdot\left\Vert f\right\Vert _{{\rm osc}}\leqslant\alpha_{1}\left(r\right)\cdot\mathcal{E}_{1}\left(P,f\right)
\]
to hold for a given $\alpha_{1}\left(r\right)\geqslant0$ is that
\begin{align*}
\mu\left(A\right)-r & \leqslant\alpha_{1}\left(r\right)\cdot\left(\mu\otimes P\right)\left(A\times A^{\complement}\right),
\end{align*}
for all $A\in\mathscr{E}$ with $\mu\left(A\right)\leqslant\frac{1}{2}$.
From the assumed WCP, we have that 
\[
\left(\mu\otimes P\right)\left(A\times A^{\complement}\right)\geqslant\mu\left(A\right)\cdot\Phi_{P}^{\mathrm{W}}\left(\mu\left(A\right)\right)\,,
\]
implying that a sufficient condition is, with $s:=\mu\left(A\right)$
and for $r\leqslant s$ to avoid a trivial bound,
\begin{align*}
s-r & \leqslant\alpha_{1}\left(r\right)\cdot s\cdot\Phi_{P}^{\mathrm{W}}\left(s\right),\\
\alpha_{1}\left(r\right) & :=\sup_{r\leqslant s\leqslant\frac{1}{2}}\left\{ \frac{s-r}{s\cdot\Phi_{P}^{\mathrm{W}}\left(s\right)}\right\} >0,
\end{align*}
which is a decreasing function on $\left(0,\frac{1}{2}\right)$. For
$f\in\mathrm{L}^{1}\left(\mu\right)$ satisfying $\mathrm{ess_{\mu}}\sup f=\mathrm{ess_{\mu}}\sup\left(-f\right)=1$
and $\mu\left(\left\{ f\geqslant0\right\} \right)\geqslant\frac{1}{2},\mu\big(\left\{ f\leqslant0\right\} \big)\geqslant\frac{1}{2}$,
consider the two functions $f_{\pm}=\max\left\{ \pm f,0\right\} $
and observe that $\mu\left(\left\{ f_{\pm}=0\right\} \right)\geqslant\frac{1}{2}$.
The result above yields 
\[
\mu\left(f_{\pm}\right)\leqslant\alpha_{1}\left(r\right)\cdot\mathcal{E}_{1}\left(P,f_{\pm}\right)+r\cdot\left\Vert f_{\pm}\right\Vert _{{\rm osc}}.
\]
Adding up these two inequalities and using that $\mu\left(f_{+}\right)+\mu\left(f_{-}\right)=\mu\left(\left|f\right|\right)$,
$\mathcal{E}_{1}\left(P,f_{+}\right)+\mathcal{E}_{1}\left(P,f_{-}\right)=\mathcal{E}_{1}\left(P,f\right)$
and $\left\Vert f_{+}\right\Vert _{{\rm osc}}+\left\Vert f_{-}\right\Vert _{{\rm osc}}=\left\Vert f\right\Vert _{{\rm osc}}$
gives that
\[
\mu\left(\left|f\right|\right)\leqslant\alpha_{1}\left(r\right)\cdot\mathcal{E}_{1}\left(P,f\right)+r\cdot\left\Vert f\right\Vert _{{\rm osc}},
\]
and it holds in general that $\var_{\mu}^{\left(1\right)}\left(f\right)\leqslant\mu\left(\left|f\right|\right)$.
For general bounded $f\in\mathrm{L}^{1}\left(\mu\right)$, one can
always find constants $c_{1},c_{2}\in\mathbb{R}$ such that $f=c_{1}+c_{2}\cdot\tilde{f}$
with $\mathrm{ess_{\mu}}\sup\tilde{f}=\mathrm{ess_{\mu}}\sup\big(-\tilde{f}\big)=1$
and apply the above; we thus conclude.

Assume now that an $\mathrm{L}^{1}$--WPI holds. For $A\in\mathcal{E}$,
let $f=\mathbf{1}_{A}$ and note that $\var_{\mu}^{\left(1\right)}\left(\mathbf{1}_{A}\right)=\min\left\{ \mu\left(A\right),\mu\left(A^{\complement}\right)\right\} $.
Now, from Lemma~\ref{lem:dirichlet-form-indicator} we have $\mathcal{E}_{1}\left(P,\mathbf{1}_{A}\right)=\mu\otimes P\left(A\times A^{\complement}\right)$
and the $\mathrm{L^{1}}$--WPI therefore implies for all $r>0$ that
\begin{align*}
\min\left\{ \mu\left(A\right),\mu\left(A^{\complement}\right)\right\}  & \leqslant\alpha_{1}\left(r\right)\cdot\mu\otimes P\left(A\times A^{\complement}\right)+r\cdot1.
\end{align*}
Notice again that for $\mu\left(A\right)=\frac{1}{2}$, we see that
this requires $\alpha_{1}\left(r\right)>0$ for $r\in\left(0,\frac{1}{2}\right)$.
Therefore for $r\in\left(0,\frac{1}{2}\right)$,
\[
\frac{\min\left\{ \mu\left(A\right),\mu\left(A^{\complement}\right)\right\} -r}{\alpha_{1}\left(r\right)\cdot\min\left\{ \mu\left(A\right),\mu\left(A^{\complement}\right)\right\} }\leqslant\frac{\mu\otimes P\left(A\times A^{\complement}\right)}{\min\left\{ \mu\left(A\right),\mu\left(A^{\complement}\right)\right\} },
\]
and hence for $v\in\left(0,\frac{1}{2}\right]$, $r\in\left(0,\frac{1}{2}\right)$
we have
\begin{align*}
\Phi_{P}^{\mathrm{W}}\left(v\right) & \geqslant\inf\left\{ \frac{\min\left\{ \mu\left(A\right),\mu\left(A^{\complement}\right)\right\} -r}{\alpha_{1}\left(r\right)\cdot\min\left\{ \mu\left(A\right),\mu\left(A^{\complement}\right)\right\} }:\min\left\{ \mu\left(A\right),\mu\left(A^{\complement}\right)\right\} \geqslant v\right\} \\
 & \geqslant\inf\left\{ \frac{s-r}{\alpha_{1}\left(r\right)\cdot s}:s\in\left[v,\frac{1}{2}\right]\right\} \\
 & =\frac{v-r}{\alpha_{1}\left(r\right)\cdot v}.
\end{align*}
Finally, for $v\in\left(0,\frac{1}{2}\right]$ take a supremum over
$r\in\left(0,v\right]$ to see that
\[
\Phi_{P}^{\mathrm{W}}\left(v\right)\geqslant\sup_{0<r\leqslant v}\left\{ \frac{v-r}{\alpha_{1}\left(r\right)\cdot v}\right\} >0,
\]
showing that $P$ has a positive WCP.

For sharpness, the preceding calculations demonstrate that the optimal
$\Phi_{P}^{\mathrm{W}}$ and $\alpha_{1}$ are in bijection with one
another as the sharp functions in the inequality
\[
\Phi_{P}^{\mathrm{W}}\left(v\right)\cdot\alpha_{1}\left(r\right)\geqslant\frac{v-r}{v},\qquad v\in\left(0,\frac{1}{2}\right],\,r\in\left(0,\frac{1}{2}\right),
\]
from which the final claim follows.
\end{proof}
\begin{rem}
\label{rem:bound-alpha-1-Phi}From the existence of a positive WCP
and the theorem above, one can obtain a bound on the corresponding
$\mathrm{L^{1}}$--WPI function, for $r\in\left(0,\frac{1}{2}\right)$:

\[
\alpha_{1}\left(r\right)=\sup_{r\leqslant s\leqslant\frac{1}{2}}\left\{ \frac{s-r}{s\cdot\Phi_{P}^{\mathrm{W}}\left(s\right)}\right\} \leqslant\sup_{r\leqslant s\leqslant\frac{1}{2}}\left\{ \frac{1}{\Phi_{P}^{\mathrm{W}}\left(s\right)}\right\} \leqslant\Phi_{P}^{\mathrm{W}}\left(r\right)^{-1},
\]
since $\Phi_{P}^{\mathrm{W}}$ is non-decreasing.
\end{rem}

We now show, by following arguments of \cite{bobkov2007large}, that
an $\mathrm{L}^{1}$--WPI implies an $\alpha$--WPI, with an explicit
bound on the function $\alpha$.
\begin{lem}
\label{lem:WCI-implies-WPI}Let $P$ be a $\mu$-invariant Markov
kernel satisfying an $\mathrm{L}^{1}$--WPI with function $\alpha_{1}$.
Then $P$ also satisfies an $\alpha$--WPI with 
\[
\alpha\left(r\right)=\inf_{\theta\in\left(0,1\right)}\left\{ \frac{\alpha_{1}\left(\left(1-\theta\right)\cdot r\right)^{2}}{2\cdot\theta\cdot\left(1-\theta\right)}\right\} \leqslant2\cdot\alpha_{1}\left(\frac{1}{2}\cdot r\right)^{2},\quad r>0.
\]
\end{lem}

\begin{proof}
Assume that a $\mathrm{L^{1}}$--WPI holds. Let $f\in\ELL\left(\mu\right)$
be nonnegative with $\mu\left(\left\{ f=0\right\} \right)=\mu\left(\left\{ f^{2}=0\right\} \right)\geqslant\frac{1}{2}$.
We then apply the $\mathrm{L}^{1}$--WPI to $f^{2}$, followed by
the Cauchy--Schwarz and triangle inequalities, together with the
inequality $2ab\leqslant\lambda a^{2}+\lambda^{-1}b^{2}$, to see
for $r>0$, $\theta\in\left(0,1\right)$ that
\begin{align*}
\left\Vert f\right\Vert _{2}^{2}=\var_{\mu}^{(1)}\left(f^{2}\right) & \leqslant\alpha_{1}\left(r\right)\cdot\mathcal{E}_{1}\left(P,f^{2}\right)+r\cdot\left\Vert f^{2}\right\Vert _{\mathrm{osc}}\\
 & \leqslant\sqrt{2}\cdot\alpha_{1}\left(r\right)\cdot\left\Vert f\right\Vert _{2}\cdot\mathcal{E}\left(P,f\right)^{1/2}+r\cdot\left\Vert f\right\Vert _{\mathrm{osc}}^{2}\\
 & \leqslant\theta\cdot\left\Vert f\right\Vert _{2}^{2}+\frac{\alpha_{1}\left(r\right)^{2}}{2\cdot\theta}\cdot\mathcal{E}\left(P,f\right)+r\cdot\left\Vert f\right\Vert _{\mathrm{osc}}^{2}.
\end{align*}
Rearranging, we obtain for $r>0$ that
\[
\left\Vert f\right\Vert _{2}^{2}\leqslant\frac{\alpha_{1}\left(r\right)^{2}}{2\cdot\theta\cdot\left(1-\theta\right)}\cdot\mathcal{E}\left(P,f\right)+\frac{r}{1-\theta}\cdot\left\Vert f\right\Vert _{\mathrm{osc}}^{2},
\]
that is, for $\tilde{r}>0$,
\begin{align*}
\left\Vert f\right\Vert _{2}^{2} & \leqslant\frac{\alpha_{1}\left(\tilde{r}\cdot\left(1-\theta\right)\right)^{2}}{2\cdot\theta\cdot\left(1-\theta\right)}\cdot\mathcal{E}\left(P,f\right)+\tilde{r}\cdot\left\Vert f\right\Vert _{\mathrm{osc}}^{2}.
\end{align*}
Now, for general bounded $f\in\ELL\left(\mu\right)$ such that $\mu\left(\left\{ f\geqslant0\right\} \right)\geqslant\frac{1}{2},\mu\big(\left\{ f\leqslant0\right\} \big)\geqslant\frac{1}{2}$
we apply the above to $f_{\pm}=\max\left\{ \pm f,0\right\} $ such
that $\mu\left(\left\{ f_{+}\geqslant0\right\} \right),\mu\big(\left\{ f_{-}\geqslant0\right\} \big)\geqslant\frac{1}{2}$.
Recalling that $\left\Vert f_{+}\right\Vert _{2}^{2}+\left\Vert f_{-}\right\Vert _{2}^{2}=\left\Vert f\right\Vert _{2}^{2}$,
$\mathcal{E}\left(P,f_{+}\right)+\mathcal{E}\left(P,f_{-}\right)\leqslant\mathcal{E}\left(P,f\right)$,
and $\left\Vert f_{+}\right\Vert _{\mathrm{osc}}^{2}+\left\Vert f_{-}\right\Vert _{\mathrm{osc}}^{2}\leqslant\left\Vert f\right\Vert _{\mathrm{osc}}^{2}$,
it follows that for $r>0$ and $\theta\in\left(0,1\right)$,
\[
\left\Vert f\right\Vert _{2}^{2}\leqslant\frac{\alpha_{1}\left(\left(1-\theta\right)\cdot r\right)^{2}}{2\cdot\theta\cdot\left(1-\theta\right)}\cdot\mathcal{E}\left(P,f\right)+r\cdot\left\Vert f\right\Vert _{\mathrm{osc}}^{2}.
\]
Defining $\alpha:\left(0,\infty\right)\to\left[0,\infty\right)$ by
\begin{align*}
\alpha\left(r\right) & =\inf_{\theta\in\left(0,1\right)}\frac{\alpha_{1}\left(\left(1-\theta\right)\cdot r\right)^{2}}{2\cdot\theta\cdot\left(1-\theta\right)}\\
 & \leqslant2\cdot\alpha_{1}\left(\frac{1}{2}\cdot r\right)^{2},
\end{align*}
we thus have with $\bar{f}=f-\mu\left(f\right)$ that
\[
\left\Vert \bar{f}\right\Vert _{2}^{2}\leqslant\left\Vert f\right\Vert _{2}^{2}\leqslant\alpha\left(r\right)\cdot\mathcal{E}\left(P,\bar{f}\right)+r\cdot\left\Vert \bar{f}\right\Vert _{\mathrm{osc}}^{2}.
\]
Finally, for general $f\in\ELL\left(\mu\right)$, we apply the above
to $f-\mathrm{med}_{\mu}\left(f\right)$, where $\mathrm{med}_{\mu}\left(f\right)$
is a median of $f$ under $\mu$, and conclude that an $\alpha$--WPI
holds since $\alpha$ is decreasing. 
\end{proof}
We now present the proof of the main result, i.e. Theorem \ref{thm:positive-CP-implies-optim-WPI}
\begin{proof}[Proof of Theorem~\ref{thm:positive-CP-implies-optim-WPI}]
 For the first statement. By Lemmas~\ref{lem:WCP-equiv-WCI},~\ref{lem:WCI-implies-WPI},
we can define $\alpha\left(r\right):=2\cdot\alpha_{1}\left(\frac{1}{2}\cdot r\right)^{2}$
and $\alpha_{1}\colon\left(0,\frac{1}{2}\right)\to\left(0,\infty\right)$
is as in Lemma~\ref{lem:WCP-equiv-WCI}, so that for all $f\in\ELL_{0}\left(\pi\right)$
and $r>0$, we have that 
\begin{align*}
\left\Vert f\right\Vert _{2}^{2} & \leqslant\alpha\left(r\right)\cdot\mathcal{E}\left(P,f\right)+r\cdot\left\Vert f\right\Vert _{\mathrm{osc}}^{2}.
\end{align*}
Furthermore for $r\in\left(0,\frac{1}{2}\right)$ we can write
\[
\frac{\mathcal{E}\left(P,f\right)}{\left\Vert f\right\Vert _{\mathrm{osc}}^{2}}\geqslant\frac{1}{\alpha\left(r\right)}\cdot\left\{ \frac{\left\Vert f\right\Vert _{2}^{2}}{\left\Vert f\right\Vert _{\mathrm{osc}}^{2}}-r\right\} .
\]
Now defining 
\begin{equation}
K^{*}\left(v\right):=\sup_{r\in\left(0,\frac{1}{4}\right]}\left\{ \frac{v-r}{\alpha\left(r\right)}\right\} ,\label{eq:CP-WPI-K*}
\end{equation}
and recalling that $\left\Vert f\right\Vert _{2}^{2}/\left\Vert f\right\Vert _{\mathrm{osc}}^{2}\leqslant\frac{1}{4}$,
we see that this implies $\frac{\mathcal{E}\left(P,f\right)}{\left\Vert f\right\Vert _{\mathrm{osc}}^{2}}\geqslant K^{*}\left(\frac{\left\Vert f\right\Vert _{2}^{2}}{\left\Vert f\right\Vert _{\mathrm{osc}}^{2}}\right)$.
From the definition of $K^{*}$ it is clear that it is nondecreasing,
and as a supremum of linear functions it is convex. From Lemma~\ref{lem:WCI-implies-WPI}
and Remark~\ref{rem:bound-alpha-1-Phi}, we have $\alpha\left(r\right)\leqslant2\cdot\Phi_{P}^{\mathrm{W}}\left(\frac{1}{2}\cdot r\right)^{-2}$
for $r\in\left[0,1\right]$, and hence taking $r=\frac{1}{2}\cdot v$
in the definition of $K^{*}$, one obtains for $v\in\left[0,1\right]$
that
\begin{equation}
K^{*}\left(v\right)\geqslant\frac{1}{2}\cdot v\cdot\alpha\left(\frac{1}{2}\cdot v\right)^{-1}\geqslant\frac{1}{2}\cdot v\cdot\left(2\cdot\Phi_{P}^{\mathrm{W}}\left(\frac{1}{4}\cdot v\right)^{-2}\right)^{-1}=2^{-2}\cdot v\cdot\Phi_{P}^{\mathrm{W}}\left(2^{-2}\cdot v\right)^{2};\label{eq:lower-bound-Kstar}
\end{equation}
we are thus able to conclude. For the second statement, let $A\in\mathcal{E}$
such that $\mu\left(A\right)\in\left(0,1\right)$ and consider the
function 
\[
f=\frac{\mathbf{1}_{A}-\mu\left(A\right)}{\sqrt{\mu\left(A\right)\cdot\mu\left(A^{\complement}\right)}},
\]
so that from Lemma~\ref{lem:dirichlet-form-indicator}, it holds
that $\left\Vert f\right\Vert _{2}^{2}=1$ and $1/\left\Vert f\right\Vert _{\mathrm{osc}}^{2}=\mu\left(A\right)\cdot\mu\left(A^{\complement}\right)$.
Assumed existence of a $K^{*}$--WPI and Lemma~\ref{lem:dirichlet-form-indicator}
together imply that
\[
\frac{\mu\otimes P\left(A\times A^{\complement}\right)}{\mu\left(A\right)}\geqslant\frac{K^{*}\left(\mu\otimes\mu\left(A\times A^{\complement}\right)\right)}{\mu\left(A\right)\cdot\mu\otimes\mu\left(A\times A^{\complement}\right)},
\]
from which we deduce for $v\in\left(0,\frac{1}{2}\right]$ that
\begin{align*}
\inf\left\{ \frac{\left(\mu\otimes P\right)\left(A\times A^{\complement}\right)}{\mu\left(A\right)}:v\leqslant\mu\left(A\right)\leqslant\frac{1}{2}\right\}  & \geqslant\inf\left\{ \frac{K^{*}\left(s\cdot\left(1-s\right)\right)}{s\cdot s\cdot\left(1-s\right)}:v\leqslant s\leqslant\frac{1}{2}\right\} \\
 & \geqslant2\cdot\frac{K^{*}\left(v\cdot\left(1-v\right)\right)}{v\cdot\left(1-v\right)}\,,
\end{align*}
where the last inequality follows from noting that $v\mapsto v^{-1}\cdot K^{*}\left(v\right)$
is increasing on $\left[0,\frac{1}{4}\right]$ from convexity and
$K^{*}\left(0\right)=0$. Now for $v\in\left(0,\frac{1}{2}\right]$,
observe that $u=v\cdot\left(1-v\right)\in\left(0,\frac{1}{4}\right]$,
and so taking $v=\left(1-\sqrt{1-4u}\right)/2$ allows us to conclude.
\end{proof}
\begin{proof}[Proof of Corollary~\ref{cor:mixing-with-weak-conductance}]
With $F$ as defined in Theorem \ref{thm:WPI_F_bd}, the upper and
lower bounds follow immediately from the upper and lower bounds on
$K^{*}$ from Theorem \ref{thm:positive-CP-implies-optim-WPI}. To
establish the subsequent claim that taking $n$ as written will guarantee
$\left\Vert P^{n}f\right\Vert _{2}^{2}/\left\Vert f\right\Vert _{\mathrm{osc}}^{2}\leqslant\epsilon$,
it remains to ensure that $K^{*}$ has a form which allows us to apply
the reasoning from \cite[Proof of Theorem 8]{andrieu2022comparison_journal},
namely that $K^{*}$ is the convex conjugate of $K\left(u\right):=u\,\beta\left(1/u\right)$
for $u>0$ and $K\left(0\right)=0$, where $\beta:=\alpha^{\shortminus}$
with $\alpha^{\shortminus}\left(s\right):=\inf\left\{ u>0\colon\alpha\left(u\right)\leqslant s\right\} $
for $s>0$. One can check that $\alpha_{1}\colon\left(0,\frac{1}{2}\right)\to\left(0,\infty\right)$
in Lemma~\ref{lem:WCP-equiv-WCI} is strictly decreasing, implying
that $\alpha\colon\left(0,\frac{1}{4}\right]\to\left[0,\infty\right)$
is also strictly decreasing. Setting $\alpha\left(r\right)=0$ for
$r>\mathfrak{a}=\frac{1}{4}$ (as we as free to do) and appealing
to \cite[Proposition 5]{andrieu2022poincare_tech} yields that $\alpha^{-}\circ\alpha\left(r\right)=r$
for $r\in\left(0,\frac{1}{4}\right]$. Consequently with $u:=\left[\alpha\left(r\right)\right]^{-1}$,
we have that $\alpha^{-}\left(u^{-1}\right)=r$, and deduce the representation
\[
\sup_{r\in\left(0,\frac{1}{4}\right]}\left\{ \frac{v-r}{\alpha\left(r\right)}\right\} =\sup_{u\in\left[0,\infty\right)}\left\{ v\cdot u-u\cdot\alpha^{-}\left(u^{-1}\right)\right\} .
\]
We can therefore use the properties of \cite[Lemma~1, Supp. Material]{andrieu2022comparison_journal}.
From the assumed positivity and reversibility of $P$ and Theorem~\ref{thm:positive-CP-implies-optim-WPI},
we have for $f\in\ELL\left(\mu\right)$,
\[
\frac{\mathcal{E}\left(P^{*}P,f\right)}{\left\Vert f\right\Vert _{\mathrm{osc}}^{2}}=\frac{\mathcal{E}\left(P^{2},f\right)}{\left\Vert f\right\Vert _{\mathrm{osc}}^{2}}\geqslant\frac{\mathcal{E}\left(P,f\right)}{\left\Vert f\right\Vert _{\mathrm{osc}}^{2}}\geqslant K^{*}\left(\frac{\left\Vert f\right\Vert _{2}^{2}}{\left\Vert f\right\Vert _{\mathrm{osc}}^{2}}\right).
\]
Finally, following \cite[Proof of Theorem 8, Supp. Material]{andrieu2022comparison_journal},
with $n\in\mathbb{N}$, $v_{0}=\mathfrak{a}=\frac{1}{4}$ and $v_{n}=\left\Vert P^{n}f\right\Vert _{2}^{2}/\left\Vert f\right\Vert _{\mathrm{osc}}^{2}$,
there holds the bound
\begin{align*}
n\leqslant\int_{v_{n}}^{v_{0}}\frac{{\rm d}v}{K^{*}\left(v\right)} & \leqslant2^{2}\cdot\int_{v_{n}}^{v_{0}}\frac{{\rm d}v}{v\cdot\Phi_{P}^{\mathrm{W}}\left(2^{-2}\cdot v\right)^{2}}=2^{2}\cdot\int_{2^{-2}\cdot v_{n}}^{2^{-4}}\frac{{\rm d}v}{v\cdot\Phi_{P}^{\mathrm{W}}\left(v\right)^{2}}.
\end{align*}
As such, taking $n\geqslant2^{2}\cdot\int_{2^{-2}\cdot\epsilon}^{2^{-4}}\frac{\mathrm{d}v}{v\cdot\Phi_{P}^{\mathrm{W}}\left(v\right)^{2}}$
guarantees that $v_{n}\leqslant\epsilon$, and so we conclude.
\end{proof}

\subsection{$\mu$--irreducibility, aperiodicity and WPIs\label{subsec:establish-WPI-WPIs-from-RUPI}}

Building on the concepts and results given in Section~\ref{subsec:Cheeger-inequalities},
this section establishes necessary and sufficient conditions for the
existence of WPIs, in terms of probabilistic properties of the Markov
chain. The results are predominantly theoretical in nature; quantitative
bounds are discussed in Section~\ref{subsec:WPIs-CC-Isop}. We consider
here $\left\Vert \cdot\right\Vert _{\mathrm{osc}}^{2}$\textit{\emph{--}}WPIs
and $\left\Vert \cdot\right\Vert _{\mathrm{osc}}^{2}$-convergence.
Throughout this section, we are working with a $\mu$-invariant Markov
kernel $P$, where $\mu$ is a probability measure.

From earlier sections, we have seen that the existence of a WPI for
the kernel $T_{k}:=\left(P^{*}\right)^{k}P^{k}$ for some $k\in\mathbb{N}$
enables one to deduce (subgeometric) convergence bounds for $\left\Vert P^{k\cdot n}f\right\Vert _{2}$,
and hence $\left\Vert P^{n}f\right\Vert _{2}$, for bounded $f$.
This section is concerned with 
\begin{enumerate}
\item the necessity of the existence of a WPI for some $T_{k}$ to ensure
convergence of $P$;
\item whether simple conditions such as $\mu$-irreducibility and aperiodicity
can imply the existence of a WPI.
\end{enumerate}
More specifically, we know that for a $\mu$-invariant, aperiodic
and $\mu$-irreducible Markov kernel, under mild regularity conditions,
for $\mu$-almost all $x\in\E$, 
\[
\lim_{n\rightarrow\infty}\left\Vert \delta_{x}P^{n}-\mu\right\Vert _{\mathrm{TV}}=0;
\]
see \cite[Theorem~13.3.5]{meyn1993markov}. This implies that 
\begin{equation}
\left\Vert P^{n}f-\mu\left(f\right)\right\Vert _{2}\leqslant\mu\left(\left\Vert \delta_{x}P^{n}-\mu\right\Vert _{\mathrm{TV}}^{2}\right)^{1/2}\cdot\left\Vert f\right\Vert _{\mathrm{osc}},\label{eq:TV-CV-implies-osc}
\end{equation}
from the characterization $\left\Vert \nu\right\Vert _{\mathrm{TV}}=\sup\left\{ \nu\left(f\right):\left\Vert f\right\Vert _{\mathrm{osc}}\leqslant1\right\} $
for signed measures with $\nu\left(\mathsf{E}\right)=0$, and the
dominated convergence theorem implies convergence. When $P$ is reversible,
\cite[Proposition~23]{andrieu2022comparison_journal} establishes
that this implies that a WPI holds for $P^{2}$, therefore implying
the necessity of the WPI condition. A natural question is whether
$\mu$-irreducibility and aperiodicity, or convergence, always imply
the existence of a WPI for $T_{k}$ for some $k\in\mathbb{N}$. The
answer is negative and we give a counterexample.
\begin{defn}
We say that a Markov kernel $T$ on $\left(\E,\mathscr{E}\right)$
is \textit{$\nu$-irreducible} for a measure $\nu$ on $\left(\E,\mathscr{E}\right)$
if for any measurable set $A\in\mathscr{E}$ with $\nu\left(A\right)>0$,
we have 
\[
\sum_{n=0}^{\infty}\lambda^{n}\cdot T^{n}\left(x,A\right)>0,\quad\forall x\in\mathsf{E},
\]
for some (and hence all) $0<\lambda<1$.
\end{defn}

\begin{defn}
\label{def:lazy}For a $\mu$-invariant Markov kernel $T$, and $\epsilon\in\left(0,1\right)$,
we denote by $T_{\epsilon}$ the $\mu$-invariant kernel $T_{\epsilon}=\epsilon\cdot{\rm Id}+\left(1-\epsilon\right)\cdot T$.
\end{defn}

\begin{defn}
We say a Markov kernel $T$ has \textit{positive holding probabilities
almost everywhere}, or satisfies (\textbf{PH}), if
\[
\mu\left(\left\{ x\in\E:T\left(x,\left\{ x\right\} \right)=0\right\} \right)=0.
\]
\end{defn}

Throughout this section, we write $S:=\left(P+P^{*}\right)/2$ for
the additive reversibilization of $P$.

The following summarizes the most salient general results, i.e. without
assuming (\textbf{PH}), in this section:
\begin{thm}
Throughout, we write `convergent' for `$\left\Vert \cdot\right\Vert _{\mathrm{osc}}^{2}$-convergent'.
The following hold:
\begin{enumerate}
\item $P^{*}P$ satisfies a WPI $\Rightarrow$ $P$ convergent {[}Theorem~\ref{thm:WPI_F_bd}{]}
\item $P$ $\mu$-irreducible $\Rightarrow$ $P$ satisfies a WPI {[}Corollary~\ref{cor:WPI_from_irred}{]}
\item $P$ convergent $\Rightarrow$ $P^{k}$ satisfies a WPI for all $k\in\mathbb{N}$.
{[}Lemma~\ref{lem:osc-conv-RUPI}{]}
\item $P$ convergent $\Rightarrow$ $S$ convergent {[}Proposition~\ref{prop:P/S-convergence-P/S}{]}
\item \label{enu:wpi-pepe}The following are equivalent: {[}Proposition~\ref{prop:WPI-epsilon-P*P-WPI}{]}
\begin{enumerate}
\item $P$ satisfies a WPI.
\item For some $\epsilon\in\left(0,1\right)$, $P_{\epsilon}^{*}P_{\epsilon}$
satisfies a WPI.
\item For all $\epsilon\in\left(0,1\right)$, $P_{\epsilon}^{*}P_{\epsilon}$
satisfies a WPI.
\end{enumerate}
\item The following are equivalent: {[}Remark~\ref{rem:nonrev_WPI}{]}
\begin{enumerate}
\item $P$ satisfies a WPI
\item $P^{*}$ satisfies a WPI
\item $S$ satisfies a WPI
\end{enumerate}
\item $P$ satisfying a WPI does not imply $P$ is convergent {[}Example~\ref{exa:circle-walks}{]}.
\item $S$ convergent does not imply $P$ convergent {[}Example~\ref{exa:circle-walks}{]}.
\item $P$ convergent does not imply $\left(P^{*}\right)^{k}P^{k}$ satisfies
a WPI for some $k\in\mathbb{N}$ {[}Proposition~\ref{prop:preparatory-counterexample}{]}.
\end{enumerate}
\end{thm}

Item~\ref{enu:wpi-pepe} above can be connected with the notion of
\textit{variance bounding}: \cite[Theorem~6]{Roberts2008} shows that
the variance bounding property is equivalent to \textit{geometric}
$\ELL$-convergence of $P_{\epsilon}$ for any (equivalently, all)
$\epsilon\in\left(0,1\right)$, for reversible $P$. The situation
under (\textbf{PH}) is much simpler. The proof of Theorem~\ref{thm:hold-equiv-wpi-conv-pp}
can be found in Section~\ref{subsec:Holding-probabilities,-WPIs}.
\begin{thm}
\label{thm:hold-equiv-wpi-conv-pp}Assume $P$ satisfies (\textbf{PH}).
The following are equivalent:
\begin{enumerate}
\item $P$ satisfies a WPI;
\item $P^{*}P$ satisfies a WPI;
\item $P$ is convergent;
\item $PP^{*}$ satisfies a WPI;
\item $P^{*}$ is convergent.
\end{enumerate}
\end{thm}

We note that $P$ satisfies (\textbf{PH}) $\iff$ $P^{*}$ satisfies
(\textbf{PH}) $\iff$ $S$ satisfies (\textbf{PH}) by Lemma~\ref{lem:stick-probs-equal}.

These implications and equivalences are quite useful for showing several
derivative results. For example, we may deduce that reversible Markov
kernels $P$ satisfies $P$ convergent $\iff$ $P^{2}=P^{*}P$ satisfies
a WPI by Lemma~\ref{lem:osc-conv-RUPI} and Theorem~\ref{thm:WPI_F_bd}.
Similarly, it is straightforward to deduce that under (\textbf{PH}),
$S$ convergent $\iff$ $S$ satisfies a WPI $\iff$ $P$ satisfies
a WPI $\iff$ $P^{*}P$ satisfies a WPI $\iff$ $P$ convergent. This
can be contrasted with the general case where one only has $P$ convergent
$\Rightarrow$ $S$ convergent, and $S$ convergent $\Rightarrow$
for all $\epsilon\in\left(0,1\right)$, $P_{\epsilon}^{*}P_{\epsilon}$
satisfies a WPI; see Proposition~\ref{prop:S-conv-epsilon-P*P-WPI}. 

\subsubsection{Equivalence of $\left\Vert \cdot\right\Vert _{{\rm osc}}^{2}$\textit{\emph{--}}WPI
and RUPI\label{subsec:Equivalence-of--WPI}}

We will see that for a Markov operator $T$, a necessary and sufficient
condition for a $\left(\left\Vert \cdot\right\Vert _{\mathrm{osc}}^{2},\alpha\right)-$WPI
to hold is the abstract \textit{resolvent-uniform-positivity-improving}
(RUPI) property. This property appeared in \cite{gong2006spectral},
and in \cite{wang2014criteria} it was suggested that an equivalence
between the RUPI property and the existence of a WPI was already established
in an unpublished manuscript by L.~Wu. However, we have not been
able to access this manuscript, and so in this section we provide
a direct proof of this equivalence. This abstract RUPI condition will
provide the crucial link between irreducibility and existence of a
WPI in Corollary~\ref{cor:WPI_from_irred}; it has been established
in \cite[Corollary 4.5]{gong2006spectral} that $\mu$-irreducibility
implies the RUPI property.
\begin{defn}[UPI and RUPI]
 A $\mu$-invariant kernel $T$ is \textit{uniform-positivity-improving}
(UPI) if for each $\epsilon>0$,
\[
\inf\left\{ \left\langle \mathbf{1}_{A},T\mathbf{1}_{B}\right\rangle :\mu\left(A\right)\wedge\mu\left(B\right)\geqslant\epsilon\right\} >0.
\]
A Markov kernel $T$ is said to be \textit{resolvent-uniform-positivity-improving}
(RUPI) if for some (and hence all) $0<\lambda<1$, we have that the
resolvent
\[
R\left(\lambda,T\right):=\sum_{n=0}^{\infty}\lambda^{n}\cdot T^{n}=\left(\Id-\lambda\cdot T\right)^{-1}\,,
\]
is UPI.
\end{defn}

Although the RUPI condition is abstract, in order to establish that
a given kernel $T$ is RUPI, it is sufficient to show that a simple
irreducibility condition holds. Recall that we are assuming in this
section that $\mu$ is a probability measure, and that $T$ is $\mu$-invariant.

\begin{prop}[{\cite[Corollary 4.5]{gong2006spectral}}]
\label{prop:irred_RUPI} Suppose that the $\mu$-invariant kernel
$T$ is $\mu$-irreducible. Then $T$ is RUPI.
\end{prop}

The aim of this subsection is to establish the following abstract
result, which immediately implies the subsequent practical corollary.
\begin{thm}
\label{thm:WPI_iff_RUPI}Suppose that $T$ is a $\mu$--invariant
Markov kernel. Then $T$ satisfies a WPI if and only if $T$ is RUPI.
\end{thm}

\begin{proof}
This follows from Proposition~\ref{prop:RUPI-to-WPI} and Proposition~\ref{prop:wpi-to-rupi}
below.
\end{proof}
\begin{cor}
\label{cor:WPI_from_irred}Suppose the Markov kernel $T$ is $\mu$-irreducible.
Then $T$ possesses a WPI by Proposition~\ref{prop:irred_RUPI} and
Theorem~\ref{thm:WPI_iff_RUPI}.
\end{cor}

We follow \cite{wang2014criteria} and give an equivalent condition
for RUPI which will be convenient to work with.
\begin{lem}
\label{lem:RUPI_equiv}An equivalent condition for a Markov kernel
$T$ to be RUPI is the following: for any $\epsilon>0$, there exists
$m\in\mathbb{N}$ such that 
\begin{equation}
\inf\left\{ \left\langle \mathbf{1}_{A},\sum_{n=0}^{m}T^{n}\mathbf{1}_{B}\right\rangle :\mu\left(A\right)\wedge\mu\left(B\right)\geqslant\epsilon\right\} >0.\label{eq:RUPI-equivalence}
\end{equation}
\end{lem}

\begin{proof}
The condition in Lemma~\ref{lem:RUPI_equiv} directly implies RUPI.
To see this, take $\lambda\in\left(0,1\right)$, $\epsilon>0$ and
let $m\in\mathbb{N}$ such that $\inf\left\{ \left\langle \mathbf{1}_{A},\sum_{n=0}^{m}T^{n}\mathbf{1}_{B}\right\rangle :\mu\left(A\right)\wedge\mu\left(B\right)\geqslant\epsilon\right\} >0$,
which exists by assumption. Write

\begin{align*}
\left\langle \mathbf{1}_{A},R\left(\lambda,T\right)\mathbf{1}_{B}\right\rangle  & =\left\langle \mathbf{1}_{A},\sum_{n=0}^{\infty}\lambda^{n}\cdot T^{n}\mathbf{1}_{B}\right\rangle \\
 & \geqslant\left\langle \mathbf{1}_{A},\sum_{n=0}^{m}\lambda^{n}\cdot T^{n}\mathbf{1}_{B}\right\rangle \\
 & \geqslant\lambda^{m}\cdot\left\langle \mathbf{1}_{A},\sum_{n=0}^{m}T^{n}\mathbf{1}_{B}\right\rangle ,
\end{align*}
to deduce that

\begin{multline*}
\inf\{\left\langle \mathbf{1}_{A},R\left(\lambda,T\right)\mathbf{1}_{B}\right\rangle :\mu\left(A\right)\wedge\mu\left(B\right)\geqslant\epsilon\}\\
\geqslant\lambda^{m}\cdot\inf\left\{ \left\langle \mathbf{1}_{A},\sum_{n=0}^{m}T^{n}\mathbf{1}_{B}\right\rangle :\mu\left(A\right)\wedge\mu\left(B\right)\geqslant\epsilon\right\} >0,
\end{multline*}
from which the RUPI condition follows.

Conversely, suppose that $T$ is RUPI, fix $\lambda\in\left(0,1\right)$
and assume that for some $\epsilon>0$, (\ref{eq:RUPI-equivalence})
does not hold for any $m\in\mathbb{N}$. We show that this leads to
a contradiction. By the RUPI assumption we have that 
\[
\delta:=\inf\left\{ \left\langle \mathbf{1}_{A},\sum_{n=0}^{\infty}\lambda^{n}\cdot T^{n}\mathbf{1}_{B}\right\rangle :\mu\left(A\right)\wedge\mu\left(B\right)\geqslant\epsilon\right\} >0.
\]
Choose $m\in\mathbb{N}$ large enough so that
\[
\sum_{n=m+1}^{\infty}\lambda^{n}<\frac{\delta}{2}.
\]
Since we have assumed that (\ref{eq:RUPI-equivalence}) is violated
for $\epsilon>0$ and $m\in\mathbb{N}$ as chosen above, there exists
a sequence $\left\{ \left(A_{j},B_{j}\right)\right\} _{j=1}^{\infty}$
of sets all with mass at least $\epsilon$ such that $\left\langle \mathbf{1}_{A_{j}},\sum_{n=0}^{m}T^{n}\mathbf{1}_{B_{j}}\right\rangle \to0$,
therefore implying for any $j\in\mathbb{N}$,
\begin{align*}
\delta\leqslant\left\langle \mathbf{1}_{A_{j}},\sum_{n=0}^{\infty}\lambda^{n}T^{n}\mathbf{1}_{B_{j}}\right\rangle  & =\left\langle \mathbf{1}_{A_{j}},\sum_{n=0}^{m}\lambda^{n}T^{n}\mathbf{1}_{B_{j}}\right\rangle +\left\langle \mathbf{1}_{A_{j}},\sum_{n=m+1}^{\infty}\lambda^{n}T^{n}\mathbf{1}_{B_{j}}\right\rangle \\
 & \leqslant\left\langle \mathbf{1}_{A_{j}},\sum_{n=0}^{m}T^{n}\mathbf{1}_{B_{j}}\right\rangle +\sum_{n=m+1}^{\infty}\lambda^{n}\\
 & \leqslant\left\langle \mathbf{1}_{A_{j}},\sum_{n=0}^{m}T^{n}\mathbf{1}_{B_{j}}\right\rangle +\frac{\delta}{2}\\
 & \overset{j\to\infty}{\rightarrow}\frac{\delta}{2}\,,
\end{align*}
therefore leading to a contradiction. The conclusion follows.
\end{proof}
We first establish that for reversible kernels, RUPI implies a WPI
for the resolvent.
\begin{lem}
\label{prop:RUPI_resolvent}Suppose that a $\mu$-reversible Markov
kernel $T$ is RUPI. Then for any $\lambda\in\left(0,1\right)$, the
resolvent Markov kernel $S_{\lambda}:=\left(1-\lambda\right)\cdot R\left(\lambda,T\right)$
is $\mu$-reversible and has the following property: for any $0<\epsilon\leqslant\frac{1}{2}$,
for the quantity appearing in the WCP,
\[
\inf_{A:\epsilon\leqslant\mu\left(A\right)\leqslant\frac{1}{2}}\frac{\mathcal{E}\left(S_{\lambda},\mathbf{1}_{A}\right)}{\mu\left(A\right)}>0.
\]
Thus by Theorem~\ref{thm:positive-CP-implies-optim-WPI}, $S_{\lambda}$
satisfies a WPI.
\end{lem}

\begin{proof}
Fix $0<\epsilon\leqslant\frac{1}{2}$ and $\lambda\in\left(0,1\right)$.
By the RUPI condition, we know that $\inf\left\{ \left\langle \mathbf{1}_{A},S_{\lambda}\mathbf{1}_{B}\right\rangle :\mu\left(A\right)\wedge\mu\left(B\right)\geqslant\epsilon\right\} >0$.
In particular, if $A$ is such that $\epsilon\leqslant\mu\left(A\right)\leqslant\frac{1}{2}$,
we must have that both $\mu\left(A\right)\geqslant\epsilon$ and $\mu\left(A^{\complement}\right)\geqslant\frac{1}{2}\geqslant\epsilon$.
Thus since 
\[
\mathcal{E}\left(S_{\lambda},\mathbf{1}_{A}\right)=\left\langle \mathbf{1}_{A},S_{\lambda}\mathbf{1}_{A^{\complement}}\right\rangle ,
\]
by Lemma~\ref{lem:dirichlet-form-indicator} we must have that
\[
\left\langle \mathbf{1}_{A},S_{\lambda}\mathbf{1}_{A^{\complement}}\right\rangle \geqslant\delta>0,
\]
for some $\delta>0$, whenever $\epsilon<\mu\left(A\right)\leqslant\frac{1}{2}$.
\end{proof}
We now establish one direction of Theorem~\ref{thm:WPI_iff_RUPI}
through a sequence of lemmas: we first consider the case when $T$
is reversible, and then deduce the case for general $T$; see Remark~\ref{rem:nonrev_WPI}.
\begin{lem}
\label{lem:T_rev_RUPI}Suppose $T$ is a $\mu$-reversible Markov
kernel that is RUPI. Then $T$ satisfies a WPI.
\end{lem}

\begin{proof}
Since $T$ is RUPI, we have established above in Lemma~\ref{prop:RUPI_resolvent}
that the resolvent $S_{\lambda}:=\left(1-\lambda\right)\cdot R\left(\lambda,T\right)$
satisfies a WPI for any $\lambda\in\left(0,1\right)$. In other words,
we can find some $\alpha_{\lambda}:\left(0,\infty\right)\to\left[0,\infty\right)$
such that for any $f\in\mathrm{L}_{0}^{2}\left(\mu\right)$ and $r>0$,
it holds that
\[
\left\Vert f\right\Vert _{2}^{2}\leqslant\alpha_{\lambda}\left(r\right)\cdot\left\langle \left(\mathrm{Id}-S_{\lambda}\right)f,f\right\rangle +r\cdot\left\Vert f\right\Vert _{\mathrm{osc}}^{2}.
\]
Now, given a function $g\in\mathrm{L}_{0}^{2}\left(\mu\right)$, define
$f:=\left(1-\lambda\right)^{-1}\cdot\left(\mathrm{Id}-\lambda\cdot T\right)\cdot g$,
so that $g=\left(1-\lambda\right)\cdot\left(\mathrm{Id}-\lambda\cdot T\right)^{-1}\cdot f$.
Note that since $0<\lambda<1$, the operator $\left(\mathrm{Id}-\lambda\cdot T\right)$
is invertible.

Now since $g\in\mathrm{L}_{0}^{2}\left(\mu\right)$, we have that
$f\in\mathrm{L}_{0}^{2}\left(\mu\right)$ (by, for instance, considering
the power series representation of $R\left(\lambda,T\right)$). Furthermore,
we have that 
\begin{align*}
\left\Vert g\right\Vert _{2}^{2} & =\left(1-\lambda\right)^{2}\cdot\left\Vert \left(\mathrm{Id}-\lambda\cdot T\right)^{-1}f\right\Vert _{2}^{2}\\
 & \leqslant\left(1-\lambda\right)^{2}\cdot\left\Vert \left(\mathrm{Id}-\lambda\cdot T\right)^{-1}\right\Vert _{\mathrm{op}}\cdot\left\Vert f\right\Vert _{2}^{2}\\
 & \leqslant\left\Vert f\right\Vert _{2}^{2}
\end{align*}
since the operator norm can be bounded as $\left\Vert \left(\mathrm{Id}-\lambda\cdot T\right)^{-1}\right\Vert _{\mathrm{op}}\leqslant\frac{1}{1-\lambda}$,
by standard arguments.

Thus, we have that
\begin{align*}
\left\Vert g\right\Vert _{2}^{2}\leqslant & \left\Vert f\right\Vert _{2}^{2}\\
\leqslant & \alpha_{\lambda}\left(r\right)\cdot\left\langle \left(\mathrm{Id}-S_{\lambda}\right)f,f\right\rangle +r\cdot\left\Vert f\right\Vert _{\mathrm{osc}}^{2}\\
= & \alpha_{\lambda}\left(r\right)\cdot\left\langle \left(\mathrm{Id}-\left(1-\lambda\right)\cdot\left(\mathrm{Id}-\lambda\cdot T\right)^{-1}\right)f,f\right\rangle +r\cdot\left\Vert f\right\Vert _{\mathrm{osc}}^{2}\\
= & \alpha_{\lambda}\left(r\right)\cdot\left\langle \left(1-\lambda\right)^{-1}\cdot\left(\mathrm{Id}-\lambda\cdot T\right)g-g,\left(1-\lambda\right)^{-1}\cdot\left(\mathrm{Id}-\lambda\cdot T\right)g\right\rangle +r\cdot\left\Vert f\right\Vert _{\mathrm{osc}}^{2}\\
= & \alpha_{\lambda}\left(r\right)\cdot\left\langle \lambda\cdot\left(1-\lambda\right)^{-1}\cdot\left(\mathrm{Id}-T\right)g,g+\lambda\cdot\left(1-\lambda\right)^{-1}\cdot\left(\mathrm{Id}-T\right)g\right\rangle +r\cdot\left\Vert f\right\Vert _{\mathrm{osc}}^{2}\\
= & \alpha_{\lambda}\left(r\right)\cdot\left\{ \left\langle \lambda\cdot\left(1-\lambda\right)^{-1}\cdot\left(\mathrm{Id}-T\right)g,g\right\rangle +\lambda^{2}\cdot\left(1-\lambda\right)^{-2}\cdot\left\Vert \left(\mathrm{Id}-T\right)g\right\Vert _{2}^{2}\right\} \\
 & \hspace{9cm}+r\cdot\left\Vert f\right\Vert _{\mathrm{osc}}^{2}.
\end{align*}
Note that
\begin{align*}
\left\Vert \left(\mathrm{Id}-T\right)g\right\Vert _{2}^{2} & =\left\langle \left(\mathrm{Id}-T\right)g,\left(\mathrm{Id}-T\right)g\right\rangle \\
 & =\left\langle \left(\mathrm{Id}-T\right)g,g\right\rangle -\left\langle \left(\mathrm{Id}-T\right)g,Tg\right\rangle .
\end{align*}
It is enough to bound this final term by
\[
-\left\langle \left(\mathrm{Id}-T\right)g,Tg\right\rangle \leqslant\left\langle \left(\mathrm{Id}-T\right)g,g\right\rangle .
\]
To see why this inequality is true, note that it is equivalent to
\begin{align*}
0 & \leqslant\left\langle \left(\mathrm{Id}-T\right)g,\left(\mathrm{Id}+T\right)g\right\rangle \\
 & =\left\langle \left(\mathrm{Id}+T\right)\cdot\left(\mathrm{Id}-T\right)g,g\right\rangle \\
 & =\left\langle \left(\mathrm{Id}-T^{2}\right)g,g\right\rangle 
\end{align*}
where we have made use of reversibility of $T$, and we certainly
have that $0\leqslant\left\langle \left(\mathrm{Id}-T^{*}T\right)g,g\right\rangle =\left\langle \left(\mathrm{Id}-T^{2}\right)g,g\right\rangle $. 

Overall, this gives us that 
\begin{align*}
\left\Vert g\right\Vert _{2}^{2} & \leqslant\alpha_{\lambda}\left(r\right)\cdot\left\{ \left\langle \lambda\cdot\left(1-\lambda\right)^{-1}\cdot\left(\mathrm{Id}-T\right)g,g\right\rangle +\lambda^{2}\cdot\left(1-\lambda\right)^{-2}\cdot\left\Vert \left(\mathrm{Id}-T\right)g\right\Vert _{2}^{2}\right\} \\
 & \quad+r\cdot\left\Vert \left(1-\lambda\right)^{-1}\cdot\left(\mathrm{Id}-\lambda\cdot T\right)g\right\Vert _{\mathrm{osc}}^{2}\\
 & \leqslant\alpha_{\lambda}\left(r\right)\cdot\left\{ \lambda\cdot\left(1-\lambda\right)^{-1}+2\cdot\lambda^{2}\cdot\left(1-\lambda\right)^{-2}\right\} \cdot\left\langle \left(\mathrm{Id}-T\right)g,g\right\rangle \\
 & \quad+r\cdot\left(1+\lambda\right)^{2}\cdot\left(1-\lambda\right)^{-2}\cdot\left\Vert g\right\Vert _{\mathrm{osc}}^{2},
\end{align*}
and by reparameterizing with $r'=r\cdot\left(1+\lambda\right)^{2}\cdot\left(1-\lambda\right)^{-2}$,
one obtains a standard WPI for $T$.
\end{proof}
\begin{prop}
\label{prop:RUPI-to-WPI}Suppose a $\mu$-invariant Markov kernel
$T$ is RUPI. Then $T$ satisfies a WPI.
\end{prop}

\begin{proof}
It suffices to show that $\left(T+T^{*}\right)/2$ is RUPI, as then
by Lemma~\ref{lem:T_rev_RUPI}, $\left(T+T^{*}\right)/2$ possesses
a WPI, which is equivalent to $T$ possessing a WPI (see Remark~\ref{rem:nonrev_WPI}).
Since $T$ is RUPI, for any $\epsilon>0$, we can find some $\delta>0$
and $N\in\mathbb{N}$ such that whenever $\mu\left(A\right)\wedge\mu\left(B\right)\geqslant\epsilon$,
\begin{equation}
\left\langle \mathbf{1}_{A},\sum_{n=0}^{N}T^{n}\mathbf{1}_{B}\right\rangle \geqslant\delta>0.\label{eq:lem_T_RUPI}
\end{equation}
We now wish to obtain such a statement for the kernel $\left(T+T^{*}\right)/2$.
To this end, fix $\epsilon>0$, and consider 
\[
\left\langle \mathbf{1}_{A},\sum_{n=0}^{N}\left(\frac{T+T^{*}}{2}\right)^{n}\mathbf{1}_{B}\right\rangle =\left\langle \mathbf{1}_{A},\sum_{n=0}^{N}2^{-n}\cdot T^{n}\mathbf{1}_{B}\right\rangle +\left\langle \mathbf{1}_{A},R\mathbf{1}_{B}\right\rangle ,
\]
where $R$ is a sum of operators of the form $c\cdot T^{a_{1}}\cdot\left(T^{*}\right)^{b_{1}}\cdot\dots\cdot T^{a_{r}}\cdot\left(T^{*}\right)^{b_{r}}$
for some $r\in\mathbb{N}$, $a_{i},b_{i}\in\mathbb{N}_{0}$ for all
$i=1,\dots,r$ and $c\geqslant0$. Thus since $T$ and $T^{*}$ are
Markov kernels, we have that $\left\langle \mathbf{1}_{A},R\mathbf{1}_{B}\right\rangle \geqslant0.$
So we can continue and have, for any sets with $\mu\left(A\right)\wedge\mu\left(B\right)\geqslant\epsilon$,
\begin{align*}
\left\langle \mathbf{1}_{A},\sum_{n=0}^{N}\left(\frac{T+T^{*}}{2}\right)^{n}\mathbf{1}_{B}\right\rangle  & \geqslant\left\langle \mathbf{1}_{A},\sum_{n=0}^{N}2^{-n}\cdot T^{n}\mathbf{1}_{B}\right\rangle \\
 & \geqslant2^{-N}\cdot\left\langle \mathbf{1}_{A},\sum_{n=0}^{N}T^{n}\mathbf{1}_{B}\right\rangle \\
 & \geqslant2^{-N}\cdot\delta>0,
\end{align*}
since each summand is positive, and we have used the fact that $T$
is RUPI (\ref{eq:lem_T_RUPI}).
\end{proof}
We now prove the other direction. The following is a useful implication
of convergence, meaning (\ref{eq:P-Phi-gamma-convergent}) holds,
which we will rely on below and also in Section~\ref{subsec:Holding-probabilities,-WPIs}.
\begin{lem}
\label{lem:osc-conv-RUPI}Assume $T$ is convergent. Then for any
$\epsilon>0$, there exists $n_{0}\in\mathbb{N}$ such that for any
$N\geqslant n_{0}$,
\[
\inf\left\{ \left\langle \mathbf{1}_{A},T^{N}\mathbf{1}_{B}\right\rangle :\mu\left(A\right)\wedge\mu\left(B\right)\geqslant\epsilon\right\} >0.
\]
In particular, for all $k\in\mathbb{N}$, $T^{k}$ is RUPI and $T^{k}$
satisfies a WPI.
\end{lem}

\begin{proof}
Let $\epsilon>0$ be arbitrary. Since $T$ is convergent, we may take
$n_{0}\in\mathbb{N}$ large enough such that $\left\Vert T^{N}f\right\Vert _{2}\leqslant\left\Vert f\right\Vert _{{\rm osc}}\cdot\epsilon^{2}/2$
for all $N\geqslant n_{0}$. Let $A,B\in\mathscr{E}$ be such that
$\mu\left(A\right)\wedge\mu\left(B\right)\geqslant\epsilon$. For
any $N\geqslant n_{0}$ we have
\begin{align*}
\left\langle \mathbf{1}_{A},T^{N}\mathbf{1}_{B}\right\rangle  & =\left\langle \mathbf{1}_{A},\left(T^{N}-\mu\right)\mathbf{1}_{B}\right\rangle +\left\langle \mathbf{1}_{A},\mu\mathbf{1}_{B}\right\rangle \\
 & =\left\langle \mathbf{1}_{A},\left(T^{N}-\mu\right)\mathbf{1}_{B}\right\rangle +\mu\left(A\right)\cdot\mu\left(B\right).
\end{align*}
Let $f={\bf 1}_{B}-\mu\left(B\right)$, so that by Cauchy--Schwarz,
we have that
\[
\left|\left\langle \mathbf{1}_{A},\left(T^{N}-\mu\right)\mathbf{1}_{B}\right\rangle \right|=\left|\left\langle \mathbf{1}_{A},T^{N}f\right\rangle \right|\leqslant\mu\left(A\right)^{1/2}\cdot\left\Vert f\right\Vert _{{\rm osc}}\cdot\epsilon^{2}/2\leqslant\epsilon^{2}/2,
\]
and therefore 
\[
\left\langle \mathbf{1}_{A},T^{N}\mathbf{1}_{B}\right\rangle \geqslant-\epsilon^{2}/2+\epsilon^{2}=\epsilon^{2}/2>0,
\]
from which we can conclude. Now let $k\in\left\{ 1,2,\ldots\right\} $
be arbitrary. Since we may choose $N$ to be a multiple of $k$, it
follows from Lemma~\ref{lem:RUPI_equiv} that $T^{k}$ is RUPI. Hence,
by Proposition~\ref{prop:RUPI-to-WPI}, $T^{k}$ satisfies a WPI.
\end{proof}
Lemma~\ref{lem:P-to-PP-WPI}\ref{enu:lemma_56b} is an extension
of the argument referenced by \cite[Remark~2.16]{montenegro2006mathematical}.
\begin{lem}
\label{lem:P-to-PP-WPI}Assume $P$ satisfies\textbf{ (PH}). Then
\begin{enumerate}
\item The following (chained) WPI holds: for $s>0$, $f\in\ELL_{0}\left(\mu\right)$,
\begin{equation}
\mathcal{\calE}\left(P,f\right)\leqslant s\cdot\calE\left(P^{*}P,f\right)+\beta\left(s\right)\cdot\left\Vert f\right\Vert _{{\rm osc}}^{2},\label{eq:WPI_P-P*P}
\end{equation}
for the function 
\begin{equation}
\beta\left(s\right):=\mu\otimes\left(\frac{P+P^{*}}{2}\right)\left(A\left(s\right)^{\complement}\cap\left\{ X\neq Y\right\} \right),\label{eq:wpi-beta-P-P*P}
\end{equation}
where $A\left(s\right):=\left\{ \left(x,y\right)\in\E\times\E:s\cdot\delta\left(x,y\right)>1\right\} $,
where $\delta\left(x,y\right):=2\cdot\left(\epsilon\left(x\right)\wedge\epsilon\left(y\right)\right)$,
and $\epsilon\left(x\right):=P\left(x,\left\{ x\right\} \right)$.
\item \label{enu:lemma_56b}If $P\left(x,\left\{ x\right\} \right)\geqslant\varepsilon>0$
for $\mu$-almost all $x\in E$, then for $f\in\ELL_{0}\left(\mu\right)$,
$\mathcal{E}\left(P^{*}P,f\right)\geqslant2\cdot\varepsilon\cdot\mathcal{E}\left(P,f\right)$.
\end{enumerate}
In particular, if $P$ satisfies a WPI then $P^{*}P$ satisfies a
WPI.

\end{lem}

\begin{proof}
First, we represent the kernel $P$ as 
\[
P\left(x,\dif y\right)=\epsilon\left(x\right)\cdot\delta_{x}\left(\dif y\right)+R\left(x,\dif y\right),
\]
for a function $\epsilon:\E\to\left[0,1\right]$ given by $\epsilon\left(x\right)=P\left(x,\left\{ x\right\} \right)$
and some (possibly non-Markov) kernel $R$. Note that since $f\mapsto Pf$
and the pointwise product $f\mapsto\epsilon\cdot f$ are bounded linear
operators on $\ELL\left(\mu\right)$, it follows that $f\mapsto Rf$
is also a bounded linear operator on $\ELL\left(\mu\right)$, and
hence there exists a bounded linear adjoint operator $R^{*}$.

Thus we have the representation
\[
P^{*}\left(x,\dif y\right)=\epsilon\left(x\right)\cdot\delta_{x}\left(\dif y\right)+R^{*}\left(x,\dif y\right).
\]
In order to establish the desired WPI, we seek to apply \cite[Theorem~38]{andrieu2022comparison_journal},
namely we seek to establish for any measurable set $A\subset\E$,
and $x\in\E$,
\[
P^{*}P\left(x,A\setminus\left\{ x\right\} \right)\geqslant\int_{A\backslash\left\{ x\right\} }2\cdot\min\left\{ \epsilon\left(x\right),\epsilon\left(y\right)\right\} \cdot\left(\frac{P+P^{*}}{2}\right)\left(x,\dif y\right),
\]
Note that this is sufficient since $\calE\left(P,f\right)=\calE\left(\left(P+P^{*}\right)/2,f\right)$.

Through direct calculation, we find that
\begin{align*}
P^{*}P\left(x,A\setminus\left\{ x\right\} \right) & =\int\left(\epsilon\left(x\right)R\left(x,\mathrm{d}y\right)+R^{*}\left(x,\mathrm{d}y\right)\epsilon\left(y\right)\right)\mathbf{1}_{A\setminus\left\{ x\right\} }\left(y\right)\\
 & \hspace{4cm}+\int R^{*}\left(x,\mathrm{d}z\right)R\left(z,\mathrm{d}y\right)\mathbf{1}_{A\setminus\left\{ x\right\} }\left(y\right)\\
 & \geqslant\int\left(\epsilon\left(x\right)R\left(x,\mathrm{d}y\right)+R^{*}\left(x,\mathrm{d}y\right)\epsilon\left(y\right)\right)\mathbf{1}_{A\setminus\left\{ x\right\} }\left(y\right)\\
 & \geqslant\int\min\left\{ \epsilon\left(x\right),\epsilon\left(y\right)\right\} \cdot\left[R\left(x,\mathrm{d}y\right)+R^{*}\left(x,\mathrm{d}y\right)\right]\\
 & =2\cdot\int\min\left\{ \epsilon\left(x\right),\epsilon\left(y\right)\right\} \cdot\left(\frac{P+P^{*}}{2}\right)\left(x,\mathrm{d}y\right).
\end{align*}
Thus by \cite[Theorem~36]{andrieu2022comparison_journal}, we have
that (\ref{eq:WPI_P-P*P}) holds with $\beta\left(s\right)$ as in
(\ref{eq:wpi-beta-P-P*P}). Note that $\beta\left(s\right)\to0$ as
$s\to\infty$ since we have assumed that $\int\mu\left(\dif x\right)\cdot\mathbf{1}\left[\left\{ x:P\left(x,\left\{ x\right\} \right)=0\right\} \right]=0$.
The last statement is immediate.
\end{proof}
\begin{prop}
\label{prop:wpi-to-rupi}Let $T$ be a $\mu$-invariant Markov kernel
satisfying an $\alpha$\emph{--}WPI for some $\alpha:\left(0,\infty\right)\to\left[0,\infty\right)$.
Then $T$ is RUPI.
\end{prop}

\begin{proof}
Consider the Markov operator $T_{1/2}:=\frac{1}{2}\left(\Id+T\right)$,
which satisfies\\
 $T_{1/2}\left(x,\left\{ x\right\} \right)\geqslant\frac{1}{2}$ by
construction. Note that 
\begin{align*}
\mathcal{E}\left(T_{1/2},f\right) & =\left\langle \left(\mathrm{Id}-\frac{1}{2}\cdot\left(\mathrm{Id}+T\right)\right)f,f\right\rangle \\
 & =\frac{1}{2}\cdot\left\langle \left(\mathrm{Id}-T\right)f,f\right\rangle \\
 & =\frac{1}{2}\cdot\mathcal{E}\left(T,f\right).
\end{align*}
Therefore, since $T$ satisfies a $\alpha-$WPI, we have that $T_{1/2}$
satisfies a $\left(\left\Vert \cdot\right\Vert _{\mathrm{osc}}^{2},2\cdot\alpha\right)$\textit{\emph{--}}WPI:
\begin{align*}
\left\Vert f\right\Vert _{2}^{2} & \leqslant\alpha\left(r\right)\cdot\mathcal{E}\left(T,f\right)+r\cdot\left\Vert f\right\Vert _{{\rm osc}}^{2}\\
 & =2\cdot\alpha\left(r\right)\cdot\mathcal{E}\left(T_{1/2},f\right)+r\cdot\left\Vert f\right\Vert _{{\rm osc}}^{2}.
\end{align*}
Since $\mathrm{ess_{\mu}\inf}_{x}T_{1/2}\left(x,\left\{ x\right\} \right)\geqslant\frac{1}{2}$,
by Lemma~\ref{lem:P-to-PP-WPI} we have the inequality $\mathcal{E}\left(T_{1/2},f\right)\leqslant\mathcal{E}\left(T_{1/2}^{*}T_{1/2},f\right)$,
so we deduce a $(2\cdot\alpha)$\textit{\emph{--}}WPI for $T_{1/2}^{*}T_{1/2}$.
Hence, $T_{1/2}$ is convergent by Theorem~\ref{thm:WPI_F_bd}.

We will now verify the condition for RUPI in Lemma~\ref{lem:RUPI_equiv}.
Let $\epsilon\in\left(0,1\right)$ be arbitrary. Since $T_{1/2}$
is convergent, Lemma~\ref{lem:osc-conv-RUPI} implies that there
exists $N\in\mathbb{N}$ such that 
\[
\delta=\inf\left\{ \left\langle \mathbf{1}_{A},T_{1/2}^{N}\mathbf{1}_{B}\right\rangle :\mu\left(A\right)\wedge\mu\left(B\right)\geqslant\epsilon\right\} >0.
\]
Now, 
\[
T_{1/2}^{N}=\left(\frac{\Id+T}{2}\right)^{N}=\frac{1}{2^{N}}\cdot\sum_{k=0}^{N}a_{k}\cdot T^{k},
\]
for binomial coefficients $\{a_{i}\}$. Since $\sum_{i=0}^{N}a_{i}=2^{N}$
and ${\bf 1}_{A}$, ${\bf 1}_{B}$ are non-negative, we have 
\[
\left\langle \mathbf{1}_{A},\sum_{k=0}^{N}T^{k}\mathbf{1}_{B}\right\rangle \geqslant\left\langle \mathbf{1}_{A},\left(\frac{\Id+T}{2}\right)^{N}\mathbf{1}_{B}\right\rangle ,\qquad A,B\in\mathscr{E},
\]
and this implies that 
\[
\inf\left\{ \left\langle \mathbf{1}_{A},\sum_{k=0}^{N}T^{k}\mathbf{1}_{B}\right\rangle :\mu\left(A\right)\wedge\mu\left(B\right)\geqslant\epsilon\right\} \geqslant\delta>0,
\]
so $T$ is RUPI.
\end{proof}

\subsubsection{\label{subsec:Holding-probabilities,-WPIs}WPIs and convergence under
(PH)}

Recall from Theorem~\ref{thm:WPI_F_bd} that if $P^{*}P$ satisfies
a WPI, then we know that $P$ is convergent: $\left\Vert P^{n}f\right\Vert _{2}^{2}\leqslant\gamma\left(n\right)\cdot\left\Vert f\right\Vert _{{\rm osc}}^{2}$,
where $\gamma\left(n\right)\to0$. It is important to note that $P$
(rather than $P^{*}P$) possessing a WPI does not necessarily imply
that $\left\Vert P^{n}f\right\Vert _{2}\to0$ for bounded $f$. Indeed,
a reversible, periodic Markov kernel may satisfy a WPI, yet $\left\Vert P^{n}f\right\Vert _{2}$
cannot converge to $0$ for all bounded functions. Thus the goal of
this subsection is to give conditions which allow one to connect a
WPI for $P$ with convergence of $P$.
\begin{rem}
When $P$ is reversible, it is possible to deduce the existence of
a WPI for $P^{2}=P^{*}P$ from a WPI for $P$, provided that one has
some additional control on the left spectral gap; see \cite[Section 2.2.1]{andrieu2022comparison_journal}.
In turn, the existence of a WPI for $P$ can often be straightforwardly
deduced from Corollary~\ref{cor:WPI_from_irred} by establishing
irreducibility of $P$. 
\end{rem}

\begin{prop}
\label{prop:P/S-convergence-P/S}If $P$ is convergent, then $S$
is convergent.
\end{prop}

\begin{proof}
If $P$ is convergent, then Lemma~\ref{lem:osc-conv-RUPI} implies
that $P^{2}$ is RUPI and satisfies a WPI. We may then deduce that
$S^{2}$ satisfies a WPI because for any $f\in{\rm L}_{0}^{2}\left(\mu\right)$,
\begin{eqnarray*}
\mathcal{E}\left(S^{2},f\right) & = & \frac{1}{4}\cdot\left\{ \mathcal{E}\left(P^{2},f\right)+\mathcal{E}\left(\left(P^{*}\right)^{2},f\right)+\mathcal{E}\left(PP^{*},f\right)+\mathcal{E}\left(P^{*}P,f\right)\right\} \\
 & \geqslant & \frac{1}{4}\cdot\mathcal{E}\left(P^{2},f\right).
\end{eqnarray*}
It follows that $S$ is convergent by Theorem~\ref{thm:WPI_F_bd}.
\end{proof}
The converse does not hold in general, as seen by the following
example.
\begin{example}[Walks on the circle]
\label{exa:circle-walks}For $x,y\in\mathsf{E}=\left\{ 1,\ldots,m\right\} $
for $m$ odd, and let $P\left(x,y\right)={\bf 1}_{\left\{ 1,\ldots,m-1\right\} }\left(x\right)\cdot{\bf 1}_{\left\{ x+1\right\} }\left(y\right)+{\bf 1}_{\left\{ m\right\} }\left(x\right)\cdot{\bf 1}_{\left\{ 1\right\} }\left(y\right)$
so that $P^{*}\left(x,y\right)={\bf 1}_{\left\{ 2,\ldots,m\right\} }\left(x\right)\cdot{\bf 1}_{\left\{ x-1\right\} }\left(y\right)+{\bf 1}_{\left\{ 1\right\} }\left(x\right)\cdot{\bf 1}_{\left\{ m\right\} }\left(y\right)$.
Then the Markov chain associated with $P$ is deterministic and one
can deduce that $P$ is not convergent. On the other hand, $S=\left(P+P^{*}\right)/2$
encodes a random walk on $\left\{ 1,\ldots,m\right\} $ and is convergent.
\end{example}

\begin{lem}
\label{lem:stick-probs-equal}For a $\mu$-invariant kernel $P$,
$P\left(x,\left\{ x\right\} \right)=P^{*}\left(x,\left\{ x\right\} \right)$
for $\mu$-almost all $x$.
\end{lem}

\begin{proof}
Let $D=\left\{ \left(x,y\right)\in\mathsf{E}^{2}:x=y\right\} $, $s\left(x\right):=P\left(x,\left\{ x\right\} \right)$
and $s^{*}\left(x\right):=P^{*}\left(x,\left\{ x\right\} \right)$
for $x\in\mathsf{E}$. For any $B\in\mathcal{E}$, we have 
\[
\mu\left({\bf 1}_{B}\cdot s\right)=\mu\otimes P\left(D\cap\left(B\times B\right)\right)=\mu\otimes P^{*}\left(D\cap\left(B\times B\right)\right)=\mu\left({\bf 1}_{B}\cdot s^{*}\right),
\]
and so taking in turn $B=\left\{ x\in\E:s\left(x\right)>s^{*}\left(x\right)\right\} $
and $B=\left\{ x\in\E:s\left(x\right)<s^{*}\left(x\right)\right\} $,
we deduce that
\[
\mu\left(\left(s-s^{*}\right)^{+}\right)=0=\mu\left(\left(s-s^{*}\right)^{-}\right),
\]
and hence that $s=s^{*}$ $\mu$-almost everywhere.
\end{proof}
\begin{proof}[Proof of Theorem~\ref{thm:hold-equiv-wpi-conv-pp}]
(b. $\Rightarrow$ c.) follows from Theorem~\ref{thm:WPI_F_bd},
and (c. $\Rightarrow$ a.) follows from Lemma~\ref{lem:osc-conv-RUPI}.
We now show (a. $\Rightarrow$ b.). We apply Lemma~\ref{lem:P-to-PP-WPI},
and hence by chaining WPIs \cite[Theorem~33]{andrieu2022comparison_journal},
we conclude that $P^{*}P$ satisfies an WPI.

We now show that the cycle (a. $\Rightarrow$ d. $\Rightarrow$ e.
$\Rightarrow$ a.) can also be deduced. Observe that $P^{*}\left(x,\left\{ x\right\} \right)=P\left(x,\left\{ x\right\} \right)$
by Lemma~\ref{lem:stick-probs-equal}, and $P$ satisfying a WPI
is equivalent to $P^{*}$ satisfying a WPI, since $\mathcal{E}\left(P,f\right)=\mathcal{E}\left(P^{*},f\right)$.
Because $\left(P^{*}\right)^{*}P^{*}=PP^{*}$, we have that (a. $\Rightarrow$
d.) is equivalent to (a. $\Rightarrow$ b.) and (d. $\Rightarrow$
e.) is equivalent to (b. $\Rightarrow$ c.) and (e. $\Rightarrow$
a.) is equivalent to (c. $\Rightarrow$ a.).
\end{proof}

\subsubsection{WPIs and convergence in general}
\begin{prop}
\label{prop:WPI-epsilon-P*P-WPI}$P$ satisfies a WPI if and only
if for some (and hence any) $\epsilon\in\left(0,1\right)$, $P_{\epsilon}^{*}P_{\epsilon}$
satisfies a WPI.
\end{prop}

\begin{proof}
($\Rightarrow)$ If $P$ satisfies a WPI then one may deduce that
$P_{\epsilon}$ satisfies an WPI since $\mathcal{E}\left(P_{\epsilon},f\right)=\left(1-\epsilon\right)\cdot\mathcal{E}\left(P,f\right)$.
Since $P_{\epsilon}$ satisfies (\textbf{PH}), Theorem~\ref{thm:hold-equiv-wpi-conv-pp}
implies that $P_{\epsilon}^{*}P_{\epsilon}$ satisfies a WPI. ($\Leftarrow)$
If $P_{\epsilon}^{*}P_{\epsilon}$ satisfies a WPI then similarly
Theorem~\ref{thm:hold-equiv-wpi-conv-pp} implies that $P_{\epsilon}$
satisfies a WPI and so $P$ satisfies a WPI.
\end{proof}
\begin{prop}
\label{prop:S-conv-epsilon-P*P-WPI}If $S$ (or $P$) is convergent
then $P_{\epsilon}^{*}P_{\epsilon}$ satisfies a WPI, and thus $P_{\epsilon}$
is convergent.
\end{prop}

\begin{proof}
If $S$ is convergent then $S$ satisfies a WPI by Lemma~\ref{lem:osc-conv-RUPI}.
Hence $P$ satisfies a WPI by Remark~\ref{rem:nonrev_WPI} and we
conclude by Proposition~\ref{prop:WPI-epsilon-P*P-WPI}.
\end{proof}

Our counterexample of a Markov chain that is convergent but for which
$\left(P^{*}\right)^{k}P^{k}$ does not admit a WPI for any $k\in\mathbb{N}$
is the following.
\begin{example}
\label{exa:irred-pstarkpk}Let $\mathsf{E}=\left\{ 1,2,\ldots\right\} ^{2}$,
and $\nu$ a probability mass function on $\left\{ 1,2,\ldots\right\} $
such that $\nu\left(1\right)\in\left(0,1\right)$ and $\nu$ has a
finite mean. Define
\[
P\left(i,j;i^{'},j^{'}\right)=\begin{cases}
1 & j<i,i'=i,j'=j+1,\\
\nu\left(i'\right) & j=i,j'=1,\\
1 & \nu\left(i\right)\cdot\mathbf{1}\left\{ j\leqslant i\right\} =0,\left(i',j'\right)=\left(1,1\right),\\
0 & \text{otherwise}.
\end{cases}
\]
The intuition is that the Markov chain moves to the right along ``level''
$i$ deterministically until it reaches the point $\left(i,i\right)$,
at which point it jumps to the start of another level $\left(K,1\right)$
where $K\sim\nu$. The third statement is concerned with initialization
of the chain outside the support of the invariant distribution $\mu$,
given below.
\end{example}

\begin{prop}
\label{prop:preparatory-counterexample}The Markov kernel $P$ in
Example~\ref{exa:irred-pstarkpk} 
\begin{enumerate}
\item has invariant distribution
\[
\mu\left(i,j\right)=\frac{\nu\left(i\right)\cdot\mathbf{1}\left\{ j\leqslant i\right\} }{\sum_{k=1}^{\infty}\nu\left(k\right)\cdot k};
\]
\item is $\mu$-irreducible with an accessible, aperiodic atom $\left(1,1\right)$
and its Markov chain converges to $\mu$ in total variation from any
starting point .
\item its Markov chain is convergent,
\item is such that when $\nu\left(i\right)>0$ for all $i\geqslant1$,
\begin{enumerate}
\item \label{enu:counterexPstarkPkisone}then $\left(P^{*}\right)^{k}P^{k}$
is reducible for any $k\geqslant1$. More specifically for any $1\leqslant k<i$
\[
\left(P^{*}\right)^{k}P^{k}\left(i,i;i,i\right)=1;
\]
\item $\left(P^{*}\right)^{k}P^{k}$ does not admit a WPI for any $k\in\mathbb{N}$.
\end{enumerate}
\end{enumerate}
\end{prop}

\begin{proof}
The first two statements can be checked directly and using e.g., \cite[Theorem~7.6.4]{douc2018markov}.
The third statement follows from (\ref{eq:TV-CV-implies-osc}). Now,
by viewing $P^{*}$ as the time-reversal of $P$, and satisfying $\mu\left(i,j\right)\cdot P^{*}\left(i,j;i^{'},j^{'}\right)=\mu\left(i^{'},j^{'}\right)\cdot P\left(i^{'},j^{'};i,j\right)$,
we may define
\[
P^{*}\left(i,j;i^{'},j^{'}\right)=\begin{cases}
1 & j>1,i'=i,j'=j-1,\\
\nu\left(i^{'}\right) & j=1,j'=i',\\
1 & \nu\left(i\right)\cdot\mathbf{1}\left[j\leqslant i\right]=0,\big(i^{'},j^{'}\big)=\left(1,1\right),\\
0 & \text{otherwise}.
\end{cases}
\]
Consider first the case where $i_{0}\in\left\{ 2,3,\ldots\right\} $,
$\nu\left(i\right)>0$ for all $i\leqslant i_{0}$ and $\nu\left(i\right)=0$
for $i>i_{0}$. This means there is a maximum level length of $i_{0}$.
We see that if $k<i_{0}$ then
\[
\left(P^{*}\right)^{k}P^{k}\left(i_{0},i_{0};i_{0},i_{0}\right)=1,
\]
since $\left(P^{*}\right)^{k}\left(i_{0},i_{0};i_{0}-k,i_{0}-k\right)=1$
and $P^{k}\left(i_{0}-k,i_{0}-k;i_{0},i_{0}\right)=1$. Hence $\left(P^{*}\right)^{k}P^{k}$
is reducible for any $k<i_{0}-1$. On the other hand, for $k\geqslant i_{0}$,
we may deduce that $\left(P^{*}\right)^{k}P^{k}$ is $\mu$-irreducible.
In particular, since $P^{*}\left(1,1;1,1\right)=P\left(1,1;1,1\right)=\nu\left(1\right)\in\left(0,1\right)$,
we see that $\left(P^{*}\right)^{k}\left(i,j;1,1\right)>0$ for all
$\left(i,j\right)\in\mathsf{E}$, from which one may deduce that $\left(P^{*}\right)^{k}P^{k}\left(i,j;i^{'},j^{'}\right)>0$
for all $i,j,i',j'\in\mathsf{E}$ such that $\mu\left(i^{'},j^{'}\right)>0$.
Note that since $P\left(i,j;1,1\right)=1$ for all $\left(i,j\right)$
such that $\mu\left(i,j\right)=0$, this is essentially a finite state
space Markov chain after 1 step, and hence convergence is geometric.
Assume now that $\nu\left(i\right)>0$ for all $i\in\left\{ 1,2,\ldots\right\} $.
From above, for any $k\in\mathbb{N}$ we may consider level $i>k$
and we see that $\left(P^{*}\right)^{k}P^{k}\left(i,i;i,i\right)=1$
so $\left(P^{*}\right)^{k}P^{k}$ is reducible. Hence, there does
not exist $k\in\mathbb{N}$ such that $\left(P^{*}\right)^{k}P^{k}$
is $\mu$-irreducible.

For the last statement, let $k\in\mathbb{N}$ be arbitrary. Let $A_{k}=\left\{ \left(i,i\right):i>k\right\} $
and $f_{k}={\bf 1}_{A_{k}}-\mu\left(A_{k}\right)$, which satisfies
$\mu\left(f_{k}\right)=0$. Then $\left\Vert P^{k}f_{k}\right\Vert _{2}^{2}=\left\langle \left(P^{*}\right)^{k}P^{k}f_{k},f_{k}\right\rangle =\left\langle f_{k},f_{k}\right\rangle =\left\Vert f_{k}\right\Vert _{2}^{2}$,
from \ref{enu:counterexPstarkPkisone}, so $\mathcal{E}\big((P^{*})^{k}P^{k},f_{k}\big)=0$.
Since $\left\Vert f_{k}\right\Vert _{2}>0$, $\left(P^{*}\right)^{k}P^{k}$
cannot satisfy a WPI.
\end{proof}
\begin{rem}
One may assume that the non-existence of a WPI in Example~\ref{exa:irred-pstarkpk}
is the result of slow convergence for $P$. This is not the case,
as it can be shown that for the choice $\nu\left(i\right)=\left(1-a\right)\cdot a^{i-1}$
for some $a\in\left(0,1\right)$, the resulting Markov chain is $\left(\left\Vert \cdot\right\Vert _{\mathrm{osc}}^{2},\left\{ C\cdot\rho^{n}\rho^{n}\colon C>1,\rho\in\left[0,1\right),n\in\mathbb{N}\right\} \right)$-convergent
\cite[Proposition 67]{andrieu2022poincare_tech}. We note that this
scenario differs from what is normally understood as \textit{hypocoercivity},
for which one may hope to establish $\big(\left\Vert \cdot\right\Vert _{2}^{2},\left\{ C\cdot\rho^{n}\rho^{n}\colon C>1,\rho\in\left[0,1\right),n\in\mathbb{N}\right\} \big)$-convergence
i.e. for all $f\in\ELL\left(\mu\right)$
\[
\left\Vert P^{n}f\right\Vert _{2}^{2}\leqslant C\cdot\rho^{n}\cdot\left\Vert f\right\Vert _{2}^{2}\,,
\]
despite the degeneracy of the our generator ${\rm Id}-\left(P^{*}\right)^{k}P^{k}$
for any $k\in\mathbb{N}$. This would however imply the existence
of, for example, a SPI $\mathcal{E}\left(\left(P^{*}\right)^{k}P^{k},f\right)\geqslant\frac{1}{2}\left\Vert f\right\Vert _{2}^{2}$
for some $k\in\mathbb{N}$, therefore contradicting Proposition~\ref{prop:preparatory-counterexample}.
Hypocoercivity in the sense above indeed holds whenever $\nu$ has
finite support. 
\end{rem}

\subsection{WPIs from Close Coupling and Isoperimetry\label{subsec:WPIs-CC-Isop}}

This section is concerned with establishing \textit{quantitative}
bounds on WPIs and mixing times for Markov chains. Building on Section~\ref{subsec:Cheeger-inequalities},
the goal is to control the weak conductance profile of the Markov
kernel using isoperimetric properties of the target distribution and
a close coupling condition of the Markov kernel.

\subsubsection{Main results}

We first recall some concepts and results from \cite{andrieu2022explicit}.
\begin{assumption}
\label{assu:pi-density-lebesgue}The probability distribution $\pi$
has a positive density with respect to Lebesgue measure on $\mathsf{E}=\mathbb{R}^{d}$,
given by $\pi\propto\exp\left(-U\right)$, for some potential $U:\mathbb{R}^{d}\to\R$. 
\end{assumption}

\begin{defn}[Three-set isoperimetric inequality]
\label{def:three-set-iso-ineq}A probability measure $\pi$ satisfies
a \emph{three-set isoperimetric inequality} with metric $\mathsf{d}$
and function $F:\left(0,\frac{1}{2}\right]\to\left[0,\infty\right)$
if for all measurable partitions of the state space $\mathsf{E}=S_{1}\sqcup S_{2}\sqcup S_{3}$,
\[
\pi\left(S_{3}\right)\geqslant\mathsf{d}\left(S_{1},S_{2}\right)\cdot F\left(\min\left\{ \pi\left(S_{1}\right),\pi\left(S_{2}\right)\right\} \right).
\]
\end{defn}

\begin{defn}[Isoperimetric Profile]
For $A\in\mathscr{E}$ and $r\geqslant0$, let $A_{r}:=\left\{ x\in\E:\mathsf{d}\left(x,A\right)\leqslant r\right\} $,
and define the \textit{Minkowski content} of $A$ under $\pi$ with
respect to $\mathsf{d}$ by 
\[
\pi^{+}\left(A\right)=\lim\inf_{r\to0^{+}}\frac{\pi\left(A_{r}\right)-\pi\left(A\right)}{r}.
\]
The \textit{isoperimetric profile of} $\pi$ with respect to the metric
$\mathsf{d}$ is
\[
I_{\pi}\left(p\right):=\inf\left\{ \pi^{+}\left(A\right):A\in\mathscr{E},\pi\left(A\right)=p\right\} ,\qquad p\in\left(0,1\right).
\]
\end{defn}

We note briefly that the isoperimetric profile can be controlled explicitly
in many cases of interest; specific examples to this effect are provided
in Section \ref{subsec:Examples-of-Profiles}.
\begin{defn}
\label{def:iso-ineq}We say that $\tilde{I}_{\pi}:\left(0,1\right)\to\left(0,\infty\right)$
is an \textit{isoperimetric minorant of $\pi$} if $\tilde{I}_{\pi}\leqslant I_{\pi}$
pointwise. We furthermore say that $\tilde{I}_{\pi}$ is \textit{regular}
if it is symmetric about $\frac{1}{2}$, continuous, and monotone
increasing on $\left(0,\frac{1}{2}\right]$.
\end{defn}

To begin with, we recall that given any regular isoperimetric minorant
of our target measure, one can directly deduce a three-set isoperimetric
inequality for the same measure.
\begin{lem}
\label{lem:iso-to-3set}Let $\pi$ have a regular isoperimetric minorant
$\tilde{I}_{\pi}$ w.r.t. the metric $\mathsf{d}$. Then $\pi$ satisfies
a three-set isoperimetric inequality with metric $\mathsf{d}$ and
function $F=\tilde{I}_{\pi}$ on $\left(0,\frac{1}{2}\right]$.
\end{lem}

\begin{defn}
\label{def:close-coupling}For a metric $\mathsf{d}$ on $\E$ and
$\epsilon,\delta>0$, a Markov kernel $P$ evolving on $\mathsf{E}$
is \textit{$\left(\mathsf{d},\delta,\varepsilon\right)$-close coupling}
if
\[
\mathsf{d}\left(x,y\right)\;\leqslant\delta\Rightarrow\left\Vert P\left(x,\cdot\right)-P\left(y,\cdot\right)\right\Vert _{{\rm TV}}\leqslant1-\varepsilon,\qquad x,y\in\mathsf{E}.
\]
\end{defn}

\begin{lem}
\label{lem:iso-conductance-A-close-coupling}Suppose that $\pi$ satisfies
a three-set isoperimetric inequality with metric $\mathsf{d}$ and
function $F$ monotone increasing on $\left[0,\frac{1}{2}\right)$.
Let $P$ be a $\left(\mathsf{d},\delta,\varepsilon\right)$-close
coupling, $\pi$-invariant Markov kernel. Then for any $A\in\mathscr{E}$
with $\pi\left(A\right)\leqslant\frac{1}{2}$,
\[
\frac{\left(\pi\otimes P\right)\left(A\times A^{\complement}\right)}{\pi\left(A\right)}\geqslant\sup_{\theta\in\left[0,1\right]}\min\left\{ \frac{1}{2}\cdot\left(1-\theta\right)\cdot\varepsilon,\frac{1}{4}\cdot\varepsilon\cdot\delta\cdot\theta\cdot\left(F/{\rm id}\right)\left(\theta\cdot\pi\left(A\right)\right)\right\} .
\]
\end{lem}

In contrast to our earlier work, in the heavy-tailed case, isoperimetric
profiles and minorants are typically convex, meaning that $I_{\pi}/\mathrm{id}$
is usually increasing and vanishes at $0$. As such, taking an infimum
over sufficiently small sets $A$ in the above will give trivial bounds.
We thus instead take infima over sufficiently \textit{large} sets
$A$.
\begin{cor}
\label{cor:iso-couple-conductance}Suppose that $\tilde{I}_{\pi}$
is a regular, convex isoperimetric minorant of $\pi$ w.r.t. the metric
$\mathsf{d}$. Let $P$ be a $\left(\mathsf{d},\delta,\varepsilon\right)$-close
coupling, $\pi$-invariant Markov kernel. Then for any $v\in\left(0,\frac{1}{2}\right]$,
\begin{align*}
\Phi_{P}^{\mathrm{W}}\left(v\right) & \geqslant\sup_{\theta\in\left[0,1\right]}\min\left\{ \frac{1}{2}\cdot\left(1-\theta\right)\cdot\varepsilon,\frac{1}{4}\cdot\varepsilon\cdot\delta\cdot\theta\cdot\left(\tilde{I}_{\pi}/{\rm id}\right)\left(\theta\cdot v\right)\right\} \\
 & \geqslant\frac{1}{4}\cdot\varepsilon\cdot\min\left\{ 1,\frac{1}{2}\cdot\delta\cdot\frac{\tilde{I}_{\pi}\left(\frac{1}{2}\cdot v\right)}{\frac{1}{2}\cdot v}\right\} .
\end{align*}
\end{cor}

\begin{thm}
\label{thm:mixing-lower-bound-general}Let $\tilde{I}_{\pi}$ be a
regular, convex isoperimetric minorant of $\pi$ w.r.t. the metric
$\mathsf{d}$. Let $P$ be a positive, $\pi$-reversible and $\left(\mathsf{d},\delta,\varepsilon\right)$-close
coupling Markov kernel. Assume $2\cdot\delta\cdot\tilde{I}_{\pi}\left(\frac{1}{4}\right)\leqslant1$.
Then for
\[
\sup\big\{\left\Vert P^{n}f-\pi\left(f\right)\right\Vert _{2}^{2}\colon\left\Vert f\right\Vert _{\mathrm{osc}}^{2}\leqslant1\big\}\leqslant\varepsilon_{{\rm Mix}}\quad\text{for some }\varepsilon_{{\rm Mix}}\in\left(0,\frac{1}{4}\right),
\]
to hold, it suffices to take 
\begin{align*}
n & \geqslant1+2^{8}\cdot\varepsilon^{-2}\cdot\delta^{-2}\cdot\int_{2^{-3}\cdot\varepsilon_{{\rm Mix}}}^{2^{-5}}\frac{v\,\mathrm{d}v}{\tilde{I}_{\pi}\left(v\right)^{2}}.
\end{align*}
Let $\nu\ll\pi$ be a probability measure with $u:=\left\Vert \frac{\mathrm{d}\nu}{\mathrm{d}\pi}\right\Vert _{\mathrm{osc}}^{2}<\infty$.
Then in order to ensure that
\[
\chi^{2}\left(\nu P^{n},\pi\right)\leqslant\varepsilon_{{\rm Mix}}\quad\text{for some }\varepsilon_{{\rm Mix}}>0,
\]
it suffices that
\[
n\geqslant1+2^{8}\cdot\varepsilon^{-2}\cdot\delta^{-2}\cdot\int_{2^{-3}\cdot\varepsilon_{{\rm Mix}}\cdot u^{-1}}^{2^{-5}}\frac{v\,\mathrm{d}v}{\tilde{I}_{\pi}\left(v\right)^{2}}.
\]
\end{thm}

\begin{proof}
For any $f\in\ELL\left(\mu\right)$ and $n\in\mathbb{N}$, set $v_{n}:=\left\Vert P^{n}f-\pi\left(f\right)\right\Vert _{2}^{2}/\left\Vert f\right\Vert _{\mathrm{osc}}^{2}\leqslant1/4$.
Then from Corollaries~\ref{cor:mixing-with-weak-conductance} and
\ref{cor:iso-couple-conductance} we have for any $n\in\mathbb{N}$,
\begin{align*}
n & \leqslant2^{2}\cdot\int_{2^{-2}\cdot v_{n}}^{2^{-4}}\frac{\mathrm{d}v}{v\cdot\Phi_{P}^{\mathrm{W}}\left(v\right)^{2}}\\
 & \leqslant2^{2}\cdot\int_{2^{-2}\cdot v_{n}}^{2^{-4}}2^{6}\cdot\varepsilon^{-2}\cdot\delta^{-2}\cdot\frac{\left(\frac{1}{2}\cdot v\right)^{2}}{\left(\frac{1}{2}\cdot v\right)\cdot\left[\tilde{I}_{\pi}\left(\frac{1}{2}\cdot v\right)\right]^{2}}\mathrm{d}\left(\frac{1}{2}\cdot v\right)\\
 & \leqslant2^{8}\cdot\varepsilon^{-2}\cdot\delta^{-2}\cdot\int_{2^{-3}\cdot v_{n}}^{2^{-5}}\frac{v\,\mathrm{d}v}{\tilde{I}_{\pi}\left(v\right)^{2}}.
\end{align*}
Therefore for $n\in\mathbb{N}$ such that
\[
n\geqslant1+2^{8}\cdot\varepsilon^{-2}\cdot\delta^{-2}\cdot\int_{2^{-3}\cdot\varepsilon_{{\rm Mix}}}^{2^{-5}}\frac{v\,\mathrm{d}v}{\tilde{I}_{\pi}\left(v\right)^{2}},
\]
we must have $\sup\big\{\left\Vert P^{n}f-\pi\left(f\right)\right\Vert _{2}^{2}\colon\left\Vert f\right\Vert _{\mathrm{osc}}^{2}\leqslant1\big\}\leqslant\varepsilon_{{\rm Mix}}$.
The second result follows by considering $f=\frac{\mathrm{d}\nu}{\mathrm{d}\pi}$
and noting that $\Psi\left(f\right)=u$.
\end{proof}

\subsubsection{Examples of (Convex) Isoperimetric Profiles\label{subsec:Examples-of-Profiles}}

In this section, we give some examples of heavy-tailed probability
measures for which we can give bounds on the isoperimetric profile.
This will be based around results for some illustrative model classes
in general dimension with some degree of convexity, as well as some
results on how profiles can behave under tensorisation. Our main generic
examples are borrowed from \cite{cattiaux2010functional}, and some
more specific cases follow from analyses in \cite{bobkov2007large,cattiaux2010functional}.
The general theme is that for quasi-convex potentials, one can still
often obtain reasonable isoperimetric estimates, at least in terms
of capturing the correct asymptotic behaviour with respect to the
parameter $p$. In contrast, dependence on dimension (in the non-tensorised
case) and constant pre-factors can be a bit looser. As in the log-concave
case, transfer principles (e.g. Lipschitz transport, bounded change-of-measure)
remain relevant and straightforward to apply.

When $\mathsf{E}=\mathbb{R}$, consider probability distributions
of the form 
\[
\pi\left(\mathrm{d}x\right)\propto\exp\left(-U\left(\left|x\right|\right)\right)\mathrm{d}x\,,
\]
 with $U\colon\mathbb{R}_{+}\rightarrow\mathbb{R}$ concave. Letting
$\pi$ also denote the probability density and $G_{\pi}$ the corresponding
cumulative distribution function, define $J_{\pi}:=\pi\circ G_{\pi}^{-1}$.
Then \cite[Corollary 13.10]{bobkov1997some}, \cite[Eq. (5.20)]{cattiaux2010functional}
establish that 
\begin{align*}
I_{\pi}\left(p\right) & =\min\left\{ J_{\pi}\left(p\right),2\cdot J_{\pi}\left(\frac{1}{2}\cdot p\right)\right\} ,\quad p\in\left[0,\frac{1}{2}\right],
\end{align*}
corresponding to the conclusion that the extremal sets in the isoperimetric
problem take the form of either semi-infinite intervals, symmetric
pairs of semi-infinite intervals, or complements thereof. \cite[Proof of Corollary 5.16]{cattiaux2010functional}
later establish that for $d\geqslant1$, the product measures $\pi^{\otimes d}$
on $\mathsf{E}=\mathbb{R}^{d}$ for $d\geqslant1$ have isoperimetric
profiles such that for $p\in\left[0,\frac{1}{2}\right]$,
\[
d\cdot\frac{c_{1}}{c_{2}}\cdot J_{\pi}\left(\frac{1}{2\cdot c_{1}}\cdot\frac{p}{d}\right)\leqslant I_{\pi^{\otimes d}}\left(p\right)\leqslant d\cdot\frac{c_{1}}{c_{2}}\cdot J_{\pi}\left(\frac{1}{c_{1}}\cdot\frac{p}{d}\right)
\]
with $c_{1}=2\cdot\sqrt{6}$ and $c_{2}=2\cdot\left(1+c_{1}\right)$.
Convenient and accurate bounds for $J_{\pi}$ can be found with mild
additional conditions on $U$. 
\begin{prop}[{\cite[Corollary 5.23]{cattiaux2010functional}}]
 Assume $U\colon\mathbb{R}_{+}\rightarrow\mathbb{R}$ to be a non-decreasing
concave function such that
\begin{enumerate}
\item $\lim_{x\rightarrow\infty}\frac{U\left(x\right)}{x}=0$ and $U\left(0\right)<\log2$;
\item U is $\mathrm{C}^{2}$ in a neighbourhood of $\infty$;
\item $U^{\theta}$ for some $\theta>1$ is convex.
\end{enumerate}
Then,
\begin{enumerate}
\item there exist $k_{1},k_{2}>0$ such that for $p\in\left[0,1\right]$
\[
k_{1}\cdot p\cdot\left(U^{'}\circ U^{-1}\right)\left(\log\left(1/p\right)\right)\leqslant J_{\pi}\left(p\right)\leqslant k_{2}\cdot p\cdot\left(U^{'}\circ U^{-1}\right)\left(\log\left(1/p\right)\right);
\]
\item in particular, there exists $c>0$ such that for any $d\geqslant1$
\[
I_{\pi^{\otimes d}}\left(p\right)\geqslant c\cdot p\cdot\left(U^{'}\circ U^{-1}\right)\left(\log\left(\frac{d}{p}\right)\right)\quad p\in\left[0,\frac{1}{2}\right].
\]
\end{enumerate}
\end{prop}

Together, these results suggest that this estimate of the isoperimetric
profile of $\pi^{\otimes d}$ is accurate up to a factor $2$ in the
dimension $d$. In the examples below, we provide lower bounds for
$I_{\pi^{\otimes d}}\left(p\right)$, but explicit upper bounds can
also be found in \cite[p. 376]{cattiaux2010functional}, establishing
correct dependence on dimension for the two first examples.
\begin{example}
For some concrete examples, this approach yields that for $p\in\left(0,\frac{1}{2}\right)$,
\begin{itemize}
\item If $U\left(x\right)=\left|x\right|^{\eta}$ with $\eta\in\left(0,1\right)$,
then there exists $c>0$ (dependent on $\eta$ only) such that $I_{\pi^{\otimes d}}\left(p\right)\geqslant c\cdot p\cdot\left(\log\left(\frac{d}{p}\right)\right)^{-\left(1/\eta-1\right)}$;
\cite[Proposition 5.25; Remark 5.26]{cattiaux2010functional}.
\item If $U\left(x\right)=\left(1+\eta\right)\cdot\log\left(1+\left|x\right|\right)$,
then there exists $c$ dependent on $\eta$ only such that $I_{\pi^{\otimes d}}\left(p\right)\geqslant c\cdot d^{-1/\eta}\cdot p^{1+1/\eta}$
\cite[Remark 5.28]{cattiaux2010functional}; note that this is not
implied directly by \cite[Corollary 5.23]{cattiaux2010functional}
but follows from a separate result.
\item If $U\left(x\right)=\left|x\right|^{\eta_{1}}\cdot\left(\log\left(\eta_{3}+\left|x\right|\right)\right)^{\eta_{2}}$,
$\eta_{1}\in\left(0,1\right)$, $\eta_{2}\in\mathbb{R}$ and $\eta_{3}=\exp\left(\frac{2|\eta_{2}|}{\eta_{1}\cdot\left(1-\eta_{1}\right)}\right)$
(so that $U$ is globally concave on $\mathbb{R}$), then there exists
$c>0$ (dependent on $\eta_{1},\eta_{2},\eta_{3}$ only) such that
\[
I_{\pi^{\otimes d}}\left(p\right)\geqslant c\cdot p\cdot\left(\log\left(\frac{d}{p}\right)\right)^{-\left(1/\eta_{1}-1\right)}\cdot\left(\log\left(\log\left(e+\frac{d}{p}\right)\right)\right)^{\eta_{2}/\eta_{1}},
\]
(see \cite[Below Remark~5.26]{cattiaux2010functional}).
\end{itemize}
\end{example}

In particular, for sub-exponential-type targets, the effect of tensorisation
is roughly to attenuate the isoperimetric profile by a factor of $\left(\log d\right)^{-\left(1/\eta-1\right)}$.
In contrast, for targets with polynomial tails, the profile is attenuated
more dramatically by a polynomial factor of $d^{-1/\alpha}$.

From Theorem~\ref{thm:mixing-lower-bound-general} this leads to
the following mixing time upper bounds 
\begin{itemize}
\item For $I_{\pi^{\otimes d}}\left(p\right)\geqslant c\cdot d^{-1/\eta}\cdot p^{1+1/\eta}$
we have 
\begin{align*}
\int_{2^{-4}\cdot\varepsilon_{{\rm Mix}}}^{2^{-6}}\frac{v\,\mathrm{d}v}{I_{\pi^{\otimes d}}\left(v\right){}^{2}} & \leqslant c^{-2}\cdot d^{2/\eta}\cdot\int_{2^{-4}\cdot\varepsilon_{{\rm Mix}}}^{2^{-6}}v^{-1-2/\eta}{\rm d}v\\
 & \leqslant\eta\cdot2^{8/\eta-1}\cdot c^{-2}\cdot d^{2/\eta}\cdot\varepsilon_{{\rm Mix}}^{-2/\eta},
\end{align*}
that is, it is sufficient to take 
\begin{align*}
n & \geqslant1+c^{-2}\cdot\eta\cdot2^{8/\eta+8}\cdot\varepsilon^{-2}\cdot\delta^{-2}\cdot d^{2/\eta}\cdot\varepsilon_{{\rm Mix}}^{-2/\eta}\\
 & =\Omega\left(\varepsilon^{-2}\cdot\delta^{-2}\cdot d^{2/\eta}\cdot\varepsilon_{{\rm Mix}}^{-2/\eta}\right),
\end{align*}
in Theorem~\ref{thm:mixing-lower-bound-general}.
\item For $I_{\pi^{\otimes d}}\left(p\right)\geqslant c\cdot p\cdot\log\left(\frac{d}{p}\right)^{1-1/\eta}$
, $\eta\in\left(0,1\right)$ we have
\begin{align*}
\int_{2^{-4}\cdot\varepsilon_{{\rm Mix}}}^{2^{-6}}\frac{v\,\mathrm{d}v}{I_{\pi^{\otimes d}}\left(v\right)^{2}} & \leqslant c^{-2}\cdot\int_{2^{-4}\cdot\varepsilon_{{\rm Mix}}}^{2^{-6}}\frac{{\rm d}v}{v\cdot\log\left(\frac{d}{v}\right)^{2-2/\eta}}\\
 & \leqslant c^{-2}\cdot\frac{\eta}{2-\eta}\cdot\log\left(2^{4}\cdot d\cdot\varepsilon_{{\rm Mix}}^{-1}\right)^{2/\eta-1},
\end{align*}
that is it is sufficient to take 
\begin{align*}
n & \geqslant1+2^{9}\cdot c^{-2}\cdot\frac{\eta}{2-\eta}\cdot\varepsilon^{-2}\cdot\delta^{-2}\cdot\left(\log\left(2^{4}\cdot d\cdot\varepsilon_{{\rm Mix}}^{-1}\right)\right)^{2/\eta-1}\\
 & =\Omega\left(\varepsilon^{-2}\cdot\delta^{-2}\cdot\left(\log\left(d\cdot\varepsilon_{{\rm Mix}}^{-1}\right)\right)^{2/\eta-1}\right),
\end{align*}
in Theorem~\ref{thm:mixing-lower-bound-general}.
\end{itemize}
Some attempts have been made in \cite{cattiaux2010functional} to
extend the class of distributions for which lower bounds on the isoperimetric
profile can be found. In what follows, let $V\colon\mathsf{E}=\mathbb{R}^{d}\rightarrow\mathbb{R}$
denote a generic coercive convex function satisfying a growth condition
of the form
\[
\left|x\right|\geqslant R\implies V\left(x\right)\geqslant V\left(0\right)+\delta\cdot\left|x\right|,
\]
for some $R,\delta>0$. Here, $V$ will dictate the shape of the contours
of our target, and our potential $U$ will be obtained as a monotone,
sublinear function of $V$.
\begin{example}
\cite[Proposition 4.3]{cattiaux2010functional} For Cauchy-type targets
of the form $\pi\left(\mathrm{d}x\right)\propto V\left(x\right)^{-\left(d+\eta\right)}\,\mathrm{d}x$
with $\eta>0$, it holds that $I_{\pi}\left(p\right)\geqslant c\cdot p^{1+1/\eta}$
for $p\in\left[0,\frac{1}{2}\right]$, for some constant $c$ which
depends on $\left(V,\eta,d\right)$. A general result of \cite[Theorem 1.2]{bobkov2007large}
shows that one can in fact take $c$ to depend only on $\eta$ and
$\mathrm{med}_{\pi}\left(\left|x\right|\right)$. 

As an application of Bobkov's result, one can see that the multivariate
Student-T distribution $\pi\left(\mathrm{d}x\right)\propto\left(\tau+\left|x\right|_{2}^{2}\right)^{-\frac{d+\tau}{2}}\,\mathrm{d}x$
admits an isoperimetric minorant of the form $I_{\pi}\left(p\right)\geqslant c\cdot d^{-1/2}\cdot p^{1+1/\tau}$,
where $c$ can be bounded uniformly from below for $\min\left\{ d,\tau\right\} \in\Omega\left(1\right)$,
see Lemma~\ref{lem:conductance-t-student-constant-bound}. Related
functional inequalities suggest that the dimensional factor $d^{-1/2}$
is likely to be extraneous (see e.g. the (adimensional) weighted inequalities
of \cite{bobkov2009weighted}), but we are not aware of a formal proof
of this.
\end{example}

\begin{example}
\cite[Proposition 4.5]{cattiaux2010functional} For sub-exponential-type
targets of the form $\pi\left(\mathrm{d}x\right)\propto\exp\left(-V\left(x\right)^{\eta}\right)\,\mathrm{d}x$
with $\eta\in\left(0,1\right)$, it holds that $I_{\pi}\left(p\right)\geqslant c\cdot p\cdot\log\left(\frac{1}{p}\right)^{-\left(1/\eta-1\right)}$
for $p\in\left[0,\frac{1}{2}\right]$, for some constant $c$ which
depends on $\left(V,\eta,d\right)$.
\end{example}

While these results allow one to characterize rates of convergence
for fixed $d$ for a very broad class of distributions, characterization
of $c$, and therefore mixing times, in terms of the dimension $d$
would require considering particular cases, or even different techniques.

\section{Application to Pseudo-marginal MCMC\label{sec:Application-to-Pseudo-marginal}}

Fix a probability distribution $\pi$ on a measure space $\mathsf{X}$,
with a density function with respect to some measure $\nu$ denoted
$\varpi$. Pseudo-marginal algorithms \cite{andrieu2009pseudo} extend
the scope of the standard Metropolis--Hastings algorithm to the scenario
where the density $\varpi$ is intractable, but for any $x\in\mathsf{X}$,
nonnegative estimators $\hat{\varpi}\left(x\right)$ such that $\mathbb{E}\left[\hat{\varpi}\left(x\right)\right]=C\cdot\varpi\left(x\right)$
for some constant $C>0$ are available. The main idea of pseudo-marginal
algorithms is to replace the value of $\varpi\left(x\right)$ with
$\hat{\varpi}\left(x\right)$ in standard algorithms. The surprising
fact is that these algorithms are correct algorithms, in the sense
that under standard conditions they are guaranteed to produce samples
from a distribution arbitrarily close to $\pi$. This comes however
at a price, as pseudo-algorithms are slower to converge then their
exact, or `marginal', counterparts which would use the intractable
$\varpi(x)$. The degradation of performance of the algorithm stemming
from the presence of weights has been investigated in \cite{andrieu2015convergence,andrieu2016establishing}
and more recently in \cite{andrieu2022comparison_journal} and upper
bounds on convergence rates and/or complexity bounds were obtained.
The object of this section is to exploit ``the left hand side Cheeger
inequality'' in Theorem~\ref{thm:positive-CP-implies-optim-WPI}
to establish lower bounds on the rate of convergence of pseudo-marginal
algorithms.

Pseudo-marginal algorithms can be conveniently formulated as follows.
Let $\mu\left({\rm d}x,{\rm d}w\right)=\pi\left({\rm d}x\right)\cdot Q_{x}\left({\rm d}w\right)\cdot w=\pi\left({\rm d}x\right)\cdot\tilde{\pi}_{x}\left({\rm d}w\right)$
with $\int_{\mathbb{R}_{+}}w\,Q_{x}\left({\rm d}w\right)=1$ on an
extended space $\E:=\mathsf{X\times\mathbb{R}_{+}}$. We will refer
to these auxiliary $w$ random variables as \emph{weights} or \emph{perturbations}.

The \emph{marginal (Metropolis--Hastings) algorithm}, which uses
the exact density $\varpi={\rm d}\pi/{\rm d}\nu$ for some measure
$\nu$ requires the specification of a family of proposal distributions
for $\left\{ q\left(x,\cdot\right),x\in\mathsf{X}\right\} $ (again
having density with respect to $\nu$) and has transition kernel
\begin{align}
P\left(x,{\rm d}y\right) & =\left[1\wedge r\left(x,y\right)\right]\,q\left(x,{\rm d}y\right)+\delta_{x}\left({\rm d}y\right)\cdot\rho\left(x\right),\nonumber \\
 & \quad\text{where}\quad r\left(x,y\right):=\frac{\varpi\left(y\right)\cdot q\left(y,x\right)}{\varpi\left(x\right)\cdot q\left(x,y\right)},\label{eq:MH-kernel}
\end{align}
and $\rho$ is the rejection probability given by $\rho\left(x\right):=1-\int\left[1\wedge r\left(x,y\right)\right]\,q\left(x,{\rm d}y\right)$
for each $x\in\mathsf{X}$. For brevity, we will also define the acceptance
probability as $a\left(x,y\right):=1\wedge r\left(x,y\right)$.

The pseudo-marginal Metropolis--Hastings kernel is given by
\begin{equation}
\tilde{P}\left(x,w;{\rm d}y,{\rm d}u\right)=\left[1\wedge\left\{ r\left(x,y\right)\cdot\frac{u}{w}\right\} \right]\cdot q\left(x,{\rm d}y\right)\cdot Q_{y}\left({\rm d}u\right)+\delta_{x,w}\left({\rm d}y,{\rm d}u\right)\cdot\tilde{\rho}\left(x,w\right),\label{eq:PM-kernel}
\end{equation}
where the (joint) rejection probability $\tilde{\rho}\left(x,w\right)$
is defined analogously. From \cite{andrieu2015convergence,andrieu2016establishing,andrieu2022comparison_journal}
it is known that can upper bound the convergence rate of the pseudo-marginal
algorithm in terms of properties of the marginal algorithm $P$ and
tail properties of $\left\{ Q_{x}\left({\rm d}w\right),x\in\mathsf{X}\right\} $.
The following establishes lower bounds.
\begin{thm}
\label{thm:PM-Cheeger-Lower}Let $\tilde{P}$ be as in (\ref{eq:PM-kernel})
with $q$ some $\nu$-reversible kernel. For any $s>0$, let $A_{s}:=\left\{ \left(x,w\right)\in\mathsf{X}\times\mathbb{R}_{+}\colon\varpi\left(x\right)w\geqslant s\right\} $,
$\psi\colon\mathbb{R}_{+}\rightarrow\left[0,1\right]$ be such that
$\psi\left(s\right):=\mu\left(A_{s}\right)$ and for any $v\in\left(0,1\right]$,
consider the generalized inverse,
\begin{equation}
\psi^{-}(v):=\sup\left\{ s>0:\mu\left(A_{s}\right)\geqslant v\right\} .\label{eq:def-psi-gen-inv}
\end{equation}
Then with $\Phi_{\tilde{P}}^{\mathrm{W}}$ the weak conductance of
$\tilde{P}$ as given in (\ref{eq:PM-kernel}), we have for $v>0$,
\begin{enumerate}
\item 
\[
\Phi_{\tilde{P}}^{\mathrm{W}}\left(v\right)\leqslant\frac{\int\pi\left(\mathrm{d}x\right)\tilde{\pi}_{x}\big(\varpi\left(x\right)W\geqslant\psi^{-}\left(v\right)\big)\int q\left(x,\mathrm{d}y\right)\frac{\varpi\left(y\right)\tilde{\pi}_{y}\left(\varpi\left(y\right)W<\psi^{-}\left(v\right)\right)}{\psi^{-}\left(v\right)}}{\int\pi\left(\mathrm{d}x\right)\tilde{\pi}_{x}\big(\varpi\left(x\right)W\geqslant\psi^{-}\left(v\right)\big)};
\]
\item Assuming that $\bar{\mathfrak{w}}:=\sup_{x\in\mathsf{X}}\varpi\left(x\right)<\infty$,
\[
\Phi_{\tilde{P}}^{\mathrm{W}}(v)\leqslant\frac{\bar{\mathfrak{w}}}{\psi^{-}\left(v\right)}.
\]
\end{enumerate}
\end{thm}

\begin{proof}
The first statement follows from the fact that by definition for $0<v\leqslant1/2$
\[
\Phi_{\tilde{P}}^{\mathrm{W}}\left(v\right)\leqslant\inf_{\left\{ s>0\colon1/2\geqslant\mu\left(A_{s}\right)\geqslant v\right\} }\frac{\mathcal{E}\big(\tilde{P},\mathbf{1}_{A_{s}}\big)}{\mu\big(\mathbf{1}_{A_{s}}\big)},
\]
Lemma~\ref{lem:upper-bound-conductance-PM}. The second statement
is a direct consequence of the first. 

\end{proof}
One could treat the case in which $\sup_{x\in\mathsf{X}}\varpi\left(x\right)=\infty$,
but as most practical applications satisfy $\bar{\mathfrak{w}}<\infty$,
we provide only the simpler result here.
\begin{example}
\label{exa:PM-Pareto} Assume that there exists a probability distribution
$\tilde{\pi}_{*}$ and $C,\underline{\mathfrak{w}}>0$ such that for
any $\left(x,A\right)\in\left\{ y\in\mathsf{X}\colon\varpi\left(y\right)\geqslant\underline{\mathfrak{w}}\right\} \times\mathcal{B}\left(\mathbb{R}_{+}\right)$,
it holds that $\tilde{\pi}_{x}\left(A\right)\geqslant C\cdot\tilde{\pi}_{*}\left(A\right)$.
Then for any $s>0$, we can bound
\begin{align*}
\mu\left(A_{s}\right) & \geqslant C\cdot\int\mathbf{1}\left\{ \varpi\left(x\right)\geqslant\underline{\mathfrak{w}},w\geqslant s/\varpi\left(x\right)\right\} \cdot\pi\left({\rm d}x\right)\cdot\tilde{\pi}_{*}\left({\rm d}w\right)\\
 & \geqslant C\cdot\pi\left(\varpi\geqslant\underline{\mathfrak{w}}\right)\cdot\tilde{\pi}_{*}\left(W\geqslant s/\underline{\mathfrak{w}}\right)
\end{align*}
and therefore
\begin{align*}
\psi^{-}\left(v\right) & \geqslant\sup\left\{ s>0:C\cdot\pi\left(\varpi\geqslant\underline{\mathfrak{w}}\right)\cdot\tilde{\pi}_{*}\left(W\geqslant s/\underline{\mathfrak{w}}\right)\geqslant v\right\} \\
 & =\underline{\mathfrak{w}}\cdot\sup\left\{ s>0\colon\tilde{\pi}_{*}\left(W\geqslant s\right)\geqslant C^{-1}\cdot v/\pi\left(\varpi\geqslant\underline{\mathfrak{w}}\right)\right\} ,
\end{align*}
that is, the inverse $\psi_{*}^{-}$ of $\tilde{\pi}_{*}$ provides
the functional form of a lower bound. Consider for example the scenario
where $Q_{*}$ is a ${\rm Pareto}\left(\alpha,x_{m}\right)$ with
$x_{m}=1-\alpha^{-1}$ and $\alpha>1$ to ensure that $\int w\cdot Q_{*}\left({\rm d}w\right)=1$.
In this specific case, the inverse $\psi_{*}^{-1}$ is well-defined
and it is shown in Appendix~\ref{app:exa-pm-pareto} that for $v>0$,
\[
\psi_{*}^{-1}\left(v\right)=x_{m}\cdot v^{-1/\left(\alpha-1\right)}\,.
\]
As a result, as the tails of the Pareto distribution become heavier
(i.e. $\alpha\downarrow1$), the weak conductance profile $\Phi_{\tilde{P}}^{\mathrm{W}}\left(v\right)$
vanishes more rapidly at $0$, resulting in a slower convergence rate
from Theorems~\ref{thm:WPI_F_bd} and \ref{thm:PM-Cheeger-Lower}
for a positive chain.
\end{example}

\begin{rem}
The results above require $q$ to be $\nu$-reversible, which covers
the RWM algorithm, but we suspect the results to hold more generally.
A simple example is the Independent MH algorithm. Assume that for
some $\epsilon>0$, any $x\in\mathsf{X}$, $\epsilon\leqslant\varpi\left(x\right)=\frac{{\rm d}\pi}{{\rm d}q}\left(x\right)\geqslant\epsilon^{-1}$,
let $A_{s}:=\left\{ \left(x,w\right):w\geqslant s\right\} =\mathsf{X}\times\left\{ w\colon w\geqslant s\right\} $
and define the corresponding $\psi^{-}$ using (\ref{eq:def-psi-gen-inv}).
From Jensen's inequality and these assumptions, we have for $s\geqslant\epsilon^{-2}$
that
\end{rem}

\begin{align*}
\mathcal{E}\big(\tilde{P},\mathbf{1}_{A_{s}}\big)\\
=\int\pi & \left({\rm d}x\right)q\left(x,{\rm d}y\right)\tilde{\pi}_{x}\left({\rm d}w\right)Q_{y}\left({\rm d}u\right)\mathbf{1}\left\{ w\geqslant s,u<s\right\} \min\left\{ 1,r\left(x,y\right)\frac{u}{w}\right\} \\
\leqslant\int\pi & \left({\rm d}x\right)\tilde{\pi}_{x}\left({\rm d}w\right)\mathbf{1}\left\{ w\geqslant s\right\} q\left(x,{\rm d}y\right)\min\left\{ 1,r\left(x,y\right)\cdot\frac{1}{w}\right\} \\
\leqslant\int\pi & \left({\rm d}x\right)\tilde{\pi}_{x}\left({\rm d}w\right)\mathbf{1}\left\{ w\geqslant s\right\} q\left(x,{\rm d}y\right)\epsilon^{-2}\cdot\frac{1}{w}\\
 & \leqslant\mu\left(A_{s}\right)/s\,,
\end{align*}
leading to identical conclusions as in Theorem~\ref{thm:PM-Cheeger-Lower}
and Example~\ref{exa:PM-Pareto}.
\begin{lem}
\label{lem:upper-bound-conductance-PM}Let $\tilde{P}$ be as in (\ref{eq:PM-kernel})
and define, for any $s>0$, 
\[
A_{s}:=\left\{ \left(x,w\right)\in\mathsf{X}\times\mathbb{R}_{+}\colon\varpi\left(x\right)\cdot w\geqslant s\right\} \,.
\]
Then for any $s>0$,
\[
\frac{\mathcal{E}\big(\tilde{P},\mathbf{1}_{A_{s}}\big)}{\mu\big(A_{s}\big)}\leqslant\frac{\int\pi\left(\mathrm{d}x\right)\tilde{\pi}_{x}\left(\varpi\left(x\right)W\geqslant s\right)\int q\left(x,\mathrm{d}y\right)\frac{\varpi(y)}{s}\tilde{\pi}_{y}\left(\varpi\left(y\right)W<s\right)}{\int\pi\left(\mathrm{d}x\right)\tilde{\pi}_{x}\left(\varpi\left(x\right)W\geqslant s\right)}.
\]
\end{lem}

\begin{proof}
Let $s\geqslant0$. We have that
\begin{align*}
\mathcal{E}\big(\tilde{P},\mathbf{1}_{A_{s}}\big) & =\frac{1}{2}\int\nu\left(\mathrm{d}x\right)Q_{x}\left(\mathrm{d}w\right)q\left(x,\mathrm{d}y\right)Q_{y}\left(\mathrm{d}u\right)\mathbf{1}\left\{ s\leqslant\varpi\left(x\right)\cdot w\right\} \\
 & \hspace{4cm}\times\mathbf{1}\left\{ s>\varpi\left(y\right)u\right\} \min\left\{ \varpi(x)w,\varpi(y)u\right\} \\
 & =\frac{1}{2}\int\nu\left(\mathrm{d}x\right)Q_{x}\left(\mathrm{d}w\right)q\left(x,\mathrm{d}y\right)Q_{y}\left(\mathrm{d}u\right)\mathbf{1}\left\{ s\leqslant\varpi\left(x\right)w\right\} \\
 & \hspace{1cm}\times\int_{0}^{\infty}{\rm d}h\mathbf{1}\left\{ h\leq\varpi\left(x\right)w\right\} \mathbf{1}\left\{ h\leq\varpi\left(y\right)u\right\} 
\end{align*}
For $x,y\in\mathsf{X}$,
\[
\int Q_{x}\left(\mathrm{d}w\right)\mathbf{1}\left\{ h\leqslant\varpi\left(x\right)w\right\} \mathbf{1}\left\{ s\leqslant\varpi\left(x\right)w\right\} =\int_{0}^{\infty}Q_{x}\left(\varpi\left(x\right)W\geqslant h\vee s\right),
\]
and
\[
\int Q_{y}\left(\mathrm{d}u\right)\mathbf{1}\left\{ h\leqslant\varpi\left(y\right)u\right\} \mathbf{1}\left\{ s>\varpi\left(y\right)u\right\} =Q_{y}\left(h\leqslant\varpi\left(y\right)W<s\right).
\]
Further,
\begin{align*}
\int_{0}^{\infty}Q_{x} & \left(\varpi\left(x\right)W\geqslant h\vee s\right)Q_{y}\left(h\leqslant\varpi\left(y\right)W<s\right)\mathrm{d}h\\
 & =\mathbf{1}\left\{ h\leqslant s\right\} \int_{0}^{s}\,Q_{x}\left(\varpi\left(x\right)W\geqslant s\right)Q_{y}\left(h\leqslant\varpi\left(y\right)W<s\right)\mathrm{d}h\\
 & =\mathbf{1}\left\{ h\leqslant s\right\} Q_{x}\left(\varpi\left(x\right)\cdot W\geqslant s\right)\int_{0}^{s}Q_{y}\left(h\leqslant\varpi\left(y\right)W<s\right)\mathrm{d}h
\end{align*}
Changing the order of integration, for $h\leqslant s$, we obtain
that 
\begin{align*}
\int_{0}^{s}Q_{y}\left(h\leqslant\varpi\left(y\right)W<s\right)\mathrm{d}h & =\int_{0}^{s}\int\mathbf{1}\left\{ h\leqslant\varpi\left(y\right)w<s\right\} \mathrm{d}hQ_{y}\left({\rm d}w\right)\\
 & =\int\varpi\left(y\right)w\mathbf{1}\left\{ \varpi\left(y\right)w<s\right\} Q_{y}\left({\rm d}w\right)\\
 & =\varpi\left(y\right)\tilde{\pi}_{y}\left(\varpi\left(y\right)W<s\right).
\end{align*}
Consequently, for $s>0$,
\begin{align*}
\mathcal{E}\big(\tilde{P},\mathbf{1}_{A_{s}}\big)\\
=\int\nu & \left(\mathrm{d}x\right)q\left(x,\mathrm{d}y\right)\cdot Q_{x}\left(\varpi\left(x\right)\cdot W\geqslant s\right)\cdot\varpi\left(y\right)\cdot\tilde{\pi}_{y}\left(\varpi\left(y\right)\cdot W<s\right)\\
=\int\pi & \left(\mathrm{d}x\right)q\left(x,\mathrm{d}y\right)\cdot\frac{1}{\varpi\left(x\right)}\cdot Q_{x}\left(\varpi\left(x\right)\cdot W\geqslant s\right)\cdot\varpi\left(y\right)\cdot\tilde{\pi}_{y}\left(\varpi\left(y\right)\cdot W<s\right)\\
\leqslant\int\pi & \left(\mathrm{d}x\right)q\left(x,\mathrm{d}y\right)\cdot\frac{1}{s}\cdot\tilde{\pi}_{y}\left(\varpi\left(x\right)\cdot W\geqslant s\right)\cdot\varpi\left(y\right)\cdot\tilde{\pi}_{y}\left(\varpi\left(y\right)\cdot W<s\right),
\end{align*}
because
\begin{align*}
\frac{1}{\varpi\left(x\right)}Q_{x}\left(\varpi\left(x\right)\cdot W\geqslant s\right) & =\int\frac{1}{\varpi\left(x\right)}\mathbf{1}\left\{ \varpi\left(x\right)\cdot w\geqslant s\right\} \cdot Q_{x}\left({\rm d}w\right)\\
 & =\int\frac{1}{\varpi\left(x\right)\cdot w}\mathbf{1}\left\{ \varpi\left(x\right)\cdot w\geqslant s\right\} \cdot w\cdot Q_{x}\left({\rm d}w\right)\,.\\
 & \leqslant s^{-1}\cdot\int\mathbf{1}\left\{ \varpi\left(x\right)\cdot w\geqslant s\right\} \cdot w\cdot Q_{x}\left({\rm d}w\right)\,.
\end{align*}
The result follows.
\end{proof}

\section{Applications to RWM on heavy-tailed targets\label{sec:Applications-to-RWM}}

We assume now that $\pi$ has density $\varpi=\frac{{\rm d}\pi}{{\rm d}\lambda}$
with respect to some $\sigma$-finite measure $\lambda$, and that
$Q$ is a $\lambda$-reversible Markov kernel. A \textit{Metropolis
Markov kernel $P$} may be defined \cite{metropolis1953equation}
as
\begin{equation}
P\left(x,A\right)=\int_{A}Q\left(x,{\rm d}y\right)\cdot\alpha\left(x,y\right)+{\bf 1}_{A}\left(x\right)\cdot\bar{\alpha}\left(x\right),\qquad x\in\mathsf{E},A\in\mathscr{E}.\label{eq:metropolis-kernel}
\end{equation}
where for $x,y\in\E$,
\begin{equation}
\alpha\left(x,y\right):=1\wedge\frac{\varpi\left(y\right)}{\varpi\left(x\right)},\quad\bar{\alpha}\left(x\right):=1-\alpha\left(x\right),\quad\alpha\left(x\right):=\int_{\mathsf{E}}Q\left(x,{\rm d}y\right)\cdot\alpha\left(x,y\right).\label{eq:alpha_defs}
\end{equation}

\begin{example}
\label{exa:NRWM}In many applications, $Q\left(x,\cdot\right)$ is
a multivariate normal distribution with mean $x$ and covariance matrix
$\sigma^{2}\cdot I_{d}$: for some $\sigma>0$,
\[
Q\left(x,A\right)=\int{\bf 1}_{A}\left(x+\sigma\cdot z\right)\,{\cal N}\left({\rm d}z;0,I_{d}\right),\qquad x\in\mathsf{E},A\in\mathscr{E}.
\]
Since $Q$ is reversible w.r.t. the Lebesgue measure on $\mathbb{R}^{d}$,
one can take $\lambda\left({\rm d}x\right)={\rm d}x$ and consider
target distributions $\pi$ with densities w.r.t. the Lebesgue measure.
This defines the \textit{Random-Walk Metropolis} (RWM) Markov kernel.
We remark that various other proposals can also readily be analysed;
see \cite[Appendix~C]{andrieu2022explicit} for relevant details.
\end{example}

The following result can be used to establish the close coupling condition
of Definition~\ref{def:close-coupling} for Metropolis kernels, which
we will use in the sequel to analyze the RWM chain.
\begin{lem}
\label{lem:met-tv-bound}(Lemma 19 of \cite{andrieu2022explicit})
Let $Q$ be a $\lambda$-reversible Markov kernel, where $\lambda\gg\pi$
is a $\sigma$-finite measure, $P$ be the $\pi$-reversible Metropolis
kernel with proposal $Q$ defined by (\ref{eq:metropolis-kernel})
and $\alpha_{0}:=\inf_{x\in\mathsf{E}}\alpha\left(x\right)$. Then
\[
\left\Vert P\left(x,\cdot\right)-P\left(y,\cdot\right)\right\Vert _{{\rm TV}}\leqslant\left\Vert Q\left(x,\cdot\right)-Q\left(y,\cdot\right)\right\Vert _{{\rm TV}}+1-\alpha_{0},\qquad x,y\in\mathsf{E}.
\]
\end{lem}

We now assume Assumption~\ref{assu:pi-density-lebesgue} to hold,
that is $\varpi\propto\exp\left(-U\right)$ with $\lambda$ the Lebesgue
measure and $Q$ is as in Example~\ref{exa:NRWM}. It is standard
to deduce by \cite[Lemma~3.1]{baxendale2005renewal} that $P$ is
a positive Markov kernel for this particular $Q$. We note the following
useful expression
\begin{equation}
\alpha\left(x\right)=\int{\cal N}\left({\rm d}z;0,I_{d}\right)\cdot\min\left\{ 1,\exp\left(-\left(U\left(x+\sigma\cdot z\right)-U\left(x\right)\right)\right)\right\} \,.\label{eq:RWM-alpha-x}
\end{equation}

\subsection{Mixing times of RWM for $L$-smooth targets}

For the purposes of obtaining explicit bounds and matching negative
results with dimension, we impose the following further assumption
about $\pi$, noting that Assumption~\ref{assu:pi-density-lebesgue}
is already in force. As established in Lemma~39 of \cite{andrieu2022explicit},
this condition can be weakened to accommodate more general quantitative
continuity assumptions; we do not linger on this point here.
\begin{assumption}
\label{assu:target-distribution}For some $L\in\left(0,\infty\right)$,
$U$ is $L$-smooth:
\[
U\left(x+h\right)-U\left(x\right)-\left\langle \nabla U\left(x\right),h\right\rangle \leqslant\frac{L}{2}\cdot\left|h\right|^{2},\qquad x,h\in\E.
\]
\end{assumption}

\begin{lem}
\label{lem:close-coupling-concise}(\cite[Lemma~37, Corollary~40]{andrieu2022explicit})
Suppose that the potential $U:\mathbb{R}^{d}\to\mathbb{R}$ is $L$-smooth,
and let $P$ be the RWM kernel with target distribution $\pi=\exp\left(-U\right)$
and step-size $\sigma=\varsigma\cdot\left(L\cdot d\right)^{-1/2}$.
It then holds that $P$ is $\left(\left|\cdot\right|,\alpha_{0}\cdot\sigma,\frac{1}{2}\cdot\alpha_{0}\right)$-close
coupling with $\alpha_{0}\geqslant\frac{1}{2}\cdot\exp\left(-\frac{1}{2}\cdot\varsigma^{2}\right)$.
\end{lem}

\begin{thm}
\label{thm:rwm-payoff}Suppose that $\pi$ is a probability measure
on $\mathbb{R}^{d}$ with regular, convex isoperimetric minorant $\tilde{I}_{\pi}$,
and that the potential $U=-\log\pi$ is $L$-smooth. Let $P$ be the
RWM kernel with target distribution $\pi$ and step-size $\sigma=\varsigma\cdot\left(L\cdot d\right)^{-1/2}$.
Assume also that $2\cdot\varsigma\cdot\left(L^{-1/2}\cdot\tilde{I}_{\pi}\left(\frac{1}{4}\right)\right)\leqslant d^{1/2}$.
It then holds that
\end{thm}

\begin{enumerate}
\item \label{enu:thm-appli-WCP}$P$ has a positive WCP, satisfying for
$v\in\left(0,\frac{1}{2}\right]$,
\[
\Phi_{P}^{\mathrm{W}}\left(v\right)\geqslant2^{-6}\cdot\varsigma\cdot\exp\left(-\varsigma^{2}\right)\cdot\left(L\cdot d\right)^{-1/2}\cdot\frac{\tilde{I}_{\pi}\left(\frac{1}{2}\cdot v\right)}{\frac{1}{2}\cdot v}.
\]
\item \label{enu:thm-appli-Kstar-WPI}$P$ satisfies a $K^{*}$\textit{\emph{--}}WPI,
with $K^{*}$ lower-bounded as 
\[
K^{*}\left(v\right)\geqslant2^{-11}\cdot\varsigma^{2}\cdot\exp\left(-2\cdot\varsigma^{2}\right)\cdot\left(L\cdot d\right)^{-1}\cdot\frac{\tilde{I}_{\pi}\left(2^{-3}\cdot v\right)^{2}}{2^{-3}\cdot v}.
\]
\item \label{enu:thm-appli-bound-n}For bounded functions $f$, writing
$v_{n}:=\left\Vert P^{n}f-\mu\left(f\right)\right\Vert _{2}^{2}/\left\Vert f\right\Vert _{\mathrm{osc}}^{2}\in\left[0,\frac{1}{4}\right]$,
it holds that 
\[
n\leqslant2^{14}\cdot\varsigma^{-2}\cdot\exp\left(2\cdot\varsigma^{2}\right)\cdot L\cdot d\cdot\int_{2^{-3}\cdot v_{n}}^{2^{-3}\cdot v_{0}}\frac{v\,\mathrm{d}v}{\tilde{I}_{\pi}\left(v\right)^{2}}.
\]
\item \label{enu:thm-appli-bound-mixing}Writing $\nu$ for the initial
law of the chain, $u:=\left\Vert \frac{\mathrm{d}\nu}{\mathrm{d}\pi}\right\Vert _{\mathrm{osc}}^{2}<\infty$,
then in order to guarantee that
\[
\chi^{2}\left(\nu P^{n},\pi\right)\leqslant\varepsilon_{{\rm Mix}}\quad\text{for some }\varepsilon_{{\rm Mix}}>0,
\]
it suffices that 
\[
n\geqslant1+2^{14}\cdot\varsigma^{-2}\cdot\exp\left(2\cdot\varsigma^{2}\right)\cdot L\cdot d\cdot\left(\int_{2^{-3}\cdot\varepsilon_{{\rm Mix}}\cdot u^{-1}}^{2^{-5}}\frac{v\,\mathrm{d}v}{\tilde{I}_{\pi}\left(v\right)^{2}}\right).
\]
\end{enumerate}
\begin{proof}
For \ref{enu:thm-appli-WCP}, we aim to apply Corollary~\ref{cor:iso-couple-conductance},
whose notation we use. We have by Lemma~\ref{lem:close-coupling-concise}
that
\[
\frac{1}{2}\cdot\delta\cdot\frac{\tilde{I}_{\pi}\left(\frac{1}{2}\cdot v\right)}{\frac{1}{2}\cdot v}=\frac{1}{2}\alpha_{0}\cdot\sigma\cdot\frac{\tilde{I}_{\pi}\left(\frac{1}{2}\cdot v\right)}{\frac{1}{2}\cdot v}\leqslant\frac{1}{2}\cdot\varsigma\left(L\cdot d\right)^{-1/2}\cdot\frac{\tilde{I}_{\pi}\left(\frac{1}{4}\right)}{\frac{1}{4}}\leqslant1,
\]
and then Lemma~\ref{lem:close-coupling-concise} yields the lower
bound
\begin{align*}
\frac{1}{2}\alpha_{0}\cdot\sigma\cdot\frac{\tilde{I}_{\pi}\left(\frac{1}{2}\cdot v\right)}{\frac{1}{2}\cdot v} & \geqslant\frac{1}{2^{2}}\cdot\exp\left(-\frac{1}{2}\cdot\varsigma^{2}\right)\cdot\varsigma\left(L\cdot d\right)^{-1/2}\cdot\frac{\tilde{I}_{\pi}\left(\frac{1}{2}\cdot v\right)}{\frac{1}{2}\cdot v}.
\end{align*}
Again from Lemma~\ref{lem:close-coupling-concise} we have $\varepsilon=\frac{1}{2}\cdot\alpha_{0}\geqslant\frac{1}{4}\cdot\exp\left(-\frac{1}{2}\cdot\varsigma^{2}\right)$
and we conclude with Corollary~\ref{cor:iso-couple-conductance}
. For \ref{enu:thm-appli-Kstar-WPI} we use Theorem~\ref{thm:positive-CP-implies-optim-WPI}.
Statements \ref{enu:thm-appli-bound-n} and \ref{enu:thm-appli-bound-mixing}
follow from Theorem~\ref{thm:WPI_F_bd} and Theorem~\ref{thm:mixing-lower-bound-general}. 
\end{proof}
\begin{rem}
The condition $2\cdot\varsigma\cdot\left(L^{-1/2}\cdot\tilde{I}_{\pi}\left(\frac{1}{4}\right)\right)\leqslant d^{1/2}$
is extremely weak; in the `well-conditioned' log-concave case, the
quantity $L^{-1/2}\cdot\tilde{I}_{\pi}\left(\frac{1}{4}\right)$ is
(up to constant factors) equal to the reciprocal of the condition
number (this follows from e.g. Lemma~27 of \cite{andrieu2022explicit}),
which is always less than $1$. Noting that our bounds are almost
all optimised by taking $\varsigma\in\Theta\left(1\right)$, we expect
this condition to almost always be satisfied in practice, particularly
for high-dimensional problems.
\end{rem}

\begin{rem}
We comment briefly on the assumptions made in preparation for this
theorem. The assumption of $L$-smoothness is conventional in this
area, and functions largely as a simple way of ensuring that the acceptance
rates of the RWM chain are well-controlled. Section~4.1 of \cite{andrieu2022explicit}
discusses some relaxations of this assumption which accommodate potentials
with lower regularity. Control of the isoperimetric profile is crucial
to our approach, and appears to give quite honest bounds (``given
an accurate bound on $I_{\pi}$, one expects to reproduce the correct
rate of convergence''). It is not a strong assumption on a probability
measure to assume that it has a regular isoperimetric profile (or
minorant), but in specific practical scenarios, characterising the
profile accurately can be challenging. For applying our results in
practice, the main barrier is thus to obtain good control of the isoperimetric
profile.
\end{rem}

\subsection{Mixing times: examples}

In this paper we consider heavy-tailed distributions for which the
constant $L$ may be dimension dependent, which is to be contrasted
with the light-tailed distributions considered in \cite{andrieu2022explicit}.
In the following, $d\in\mathbb{N}$ denotes the dimension of the ambient
space $\mathsf{E}=\mathbb{R}^{d}$.
\begin{example}
\label{exa:student-t-smoothness}For the standard multivariate Student's
t-distribution in dimension $d$ with $\tau>0$ degrees of freedom,
one has 
\[
U(x)=\frac{d+\tau}{2}\cdot\log\big(\tau+\left|x\right|^{2}\big)\,,
\]
and can take $L=1+d\cdot\tau^{-1}$; see Appendix~\ref{sec:smoothness-estimates}
for details. Naturally with many degrees of freedom (i.e. $\tau\gtrsim d$),
$L$ becomes less dimension dependent, as the Student's t-distribution
tends towards a Gaussian limit.
\end{example}

\begin{example}
\label{exa:quasi-exponential-smoothness}For the $d$-variate distribution
with potential given by 
\[
U\left(x\right)=\left(\tau+\left|x\right|^{2}\right)^{\alpha/2}\text{ for }\alpha\in\left(0,1\right),
\]
one can take $L=\alpha\cdot\tau^{-\left(1-\frac{\alpha}{2}\right)}$;
see Appendix~\ref{sec:smoothness-estimates}  for details. 
\end{example}

For some initial law $\nu\ll\pi$ we let $u=\left\Vert \mathrm{d}\nu/\mathrm{d}\pi\right\Vert _{\mathrm{osc}}^{2}$
and let $\varsigma$ denote a non-dimensionalised step-size which
should be treated as being of constant order. We will use `warm start'
to denote an initialisation for which $u\in\mathcal{O}\left(1\right)$,
and `feasible start' to denote an initialisation for which $\log u\in\mathcal{O}\left(d\right)$;
one can see that this is a reasonable scaling by considering product-form
$\nu$ and $\pi$, i.e. if $\nu$, $\pi$ are mutually absolutely
continuous and satisfy $c_{-}\leqslant\frac{\mathrm{d}\nu}{\mathrm{d}\pi}\leqslant c_{+}$
for some $0<c_{-}<1<c_{+}<\infty$, then direct computations yield
that 
\[
\left\Vert \frac{\mathrm{d}\nu^{\otimes d}}{\mathrm{d}\pi^{\otimes d}}\right\Vert _{\mathrm{osc}}^{2}=c_{+}^{d}-c_{-}^{d}\geqslant c_{+}^{d}-1\in\exp\left(\Theta\left(d\right)\right)\,.
\]

Recalling Theorem~\ref{thm:WPI_F_bd}, one sees that for a Markov
chain initialised at $X_{0}\sim\nu$, it suffices to take $n\geqslant n\left(\varepsilon_{\mathrm{Mix}}\cdot u^{-1};\left\Vert \cdot\right\Vert _{\mathrm{osc}}^{2}\right)$
in order to ensure that $\chi^{2}\left(\nu P^{n},\pi\right)\leqslant\varepsilon_{\mathrm{Mix}}$.
As such, subsequent mixing time estimates will be presented in terms
of bounds on this quantity, which we shorten to $n_{\star}\left(\varepsilon_{\mathrm{Mix}};u\right)$
for convenience.

\begin{example}[Example \ref{exa:student-t-smoothness} ctd.]
For $\tau>0$, let $U\left(x\right)=\frac{d+\tau}{2}\cdot\log\big(\tau+\left|x\right|^{2}\big)$.
One can then take $L=1+d\cdot\tau^{-1}$, $\tilde{I}_{\pi}\left(p\right)=c\left(d,\tau\right)\cdot d^{-1/2}\cdot p^{1+1/\tau}$,
where $c\left(d,\tau\right)$ is bounded uniformly in $d$ for suitably
`high-dimensional' problems; see the statement of Lemma~\ref{lem:conductance-t-student-constant-bound}
for a precise statement. With $\sigma=\varsigma\cdot\left(L\cdot d\right)^{-1/2}$,
one then obtains
\begin{align*}
n_{\star}\left(\varepsilon_{\mathrm{Mix}};u\right) & \leqslant1+2^{14}\cdot\varsigma^{-2}\cdot\exp\left(2\cdot\varsigma^{2}\right)\cdot L\cdot d\cdot\left(\int_{2^{-3}\cdot\varepsilon_{{\rm Mix}}\cdot u^{-1}}^{2^{-5}}\frac{v\,\mathrm{d}v}{\tilde{I}_{\pi}\left(v\right)^{2}}\right)\\
 & \leqslant1+2^{14}\cdot\varsigma^{-2}\cdot\exp\left(2\cdot\varsigma^{2}\right)\cdot\frac{L\cdot d^{2}}{c\left(d,\tau\right)^{2}}\cdot\left(\int_{2^{-3}\cdot\varepsilon_{{\rm Mix}}\cdot u^{-1}}^{2^{-5}}\frac{\mathrm{d}v}{v^{1+2/\tau}}\right)\\
 & \leqslant1+2^{14}\cdot\varsigma^{-2}\cdot\exp\left(2\cdot\varsigma^{2}\right)\cdot\frac{L\cdot d^{2}}{c\left(d,\tau\right)^{2}}\cdot\frac{\tau}{2}\cdot\left(2^{-3}\cdot\varepsilon_{{\rm Mix}}\cdot u^{-1}\right)^{-2/\tau}\\
 & =1+2^{13+6/\tau}\cdot\varsigma^{-2}\cdot\exp\left(2\cdot\varsigma^{2}\right)\cdot\frac{L\cdot d^{2}}{c\left(d,\tau\right)^{2}}\cdot\tau\cdot\left(\frac{u}{\varepsilon_{{\rm Mix}}}\right)^{2/\tau}\\
 & =\mathcal{O}\left(d^{2}\cdot\left(\frac{u}{\varepsilon_{{\rm Mix}}}\right)^{2/\tau}\right).
\end{align*}
Thus, under a warm start, the mixing time scales like $d^{2}\cdot\varepsilon_{{\rm Mix}}^{-2/\tau}$,
whereas from a feasible start, it takes $\exp\left(\mathcal{O}\left(d\right)\right)$
to reach $\varepsilon_{{\rm Mix}}\asymp1$.
\end{example}

\begin{lem}
\label{lem:conductance-t-student-constant-bound} Let $\pi$ be as
defined in Example~\ref{exa:student-t-smoothness}. Assume that the
parameters $\left(d,\tau\right)\in\mathbb{N}\times\left(0,\infty\right)$
satisfy $\sqrt{\frac{2}{\min\left\{ d,\tau\right\} }}\leqslant\xi$
for some $\xi\in\left(0,1\right)$. One can then lower bound the
isoperimetric profile of $\pi$ as 
\[
\tilde{I}_{\pi}\left(p\right)\geqslant c\left(\tau\right)\cdot\left(\frac{1-\xi}{1+\xi}\right)^{1/2}\cdot d^{-1/2}\cdot p^{1+1/\tau},
\]
where the constant $c\left(\tau\right)>0$ is independent of the dimension.
\end{lem}

\begin{proof}
By \cite[Theorem 1.2]{bobkov2007large}, it is known that one can
minorise 
\[
\tilde{I}_{\pi}\left(p\right)\geqslant\frac{c\left(\tau\right)}{\mathrm{med}\left(\left|X\right|\right)}\cdot p^{1+\frac{1}{\tau}}\,,
\]
for some universal constant $c\left(\tau\right)>0$. It thus remains
to control $\mathrm{med}\left(\left|X\right|\right)$. For $X\sim\pi$,
one can verify that with $S:=\tau^{-1}\cdot\left|X\right|^{2}$ and
$B:=S/\left(1+S\right)$ then it holds that $B\overset{\mathrm{d}}{=}\mathrm{Beta}\left(d/2,\tau/2\right)$.
By a standard result for random variables with finite second order
moment, we have for $Z\sim\mathrm{Beta}(\alpha,\beta)$ with $\alpha,\beta>0$,
$\big|\mathbb{E}\left[Z\right]-\mathrm{med}\left(Z\right)\big|^{2}\leqslant\mathbb{E}\left[Z^{2}\right]-\mathbb{E}\left[Z\right]^{2}$,
so that
\begin{align*}
\mathrm{med}\left(Z\right) & \leqslant\frac{\alpha}{\alpha+\beta}+\sqrt{\frac{\alpha\cdot\beta}{(\alpha+\beta)^{3}}}.
\end{align*}
Translating this to our setting and noting that the median is covariant
under monotonic maps, one then computes that 
\begin{align*}
\mathrm{med}\left(\left|X\right|\right) & \leqslant\left(\tau\cdot\frac{d+\sqrt{\frac{2\cdot d\cdot\tau}{\left(d+\tau\right)}}}{\tau-\sqrt{\frac{2\cdot d\cdot\tau}{\left(d+\tau\right)}}}\right)^{1/2}.
\end{align*}
Compute that (using that $d\cdot\tau=\min\left\{ d,\tau\right\} \cdot\max\left\{ d,\tau\right\} $)
\begin{align*}
\sqrt{\frac{2\cdot d\cdot\tau}{\left(d+\tau\right)}} & =\min\left\{ d,\tau\right\} \cdot\sqrt{\frac{\max\left\{ d,\tau\right\} }{\left(d+\tau\right)}\cdot\frac{2}{\min\left\{ d,\tau\right\} }}\leqslant\min\left\{ d,\tau\right\} \cdot\xi,
\end{align*}
where the final inequality follows from the assumption that $\xi\geqslant\sqrt{\frac{2}{\min\left\{ d,\tau\right\} }}$.
We hence deduce that
\[
\frac{d+\sqrt{\frac{2\cdot d\cdot\tau}{\left(d+\tau\right)}}}{\tau-\sqrt{\frac{2\cdot d\cdot\tau}{\left(d+\tau\right)}}}\leqslant\frac{d+d\cdot\xi}{\tau-\tau\cdot\xi}=\frac{1+\xi}{1-\xi}\cdot\frac{d}{\tau},
\]
from which the result follows.
\end{proof}
\begin{example}[Product form of Example \ref{exa:student-t-smoothness}]
\label{exa:student-t-sequel} For $\eta>0$, let $U\left(x\right)=\sum_{i\in\llbracket d\rrbracket}\frac{1+\eta}{2}\cdot\log\left(1+\left|x_{i}\right|^{2}\right)$.
One can then take $L=1+\eta$, $I_{\pi}\left(p\right)\geqslant c\left(\eta\right)\cdot d^{-1/\eta}\cdot p^{1+1/\eta}$
. With $\sigma=\varsigma\cdot\left(L\cdot d\right)^{-1/2}$, one then
obtains
\begin{align*}
n_{\star}\left(\varepsilon_{\mathrm{Mix}};u\right) & \leqslant1+2^{14}\cdot\varsigma^{-2}\cdot\exp\left(2\cdot\varsigma^{2}\right)\cdot L\cdot d\cdot\left(\int_{2^{-3}\cdot\varepsilon_{{\rm Mix}}\cdot u^{-1}}^{2^{-5}}\frac{v\,\mathrm{d}v}{\tilde{I}_{\pi}\left(v\right)^{2}}\right)\\
 & \leqslant1+2^{14}\cdot\varsigma^{-2}\cdot\exp\left(2\cdot\varsigma^{2}\right)\cdot\frac{L\cdot d}{c\left(\eta\right)^{2}}\cdot d^{2/\eta}\cdot\left(\int_{2^{-3}\cdot\varepsilon_{{\rm Mix}}\cdot u^{-1}}^{2^{-5}}\frac{\mathrm{d}v}{v^{1+2/\eta}}\right)\\
 & =1+2^{14}\cdot\varsigma^{-2}\cdot\exp\left(2\cdot\varsigma^{2}\right)\cdot\frac{L\cdot d}{c\left(\eta\right)^{2}}\cdot d^{2/\eta}\cdot\frac{\eta}{2}\cdot\left(2^{-3}\cdot\varepsilon_{{\rm Mix}}\cdot u^{-1}\right)^{-2/\eta}\\
 & =1+2^{13+6/\eta}\cdot\varsigma^{-2}\cdot\exp\left(2\cdot\varsigma^{2}\right)\cdot\frac{L\cdot d^{1+2/\eta}}{c\left(\eta\right)^{2}}\cdot\eta\cdot\left(\frac{u}{\varepsilon_{{\rm Mix}}}\right)^{2/\eta}\\
 & =\mathcal{O}\left(d^{1+2/\eta}\cdot\left(\frac{u}{\varepsilon_{{\rm Mix}}}\right)^{2/\eta}\right).
\end{align*}
Thus, under a warm start, the mixing time scales like $d^{1+2/\eta}\cdot\varepsilon_{{\rm Mix}}^{-2/\eta}$,
whereas from a feasible start, it takes $\exp\mathcal{O}\left(d\right)$
to reach $\varepsilon_{{\rm Mix}}\asymp1$.
\end{example}

\begin{example}[Product form of Example \ref{exa:quasi-exponential-smoothness}]
 For $\eta\in\left(0,1\right)$, $\tau>0$, let $U\left(x\right)=\sum_{i\in\llbracket d\rrbracket}\left(\tau+\left|x_{i}\right|^{2}\right)^{\eta/2}$.
One can take $L=\eta\cdot\tau^{-\left(1-\frac{\eta}{2}\right)}$,
$I_{\pi}\left(p\right)\geqslant c\left(\eta,\tau\right)\cdot p\cdot\left(\log\left(\frac{d}{p}\right)\right)^{-\left(1/\eta-1\right)}$.
With $\sigma=\varsigma\cdot\left(L\cdot d\right)^{-1/2}$, one then
obtains
\begin{align*}
n_{\star}\left(\varepsilon_{\mathrm{Mix}};u\right) & \leqslant1+2^{14}\cdot\varsigma^{-2}\cdot\exp\left(2\cdot\varsigma^{2}\right)\cdot L\cdot d\cdot\left(\int_{2^{-3}\cdot\varepsilon_{{\rm Mix}}\cdot u^{-1}}^{2^{-5}}\frac{v\,\mathrm{d}v}{\tilde{I}_{\pi}\left(v\right)^{2}}\right)\\
 & \leqslant1+2^{14}\cdot\frac{\exp\left(2\cdot\varsigma^{2}\right)}{\varsigma^{2}}\cdot\frac{L\cdot d}{c\left(\eta,\tau\right)^{2}}\left(\int_{2^{-3}\cdot\varepsilon_{{\rm Mix}}\cdot u^{-1}}^{2^{-5}}\frac{\left(\log\left(\frac{d}{v}\right)\right)^{2\cdot\left(1/\eta-1\right)}}{v}\,\mathrm{d}v\right)\\
 & =1+2^{14}\cdot\frac{\exp\left(2\cdot\varsigma^{2}\right)}{\varsigma^{2}}\cdot\frac{L\cdot d}{c\left(\eta,\tau\right)^{2}}\left(\int_{\log\left(2^{5}\cdot d\right)}^{\log\left(2^{3}\cdot\frac{d\cdot u}{\varepsilon_{{\rm Mix}}}\right)}t^{2\cdot\left(1/\eta-1\right)}\,\mathrm{d}t\right)\\
 & \leqslant1+2^{14}\cdot\frac{\exp\left(2\cdot\varsigma^{2}\right)}{\varsigma^{2}}\cdot\frac{L\cdot d}{c\left(\eta,\tau\right)^{2}}\cdot\frac{\eta}{2-\eta}\left(\log\left(2^{3}\cdot\frac{d\cdot u}{\varepsilon_{{\rm Mix}}}\right)\right)^{2/\eta-1}\\
 & =\mathcal{O}\left(d\cdot\left(\log\left(\frac{d\cdot u}{\varepsilon_{{\rm Mix}}}\right)\right)^{2/\eta-1}\right).
\end{align*}
Thus, under a warm start, the mixing time scales like $d\cdot\left(\log\left(\frac{d}{\varepsilon_{{\rm Mix}}}\right)\right)^{2/\eta-1}$,
whereas from a feasible start, it takes $d^{2/\eta}$ to reach $\varepsilon_{{\rm Mix}}\asymp1$.
\end{example}

Relative to the light-tailed setting, we see that in the heavy-tailed
setting, we pay a more dramatic price for poor initialisation. In
particular, in Example~\ref{exa:student-t-sequel}, while the complexity
of mixing to within a given error $\varepsilon_{\mathrm{Mix}}$ scales
only polynomially with both $d$ and $\varepsilon_{\mathrm{Mix}}$
under a warm start (i.e. $u_{0}\in\mathcal{O}\left(1\right)$), the
dependence on the initial $u_{0}$ is a polynomial in $d$ of potentially
high degree in the case of quasi-exponential tails, and exponential
in $d$ in the case of polynomial tails. This highlights explicitly
that initialisation is of particular concern in the heavy-tailed setting.
We are not aware of explicit practical recommendations on how to ameliorate
this issue in practice.

\subsection{Asymptotic variance: examples}

Following the calculations in \cite{andrieu2022poincare_tech}, we
note that for sufficiently nice (i.e. $\Psi$-finite) functions $f$
and sufficiently fast-mixing kernels $P$, we can estimate the asymptotic
variance of estimates of $\pi\left(f\right)$ which are obtained as
ergodic averages along the Markov chain. In particular, following
Theorem~17 of \cite{andrieu2022poincare_tech}, for $\mu$-reversible
kernels $P$ such that $P^{*}P$ satisfies a $\left(\Psi,K^{*}\right)$\textit{\emph{--}}WPI,
if $\mathrm{id}/K^{*}$ is integrable around $0$, then there holds
the general bound
\[
\mathrm{var}\left(P,f\right)\leqslant4\cdot\Psi\left(f\right)\cdot B\left(\frac{\left\Vert f\right\Vert _{2}^{2}}{\Psi\left(f\right)}\right)\text{ with }B\left(v\right)=\int_{0}^{v}\frac{w}{K^{*}\left(w\right)}\,\mathrm{d}w\,.
\]

\begin{rem}
In the context of RWM, derivations in Theorem~\ref{thm:rwm-payoff}
tell us that 
\[
K^{*}\left(w\right)\gtrsim\left(L\cdot d\right)^{-1}\cdot\frac{\tilde{I}_{\pi}\left(w\right)^{2}}{w}\implies\frac{w}{K^{*}\left(w\right)}\lesssim L\cdot d\cdot\left(\frac{w}{\tilde{I}_{\pi}\left(w\right)}\right)^{2}.
\]
We thus intuit that if $\tilde{I}_{\pi}\left(w\right)\ll w^{3/2}$
at $0$, then the integral which defines $B$ will not be proper,
thus precluding a control of the asymptotic variance with the present
techniques.
\end{rem}

\begin{example}[Example \ref{exa:student-t-smoothness} ctd.]
 Consider the $d$-variate Student-t distribution with $\tau$ degrees
of freedom. For simplicity, assume that $\tau\geqslant d\geqslant2$.
Given the estimate $I_{\pi}\left(p\right)\gtrsim d^{-1/2}\cdot p^{1+1/\tau}$,
it follows that $\frac{w}{K^{*}\left(w\right)}\lesssim d^{2}\cdot w^{-2/\tau}$
and hence that $B\left(v\right)\lesssim d^{2}\cdot v^{1-2/\tau}$;
one then concludes that $\mathrm{var}\left(P,f\right)\lesssim d^{2}\cdot\Psi\left(f\right)^{2/\tau}\cdot\left\Vert f\right\Vert _{2}^{2\cdot\left(1-2/\tau\right)}$.
\end{example}

\begin{example}[Product form of Example \ref{exa:student-t-smoothness}]
 Consider now the product-form target with potential $U\left(x\right)=\sum_{i\in\llbracket d\rrbracket}\frac{1+\eta}{2}\cdot\log\left(1+\left|x_{i}\right|^{2}\right)$
for $\eta>0$. Given the estimate $I_{\pi}\left(p\right)\gtrsim d^{-1/\eta}\cdot p^{1+1/\eta}$,
it follows that $\frac{w}{K^{*}\left(w\right)}\lesssim d^{1+2/\eta}\cdot w^{-2/\eta}$
and hence that $B\left(v\right)\lesssim d^{1+2/\eta}\cdot v^{1-2/\eta}$;
one then concludes that $\mathrm{var}\left(P,f\right)\lesssim d^{1+2/\eta}\cdot\Psi\left(f\right)^{2/\eta}\cdot\left\Vert f\right\Vert _{2}^{2\cdot\left(1-2/\eta\right)}$.
\end{example}

\begin{example}[Product form of Example \ref{exa:quasi-exponential-smoothness}]
 Finally, consider the product-form target distribution with potential
$U\left(x\right)=\sum_{i\in\llbracket d\rrbracket}\left(\tau+\left|x_{i}\right|^{2}\right)^{\eta/2}$
for $\eta\in\left(0,1\right)$, $\tau>0$. Given the estimate $I_{\pi}\left(p\right)\gtrsim p\cdot\left(\log\left(\frac{d}{p}\right)\right)^{-\left(1/\eta-1\right)}$,
it follows that 
\[
\frac{w}{K^{*}\left(w\right)}\lesssim d\cdot\left(\log\left(\frac{d}{w}\right)\right)^{2\cdot\left(1/\eta-1\right)}\implies B\left(v\right)\lesssim d\cdot v\cdot\left(\log\left(\frac{d}{v}\right)\right)^{2\cdot\left(1/\eta-1\right)}
\]
 and one then concludes that 
\[
\mathrm{var}\left(P,f\right)\lesssim d\cdot\left\Vert f\right\Vert _{2}^{2}\cdot\left(\log\left(\frac{d\cdot\Psi\left(f\right)}{\left\Vert f\right\Vert _{2}^{2}}\right)\right)^{2\cdot\left(1/\eta-1\right)}.
\]
\end{example}

\section{Discussion\label{sec:Discussion}}

We now discuss relations to existing literature. WPIs were formally
introduced in \cite{rockner2001weak}, building on ideas of \cite{liggett1991l_2},
to study subgeometric convergence rates of diffusion processes, and
have been studied extensively in the continuous-time setting in works
such as \cite{cattiaux2010functional,bakry2014analysis}. More recently,
there has been a renewed interest in WPIs for the study of computational
statistical algorithms, often in the context of heavy-tailed distributions;
see \cite{andrieu2022comparison_journal,mousavi2023towards}. 

In terms of prior results on the convergence to equilibrium of RWM
on heavy-tailed target measures, prior work has been largely qualitative;
see e.g. \cite{jarner2002polynomial,jarner2003necessary,jarner2007convergence,douc2018markov}
based on the drift and minorisation approach. Optimal scaling-type
results (as in \cite{roberts1997weak,roberts2001optimal} where scaling
is with respect to $d$ for specific classes of targets) also offer
some insights on this setting \cite{roberts2016complexity}, though
require some care in interpretation. In particular, in the heavy-tailed
case, it appears that convergence of low-dimensional marginal distributions
may happen appreciably faster than joint convergence, as suggested
by our results. 

As in our previous work \cite{andrieu2022explicit}, the quality of
our convergence estimates for RWM is contingent on the availability
of accurate estimates on the isoperimetric profile of the target measure.
While in the log-concave case, many techniques are currently available
for this task, the situation remains less developed in the heavy-tailed
setting. In this work, we have relied heavily on the impressively
general results of \cite{bobkov2007large,cattiaux2010functional},
but one suspects that in many specific instances, it should be possible
to improve these estimates substantially, particularly in terms of
dimension-dependence. We would thus welcome additional developments
on the study of heavy-tailed isoperimetry, and heavy-tailed functional
inequalities more broadly.

\appendix

\section{Notation\label{subsec:Notation}}
\begin{itemize}
\item We will write $\mathbb{N}=\left\{ 1,2,\dots\right\} $ for the set
of natural numbers, $\mathbb{N}_{0}:=\mathbb{N}\cup\left\{ 0\right\} $,
and $\R_{+}=\left(0,\infty\right)$ for positive real numbers. 
\item For $d\in\mathbb{N}$, write $\llbracket d\rrbracket=\left\{ 1,2,\cdots,d\right\} $.
\item For $f\colon\mathbb{R}\rightarrow\mathbb{R}$ we say that $f$ is
increasing (resp. decreasing) if $a<b\implies f\left(a\right)\leqslant f\left(b\right)$
(resp. $a<b\implies f\left(a\right)\geqslant f\left(b\right)$) and
is strictly increasing (resp. strictly decreasing) if $a<b\implies f\left(a\right)<f\left(b\right)$
(resp. $a<b\implies f\left(a\right)>f\left(b\right)$).
\item For vectors $v\in\mathbb{R}^{d}$, write $\left|v\right|$ for the
Euclidean norm of $v$. For positive semi-definite matrices $M\in\mathbb{R}^{d\times d}$,
write $\left\Vert M\right\Vert _{\mathrm{spec}}=\sup\left\{ v^{\top}Mv:\left|v\right|\leqslant1\right\} $
for the spectral norm of $M$.
\item We adopt the following $\mathcal{O}$ (resp. $\Omega$) notation to
indicate when functions grow no faster than (resp. no slower than)
other functions. For $a\in\mathbb{R}\cup\left\{ \infty\right\} $,
\begin{itemize}
\item If $f(x)\in\mathcal{O}\left(g\left(x\right)\right)$ as $x\to a$,
this means that $\underset{x\to a}{\lim\sup}\left|\frac{f(x)}{g(x)}\right|<\infty$.
When $a=+\infty$, then we may drop explicit mention of $a$. 
\item If $f(x)\in\Omega\left(g\left(x\right)\right)$ as $x\to a$, this
means that $\underset{x\to a}{\lim\inf}\left|\frac{f(x)}{g(x)}\right|>0$.
In particular $f\in\mathcal{O}\left(g\right)\iff g\in\Omega\left(f\right)$. 
\end{itemize}
\end{itemize}
Outside of specific examples, we will be working throughout on a general
measurable space $\left(\E,\mathscr{E}\right)$. 
\begin{itemize}
\item For $\mu$ a probability measure on $\left(\E,\mathscr{E}\right)$
and $Z\sim\mu$, we write $\mathbb{E}\left[Z\right]$ for the expectation
of $Z$.
\item For $\left\{ Z_{n},n\in\mathbb{N}\right\} $ a sequence of random
variables, we let $Z_{n}\overset{L}{\rightarrow}Z$ indicate convergence
in distribution to the random variable $Z$.
\item For a set $A\in\mathscr{E}$, its complement in $\E$ is denoted by
$A^{\complement}$. We denote the corresponding indicator function
by $\mathbf{1}_{A}:\E\to\left\{ 0,1\right\} $.
\item We assume that $\left(\E,\mathscr{E}\right)$ is equipped with a probability
measure $\mu$, and write $\ELL\left(\mu\right)$ for the Hilbert
space of (equivalence classes of) real-valued $\mu$--square-integrable
measurable functions with inner product 
\[
\langle f,g\rangle=\int_{\E}f\left(x\right)g\left(x\right)\,\dif\mu\left(x\right)\,,
\]
and corresponding norm $\left\Vert \cdot\right\Vert _{2,\mu}$, and
if there is no ambiguity, we may just write $\left\Vert \cdot\right\Vert _{2}$.
We write $\ELL_{0}\left(\mu\right)$ for the set of functions $f\in\ELL\left(\mu\right)$
which also satisfy $\mu\left(f\right)=0$.
\item More generally, for $p\in\left[1,\infty\right)$, we write $\mathrm{L}^{p}\left(\mu\right)$
for the Banach space of real-valued measurable functions with finite
$p$-norm, $\left\Vert f\right\Vert _{p}:=\left(\int_{\E}\left|f\right|^{p}\,\dif\mu\right)^{1/p}$,
and $\mathrm{L}_{0}^{p}\left(\mu\right)$ for $f\in\mathrm{L}^{p}\left(\mu\right)$
with $\mu\left(f\right)=0$.
\item For $g\in\ELL\left(\mu\right)$, let $\var_{\mu}\left(g\right):=\left\Vert g-\pi\left(g\right)\right\Vert _{2}^{2}$.
We write $\ELL_{0}\left(\mu\right)$ for the set of functions $f\in\ELL\left(\mu\right)$
which also satisfy $\mu\left(f\right)=0$. 
\item For $\nu\in\mathcal{P}\left(\mathsf{E}\right)$ such that $\nu\ll\mu$
and $f=\frac{\mathrm{d}\nu}{\mathrm{d}\mu}\in\ELL\left(\mu\right)$,
write $\chi^{2}\left(\nu,\mu\right)$ for $\var_{\mu}\left(f\right)$.
\item For $g\in\mathrm{L}^{1}\left(\mu\right)$, let $\var_{\mu}^{\left(1\right)}\left(g\right):=\inf_{c\in\R}\left\Vert g-c\right\Vert _{1}$,
sometimes referred to as \textit{mean absolute deviation}, and we
denote by $\mathrm{med}_{\mu}\left(g\right)$ any $c\in\R$ which
attains this infimum; we will sometimes drop the subscript $\mu$
when it is clear from context.
\item For a measurable function $f:\mathsf{E}\to\R$, let $\left\Vert f\right\Vert _{\mathrm{osc}}:=\mathrm{ess_{\mu}}\sup f-\mathrm{ess}_{\mu}\inf f$
and $\left\Vert f\right\Vert _{\infty}:=\mathrm{ess_{\mu}}\sup\left|f\right|$.
\item For a finite signed measure $\nu$ on $\left(\E,\mathscr{E}\right)$
its total variation is given by 
\[
\left\Vert \nu\right\Vert _{\mathrm{TV}}=\frac{1}{2}\cdot\sup\left\{ \nu\left(f\right):\left\Vert f\right\Vert _{\infty}\leqslant1\right\} \,,
\]
which is equal to $\sup\left\{ \nu\left(f\right):\left\Vert f\right\Vert _{\osc}\leqslant1\right\} $
when $\nu(\mathsf{E})=0$, e.g. \cite[Proposition D.2.4]{douc2018markov}.
\item For two probability measures $\mu$ and $\nu$ on $\left(\E,\mathscr{E}\right)$
we let $\mu\otimes\nu\left(A\times B\right)=\mu\left(A\right)\nu\left(B\right)$
for $A,B\in\mathscr{E}$. For a Markov kernel $P\left(x,\dif y\right)$
on $\E\times\mathscr{E}$, we write for $\bar{A}\in\mathscr{E}\otimes\mathscr{E}$,
the minimal product $\sigma$-algebra, $\mu\otimes P\left(\bar{A}\right)=\int_{\bar{A}}\mu\left(\dif x\right)P\left(x,\dif y\right)$. 
\item A point mass distribution at $x$ will be denoted by $\delta_{x}\left(\dif y\right)$.
\item $\Id:\ELL\left(\mu\right)\to\ELL\left(\mu\right)$ denotes the identity
mapping, $f\mapsto f$. We also use this symbol for the identity ${\rm Id}\colon\mathsf{X}\rightarrow\mathsf{X}$.
\item For a bounded linear operator $T$, we write $T^{*}$ for its adjoint
operator $T^{*}:\ELL\left(\mu\right)\to\ELL\left(\mu\right)$, which
satisfies $\left\langle f,Tg\right\rangle =\left\langle T^{*}f,g\right\rangle $
for any $f,g\in\ELL\left(\mu\right)$. We denote its operator norm
by
\[
\left\Vert T\right\Vert _{\mathrm{op}}=\sup\left\{ \left\Vert Tf\right\Vert _{2}:f\in\ELL\left(\mu\right),\left\Vert f\right\Vert _{2}\leqslant1\right\} .
\]
\item Given a bounded linear operator $T:\ELL\left(\mu\right)\to\ELL\left(\mu\right)$
corresponding to a $\mu$-invariant Markov kernel, we define the Dirichlet
form to be 
\begin{align*}
\calE\left(T,f\right) & :=\left\langle \left(\Id-T\right)f,f\right\rangle \\
 & =\frac{1}{2}\int\mu\left(\mathrm{d}x\right)\cdot T\left(x,\mathrm{d}y\right)\cdot\left|f\left(y\right)-f\left(x\right)\right|^{2},
\end{align*}
for any $f\in\ELL\left(\mu\right)$. More generally, for $p\geqslant1$
and $f\in\mathrm{L}^{p}\left(\mu\right)$, we also define
\begin{equation}
\mathcal{E}_{p}\left(T,f\right):=\frac{1}{2}\cdot\int\mu\left(\mathrm{d}x\right)\cdot T\left(x,\mathrm{d}y\right)\cdot\left|f\left(y\right)-f\left(x\right)\right|^{p}.\label{eq:calE_p}
\end{equation}
\item  The spectrum of a bounded linear operator $T$ is the set
\[
\sigma\left(T\right):=\left\{ \lambda\in\mathbb{C}:T-\lambda\cdot{\rm Id}\text{ is not invertible}\right\} .
\]
\item For a $\mu$-invariant Markov kernel $T$, we denote by $\sigma_{0}\left(T\right)$
the spectrum of the restriction of $T$ to $\ELL_{0}\left(\mu\right)$.
We define the spectral gap of $T$ to be $\gamma_{T}:=1-\sup\left|\sigma_{0}\left(T\right)\right|$.
\item We say that a Markov kernel $T$ is $\mu$-reversible if $\mu\otimes T\left(A\times B\right)=\mu\otimes T\left(B\times A\right)$
for all sets $A,B\in\mathscr{E}$. We say that a $\mu$-reversible
kernel $T$ is \textit{positive} if for any $f\in\ELL\left(\mu\right)$,
we have $\left\langle f,Tf\right\rangle \geqslant0$.
\item If $T$ is a $\mu$-reversible Markov kernel, then $\sigma_{0}\left(T\right)\subseteq\left[-1,1\right]$
and we may define the right spectral gap as ${\rm Gap_{{\rm R}}}\left(T\right):=1-\sup\sigma_{0}\left(T\right)$,
which satisfies \cite[Theorem~22.A.19]{douc2018markov},
\[
{\rm Gap}_{\mathrm{R}}\left(T\right)=\inf_{g\in\ELL_{0}\left(\mu\right),g\neq0}\frac{\calE\left(T,g\right)}{\left\Vert g\right\Vert _{2}^{2}}\,.
\]
\item In practice, fluctuations of ergodic averages of $f\in\ELL\left(\mu\right)$
are of interest, and one may consider the \textit{asymptotic variance},
given by \cite[Proof of Theorem 4]{tierney1998note}
\[
{\rm var}\left(T,f\right):=\lim_{n\to\infty}n\cdot{\rm var}\left(\frac{1}{n}\sum_{i=1}^{n}f\left(X_{i}\right)\right),
\]
where $X_{0}\sim\mu$ and for $i\geqslant1$, $\mathbb{P}\left(X_{i}\in A\mid X_{0}=x_{0},\ldots,X_{i-1}=x_{i-1}\right)=T\left(x_{i-1},A\right)$.
\item If $T$ is a $\mu$-reversible Markov kernel we may write that \cite[Proof of Theorem 4]{tierney1998note}
\begin{align*}
{\rm var}\left(T,f\right) & =\lim_{\lambda\uparrow1}\left\langle f,\left(\Id-\lambda T\right)^{-1}\left(\Id+\lambda T\right)f\right\rangle .
\end{align*}
\end{itemize}

\section{Dirichlet forms of indicators\label{app:first-appendix}}

\begin{lem}
\label{lem:dirichlet-form-indicator}Let $T\colon\ELL\left(\mu\right)\rightarrow\ELL\left(\mu\right)$
be a bounded linear operator and $S=\frac{1}{2}\cdot\left(T+T^{*}\right)$
be its symmetric part. Then
\begin{enumerate}
\item For $p\in\left\{ 1,2\right\} $
\begin{align*}
\mathcal{E}_{p}\left(T,f\right) & =\mathcal{E}_{p}\left(T^{*},f\right)=\mathcal{E}_{p}\left(S,f\right).
\end{align*}
\item For $p\in\left\{ 1,2\right\} $ and any $A\in\mathcal{E}$
\[
\mathcal{E}_{p}\left(T,\mathbf{1}_{A}\right)=\mu\otimes T\big(A\times A^{\complement}\big)=\mu\otimes T\big(A^{\complement}\times A\big)\text{ and }\var_{\mu}\big(\mathbf{1}_{A}\big)=\mu\otimes\mu\big(A\times A^{\complement}\big)\;.
\]
\end{enumerate}
\end{lem}

\begin{proof}
We have for $p\in\left\{ 1,2\right\} $
\begin{align*}
\mathcal{E}_{p}\left(T,f\right) & =\frac{1}{2}\cdot\int\mu\left(\mathrm{d}x\right)\cdot T\left(x,\mathrm{d}y\right)\cdot\left|f\left(y\right)-f\left(x\right)\right|^{p}\\
 & =\frac{1}{2}\cdot\int\mu\left(\mathrm{d}y\right)\cdot T^{*}\left(y,\mathrm{d}x\right)\cdot\left|f\left(y\right)-f\left(x\right)\right|^{p}\\
 & =\mathcal{E}_{p}\left(T^{*},f\right),
\end{align*}
and by linearity that $\mathcal{E}_{p}\left(T,f\right)=\mathcal{E}_{p}\left(w\cdot T+\left(1-w\right)\cdot T^{*},f\right)$
for any $w\in\left[0,1\right]$. It follows that $\mathcal{E}_{p}\left(T,f\right)=\mathcal{E}_{p}\left(S,f\right)$.

Now, for arbitrary $\mu$-invariant $T$, it holds that 
\begin{align*}
\left(\mu\otimes T\right)\left(A\times A^{\complement}\right) & -\left(\mu\otimes T\right)\left(A^{\complement}\times A\right)\\
 & =\int\mu\left(\mathrm{d}x\right)\cdot T\left(x,\mathrm{d}y\right)\cdot\left[\mathbf{1}_{A}\left(x\right)\cdot1_{A^{\complement}}\left(y\right)-\mathbf{1}_{A}\left(y\right)\cdot1_{A^{\complement}}\left(x\right)\right]\\
 & =\int\mu\left(\mathrm{d}x\right)\cdot T\left(x,\mathrm{d}y\right)\cdot\left[\mathbf{1}_{A}\left(x\right)-\mathbf{1}_{A}\left(y\right)\right]\\
 & =\mu\left(A\right)-\mu T\left(A\right)=\mu\left(A\right)-\mu\left(A\right)=0,
\end{align*}
i.e. that $\left(\mu\otimes T\right)\left(A\times A^{\complement}\right)=\left(\mu\otimes T\right)\left(A^{\complement}\times A\right)$,
and that
\begin{align*}
\mathcal{E}_{p}\left(T,\mathbf{1}_{A}\right) & =\frac{1}{2}\int\left|\mathbf{1}_{A}\left(x\right)-\mathbf{1}_{A}\left(y\right)\right|^{p}\cdot\mu\otimes T\left({\rm d}x,{\rm d}y\right)\\
 & =\frac{1}{2}\int\left[\mathbf{1}_{A}\left(x\right)\cdot\mathbf{1}_{A^{\complement}}\left(y\right)+\mathbf{1}_{A^{\complement}}\left(x\right)\cdot\mathbf{1}_{A}\left(y\right)\right]\cdot\mu\otimes T\left({\rm d}x,{\rm d}y\right)\\
 & =\frac{1}{2}\left\{ \left(\mu\otimes T\right)\left(A\times A^{\complement}\right)+\left(\mu\otimes T\right)\left(A^{\complement}\times A\right)\right\} .
\end{align*}
Combining these two identities yields the first two equalities. Finally,
the result on the variance follows by considering $T\left(x,A\right)=\mu\left(A\right)$
for $\left(x,A\right)\in\mathsf{E}\times\mathscr{E}$ and the classical
identity $\var_{\mu}\big(\mathbf{1}_{A}\big)=\frac{1}{2}\mathbb{E}_{\mu\otimes\mu}\left[\big(\mathbf{1}_{A}\left(X\right)-\mathbf{1}_{A}\left(Y\right)\big)^{2}\right]$.
\end{proof}

\section{Improved sharpness of $\mathrm{L}^{1}$\textit{\emph{--}}WPI result\label{sec:Compared-sharpness-of}}

In this section we establish that the results of Theorem~\ref{thm:positive-CP-implies-optim-WPI}
improves mixing times estimates compared to those obtained using earlier
results in \cite[Theorem 38]{andrieu2022poincare_tech}. First, since
$\mu\otimes P\left(A\times A^{\complement}\right)=\mu\otimes P\left(A^{\complement}\times A\right)$
for any $A\in\mathscr{E}$, which does not require reversibility,
we have for $v\in\left[0,\frac{1}{4}\right]$,

\begin{align*}
\kappa(v): & =\inf\left\{ \frac{\mu\otimes P\left(A\times A^{\complement}\right)}{\mu\left(A\right)\mu\left(A^{\complement}\right)}:v\leqslant\mu\left(A\right)\mu\left(A^{\complement}\right)\leqslant\frac{1}{4}\right\} \\
 & =\inf\left\{ \frac{\mu\otimes P\left(A\times A^{\complement}\right)}{\mu\left(A\right)\mu\left(A^{\complement}\right)}:\frac{1-\sqrt{1-4v}}{2}\leqslant\mu\left(A\right)\leqslant\frac{1}{2}\right\} \\
 & \geqslant\inf\left\{ \frac{\mu\otimes P\left(A\times A^{\complement}\right)}{\mu\left(A\right)}:\frac{1-\sqrt{1-4v}}{2}\leqslant\mu\left(A\right)\leqslant\frac{1}{2}\right\} \\
 & =\Phi_{P}^{\mathrm{W}}\left(\frac{1-\sqrt{1-4v}}{2}\right).
\end{align*}
It is shown in \cite[Theorem 38]{andrieu2022poincare_tech} that $\kappa\left(v\right)>0$
for $v\in\left(0,\frac{1}{4}\right)$ implies a $\left(\Psi,\alpha\right)$-WPI
with $\alpha\left(v\right)=2^{4}\cdot\kappa\left(2^{-4}\cdot v\right)^{-2}$.
Using (\ref{eq:lower-bound-Kstar}) in the proof of Theorem~\ref{thm:positive-CP-implies-optim-WPI},
we obtain 
\begin{align*}
K^{*}\left(v\right) & \geqslant\frac{1}{2}\cdot v\cdot\alpha\left(\frac{v}{2}\right)^{-1}\\
 & \geqslant\frac{1}{2}\cdot v\cdot\left(2^{4}\cdot\kappa\left(2^{-5}\cdot v\right)^{-2}\right)^{-1}\\
 & \geqslant2^{-5}\cdot v\cdot\Phi_{P}^{\mathrm{W}}\left(\frac{1-\sqrt{1-2^{-3}\cdot v}}{2}\right)^{2}\cdot
\end{align*}
From Theorem~\ref{thm:WPI_F_bd}, given $f\in\ELL_{0}\left(\mu\right)$,
for any $\varepsilon>0$ we have $\left\Vert P^{n}f\right\Vert _{2}^{2}/\left\Vert f\right\Vert _{\mathrm{osc}}^{2}\leqslant\varepsilon$
whenever $n\geqslant F\left(\varepsilon\right)$, and we therefore
want the lowest possible upper bound on $F\left(\varepsilon\right)$.
In Corollary~\ref{cor:mixing-with-weak-conductance} our new lower
bound on $K^{*}(v)$ leads to
\[
F\left(\varepsilon\right)\leqslant2^{2}\cdot\int_{\varepsilon}^{1/4}\frac{{\rm d}v}{v\cdot\Phi_{P}^{\mathrm{W}}\left(2^{-2}\cdot v\right)^{2}}.
\]
We show that this is tighter than what is obtained with the old lower
bound above. We have $\frac{1}{2}x\leqslant1-\sqrt{1-x}\leqslant x$
for $x\in\left[0,1\right]$, and using that $u\mapsto\Phi_{P}^{\mathrm{W}}\left(u\right)$
is non-decreasing it holds that

\[
\Phi_{P}^{\mathrm{W}}\left(2^{-4}\cdot v\right)\leqslant\Phi_{P}^{\mathrm{W}}\left(\frac{1-\sqrt{1-2^{-3}\cdot v}}{2}\right)\leqslant\Phi_{P}^{\mathrm{W}}\left(2^{-3}\cdot v\right),
\]
and therefore
\begin{align*}
F\left(\varepsilon\right) & \leqslant2^{2}\cdot\int_{\varepsilon}^{1/4}\frac{{\rm d}v}{v\cdot\Phi_{P}^{\mathrm{W}}\left(2^{-2}\cdot v\right)^{2}}\\
 & \leqslant2^{2}\cdot\int_{\varepsilon}^{1/4}\frac{{\rm d}v}{v\cdot\Phi_{P}^{\mathrm{W}}\left(2^{-3}\cdot v\right)^{2}}\\
 & \leqslant2^{2}\cdot\int_{\varepsilon}^{1/4}\frac{{\rm d}v}{v\cdot\Phi_{P}^{\mathrm{W}}\left(\frac{1-\sqrt{1-2^{-3}\cdot v}}{2}\right)^{2}}\\
 & =2^{-3}\cdot\int_{\varepsilon}^{1/4}\frac{\mathrm{d}v}{2^{-5}\cdot v\cdot\Phi_{P}^{\mathrm{W}}\left(\frac{1-\sqrt{1-2^{-3}\cdot v}}{2}\right)^{2}},
\end{align*}
i.e. our new result improves estimates of the mixing time by a factor
of at least $8$. If $\Phi$ is moreover convex and satisfies $\Phi\left(0\right)=0$
(both of which are often the case for chains with slower-than-exponential
convergence), we can use the estimate $\Phi_{P}^{\mathrm{W}}\left(2^{-3}\cdot v\right)\leqslant2^{-1}\cdot\Phi_{P}^{\mathrm{W}}\left(2^{-2}\cdot v\right)$
to obtain an additional improvement by a factor of $4$.

\section{Auxiliary results for examples}

\subsection{Examples \ref{exa:student-t-smoothness}, \ref{exa:quasi-exponential-smoothness}\label{sec:smoothness-estimates}}
\begin{lem}
\label{lem:radial-derivs}Let $U(x)=A\left(\left|x\right|\right)$
for some sufficiently smooth $A:\mathbb{R}\to\mathbb{R}$. Then it
holds that for $x\neq0$ 
\begin{align*}
\nabla U(x) & =A^{'}\left(\left|x\right|\right)\cdot\frac{x}{\left|x\right|}\\
\nabla^{2}U(x) & =A^{''}\left(\left|x\right|\right)\cdot\left(\frac{x}{\left|x\right|}\right)\cdot\left(\frac{x}{\left|x\right|}\right)^{\top}\\
 & \quad+\frac{A^{'}\left(\left|x\right|\right)}{\left|x\right|}\cdot\left\{ I_{d}-\frac{xx^{\top}}{\left|x\right|^{2}}\right\} .
\end{align*}
For $d\geqslant2$, it thus holds for $x\neq0$ that
\[
\left\Vert \nabla^{2}U\left(x\right)\right\Vert _{\mathrm{spec}}=\max\left\{ \left|A^{''}\left(\left|x\right|\right)\right|,\left|\frac{A^{'}\left(\left|x\right|\right)}{\left|x\right|}\right|\right\} .
\]
\end{lem}

\begin{proof}
The first two results follow from direct calculation. For the estimate
on the spectral norm, let $v$ be a vector of norm $1$ and write
$v=v_{\perp}+v_{\parallel}$ with $v_{\parallel}=\langle v,x\rangle\cdot\frac{x}{\left|x\right|}$.
Compute directly that $v^{\top}\nabla^{2}U\left(x\right)v=\left|v_{\parallel}\right|^{2}\cdot A^{''}\left(\left|x\right|\right)+\left(1-|v_{\parallel}|^{2}\right)\cdot\frac{A^{'}\left(\left|x\right|\right)}{\left|x\right|}$;
optimising over $\left|v_{\parallel}\right|^{2}$ yields the claimed
result.
\end{proof}
\begin{lem}
Throughout, let $d\geqslant2$, $\tau>0$, $\alpha\in\left(0,1\right)$.
Then,
\begin{enumerate}
\item \label{enu:exampleU-1}For $U\left(x\right)=\frac{d+\tau}{2}\cdot\log\left(\tau+\left|x\right|^{2}\right)$,
it holds that 
\begin{align*}
\left\Vert \nabla^{2}U\left(x\right)\right\Vert _{\mathrm{spec}} & =\frac{\tau+d}{\tau+\left|x\right|^{2}}\\
\sup\left\{ \left\Vert \nabla^{2}U\left(x\right)\right\Vert _{\mathrm{spec}}:x\in\mathbb{R}^{d}\right\}  & =1+d\cdot\tau^{-1}.
\end{align*}
\item \label{enu:exampleU-2}For $U\left(x\right)=\left(\tau+\left|x\right|^{2}\right)^{\alpha/2}$,
it holds that
\begin{align*}
\left\Vert \nabla^{2}U\left(x\right)\right\Vert _{\mathrm{spec}} & =\alpha\cdot\left(\tau+\left|x\right|^{2}\right)^{\alpha/2-1}\\
\sup\left\{ \left\Vert \nabla^{2}U\left(x\right)\right\Vert _{\mathrm{spec}}:x\in\mathbb{R}^{d}\right\}  & =\alpha\cdot\tau^{-\left(1-\alpha/2\right)}.
\end{align*}
\end{enumerate}
\end{lem}

\begin{proof}
Explicit calculations. For \ref{enu:exampleU-1}, define $A\left(r\right)=\frac{d+\tau}{2}\cdot\log\left(\tau+r^{2}\right)$,
and compute that 
\begin{align*}
A^{'}\left(r\right) & =\left(d+\tau\right)\cdot\frac{r}{\tau+r^{2}}\\
A^{''}\left(r\right) & =\left(d+\tau\right)\cdot\frac{\tau-r^{2}}{\left(\tau+r^{2}\right)^{2}}.
\end{align*}
By Lemma~\ref{lem:radial-derivs}, it follows that if $\left|x\right|=r$,
then
\begin{align*}
\left\Vert \nabla^{2}U\left(x\right)\right\Vert _{\mathrm{spec}} & =\max\left\{ \left|A^{''}\left(r\right)\right|,\left|\frac{A^{'}\left(r\right)}{r}\right|\right\} \\
 & =\max\left\{ \left|\left(d+\tau\right)\cdot\frac{\tau-r^{2}}{\left(\tau+r^{2}\right)^{2}}\right|,\left|\left(d+\tau\right)\cdot\frac{1}{\tau+r^{2}}\right|\right\} \\
 & =\frac{\tau+d}{\left(\tau+r^{2}\right)^{2}}\cdot\max\left\{ \left|\tau-r^{2}\right|,\left|\tau+r^{2}\right|\right\} \\
 & =\frac{\tau+d}{\tau+r^{2}},
\end{align*}
from which the first result follows. Noting that this expression is
monotonically decreasing in $r$ yields the second result.

For \ref{enu:exampleU-2}, define similarly $A\left(r\right)=\left(\tau+r^{2}\right)^{\alpha/2}$,
and compute that
\begin{align*}
A^{'}\left(r\right) & =\alpha\cdot r\cdot\left(\tau+r^{2}\right)^{\alpha/2-1}\\
A^{''}\left(r\right) & =\alpha\cdot\left(\tau+r^{2}\right)^{\alpha/2-1}-\alpha\cdot\left(2-\alpha\right)\cdot r^{2}\cdot\left(\tau+r^{2}\right)^{\alpha/2-2}\\
 & =\alpha\cdot\left(\tau+r^{2}\right)^{\alpha/2-2}\cdot\left\{ \tau-\left(1-\alpha\right)\cdot r^{2}\right\} 
\end{align*}
whence if $\left|x\right|=r$, then
\begin{align*}
\left\Vert \nabla^{2}U\left(x\right)\right\Vert _{\mathrm{spec}}\\
=\max & \left\{ \left|A^{''}\left(r\right)\right|,\left|\frac{A^{'}\left(r\right)}{r}\right|\right\} \\
=\max & \left\{ \left|\alpha\cdot\left(\tau+r^{2}\right)^{\alpha/2-2}\cdot\left\{ \tau-\left(1-\alpha\right)\cdot r^{2}\right\} \right|,\left|\alpha\cdot\left(\tau+r^{2}\right)^{\alpha/2-1}\right|\right\} \\
=\alpha\cdot & \left(\tau+r^{2}\right)^{-\left(2-\alpha/2\right)}\cdot\max\left\{ \left|\left\{ \tau-\left(1-\alpha\right)\cdot r^{2}\right\} \right|,\left|\tau+r^{2}\right|\right\} \\
=\alpha\cdot & \left(\tau+r^{2}\right)^{-\left(1-\alpha/2\right)},
\end{align*}
yielding the first result; monotonicity again yields the second.
\end{proof}
\begin{lem}
Let $U:\mathbb{R}^{d}\to\mathbb{R}$ be a twice-differentiable mapping
such that $\sup\left\{ \left\Vert \nabla^{2}U\left(x\right)\right\Vert _{\mathrm{spec}}:x\in\mathbb{R}^{d}\right\} \leqslant L$
for some $L\in[0,\infty)$. Then $U$ is $L$-smooth.
\end{lem}

\begin{proof}
Let $x,h\in\E$, and write
\begin{align*}
U\left(x+h\right)-U\left(x\right)-\left\langle \nabla U\left(x\right),h\right\rangle  & =\int_{0}^{1}\left\langle \nabla U\left(x+t\cdot h\right)-\nabla U\left(x\right),h\right\rangle \,\mathrm{d}t\\
 & =\int_{0}^{1}\int_{0}^{t}h^{\top}\nabla^{2}U\left(x+s\cdot h\right)h\,\mathrm{d}s\,\mathrm{d}t\\
 & \leqslant\int_{0}^{1}\int_{0}^{t}\left\Vert \nabla^{2}U\left(x+s\cdot h\right)\right\Vert _{\mathrm{spec}}\cdot\left|h\right|^{2}\,\mathrm{d}s\,\mathrm{d}t\\
 & \leqslant\int_{0}^{1}\int_{0}^{t}L\cdot\left|h\right|^{2}\,\mathrm{d}s\,\mathrm{d}t\\
 & =\frac{L}{2}\cdot\left|h\right|^{2}.
\end{align*}
\end{proof}

\subsection{Example \ref{exa:PM-Pareto} \label{app:exa-pm-pareto}}

A ${\rm Pareto}\left(\alpha,x_{m}\right)$ with $x_{m}=1-\alpha^{-1}$
and $\alpha>1$ for $Q_{*}$ leads to a cumulative distribution function
for $\tilde{\pi}_{*}$
\[
\tilde{\pi}_{*}\left(W\leqslant s\right)=\begin{cases}
0 & s\leqslant x_{m}\\
1-\left(x_{m}/s\right)^{\alpha-1} & s\geqslant x_{m},
\end{cases}
\]
because
\begin{align*}
\alpha x_{m}^{\alpha}\int_{x_{m}}^{s}x^{-(1+\alpha)}x\,{\rm d}x & =\frac{\alpha x_{m}^{\alpha}}{\alpha-1}\cdot\left[x_{m}^{-(\alpha-1)}-s^{-(\alpha-1)}\right]\\
 & =\frac{\alpha x_{m}}{\alpha-1}\cdot\left[1-\left(x_{m}/s\right)^{\alpha-1}\right]\\
 & =1-\left(x_{m}/s\right)^{\alpha-1},
\end{align*}
using basic properties of the Pareto distribution. The expression
for $\psi_{*}^{-1}$ therefore follows.

\bibliographystyle{plain}
\bibliography{bib-rupi}

\begin{thebibliography}{10}

\bibitem{andrieu2022comparison_journal}
Christophe Andrieu, Anthony Lee, Sam Power, and Andi~Q. Wang.
\newblock {Comparison of Markov chains via weak Poincar{\'e} inequalities with
  application to pseudo-marginal MCMC}.
\newblock {\em The Annals of Statistics}, 50(6):3592 -- 3618, 2022.

\bibitem{andrieu2022explicit}
Christophe Andrieu, Anthony Lee, Sam Power, and Andi~Q. Wang.
\newblock Explicit convergence bounds for {M}etropolis {M}arkov chains:
  isoperimetry, spectral gaps and profiles, 2022.

\bibitem{andrieu2022poincare_tech}
Christophe Andrieu, Anthony Lee, Sam Power, and Andi~Q. Wang.
\newblock {Poincar\'e inequalities for Markov chains: a meeting with Cheeger,
  Lyapunov and Metropolis}.
\newblock Technical report: \url{https://arxiv.org/abs/2208.05239}, 2022.

\bibitem{andrieu2009pseudo}
Christophe Andrieu and Gareth~O. Roberts.
\newblock {The pseudo-marginal approach for efficient Monte Carlo
  computations}.
\newblock {\em The Annals of Statistics}, 37(2):697 -- 725, 2009.

\bibitem{andrieu2015convergence}
Christophe Andrieu and Matti Vihola.
\newblock Convergence properties of pseudo-marginal {M}arkov chain {M}onte
  {C}arlo algorithms.
\newblock {\em The Annals of Applied Probability}, 25(2):1030--1077, 2015.

\bibitem{andrieu2016establishing}
Christophe Andrieu and Matti Vihola.
\newblock Establishing some order amongst exact approximations of {MCMC}s.
\newblock {\em The Annals of Applied Probability}, 26(5):2661--2696, 10 2016.

\bibitem{bakry2014analysis}
Dominique Bakry, Ivan Gentil, Michel Ledoux, et~al.
\newblock {\em Analysis and geometry of Markov diffusion operators}, volume
  103.
\newblock Springer, 2014.

\bibitem{baxendale2005renewal}
Peter~H. Baxendale.
\newblock {Renewal theory and computable convergence rates for geometrically
  ergodic Markov chains}.
\newblock {\em The Annals of Applied Probability}, 15(1B):700--738, 2005.

\bibitem{bobkov2007large}
Sergey Bobkov.
\newblock {Large deviations and isoperimetry over convex probability measures
  with heavy tails}.
\newblock {\em Electronic Journal of Probability}, 12(none):1072 -- 1100, 2007.

\bibitem{bobkov2009distributions}
Sergey Bobkov and Boguslaw Zegarlinski.
\newblock Distributions with slow tails and ergodicity of markov semigroups in
  infinite dimensions.
\newblock {\em Around the Research of Vladimir Maz'ya I: Function Spaces},
  pages 13--79, 2009.

\bibitem{bobkov2009weighted}
Sergey~G Bobkov and Michel Ledoux.
\newblock Weighted poincar{\'e}-type inequalities for cauchy and other convex
  measures.
\newblock 2009.

\bibitem{bobkov1997some}
Serguei~Germanovich Bobkov and Christian Houdr{\'e}.
\newblock {\em Some connections between isoperimetric and Sobolev-type
  inequalities}, volume 616.
\newblock American Mathematical Soc., 1997.

\bibitem{cattiaux2012central}
Patrick Cattiaux, Djalil Chafai, and Arnaud Guillin.
\newblock {Central limit theorems for additive functionals of ergodic Markov
  diffusions processes}.
\newblock {\em ALEA, Lat. Am. J. Probab. Math. Stat}, 9(2):337--382, 2012.

\bibitem{cattiaux2010functional}
Patrick Cattiaux, Nathael Gozlan, Arnaud Guillin, and Cyril Roberto.
\newblock Functional inequalities for heavy tailed distributions and
  application to isoperimetry.
\newblock {\em Electronic Journal of Probability}, 15:346--385, 2010.

\bibitem{cattiaux2010poincare}
Patrick Cattiaux, Arnaud Guillin, and Cyril Roberto.
\newblock {Poincar{\'e} inequality and the $L^p$ convergence of semi-groups}.
\newblock {\em Electronic Communications in Probability}, 15(none):270 -- 280,
  2010.

\bibitem{chen2020fast}
Yuansi Chen, Raaz Dwivedi, Martin~J. Wainwright, and Bin Yu.
\newblock {Fast mixing of Metropolized Hamiltonian Monte Carlo: Benefits of
  multi-step gradients}.
\newblock {\em Journal of Machine Learning Research}, 21(92):1--72, 2020.

\bibitem{diaconis1996nash}
Persi Diaconis and Laurent Saloff-Coste.
\newblock {Nash inequalities for finite Markov chains}.
\newblock {\em Journal of Theoretical Probability}, 9(2):459--510, 1996.

\bibitem{douc2018markov}
Randal Douc, Eric Moulines, Pierre Priouret, and Philippe Soulier.
\newblock {\em {Markov Chains}}.
\newblock Springer, 2018.

\bibitem{dwivedi2019logconcave}
Raaz Dwivedi, Yuansi Chen, Martin~J. Wainwright, and Bin Yu.
\newblock {Log-concave sampling: Metropolis--Hastings algorithms are fast}.
\newblock {\em Journal of Machine Learning Research}, 20(183):1--42, 2019.

\bibitem{gong2006spectral}
Fuzhou Gong and Liming Wu.
\newblock {Spectral gap of positive operators and applications}.
\newblock {\em Journal de Math{\'{e}}matiques Pures et Appliqu{\'{e}}es},
  85(2):151--191, 2006.

\bibitem{jarner2002polynomial}
S{\o}ren~F. Jarner and Gareth~O. Roberts.
\newblock {Polynomial convergence rates of Markov chains}.
\newblock {\em The Annals of Applied Probability}, 12(1):224--247, 2002.

\bibitem{jarner2007convergence}
S{\o}ren~F. Jarner and Gareth~O. Roberts.
\newblock {Convergence of heavy-tailed Monte Carlo Markov chain algorithms}.
\newblock {\em Scandinavian Journal of Statistics}, 34(4):781--815, 2007.

\bibitem{jarner2003necessary}
S{\o}ren~F. Jarner and Richard~L. Tweedie.
\newblock {Necessary conditions for geometric and polynomial ergodicity of
  random-walk-type}.
\newblock {\em Bernoulli}, 9(4):559 -- 578, 2003.

\bibitem{lawler1988bounds}
Gregory~F. Lawler and Alan~D. Sokal.
\newblock {Bounds on the $L^2$ spectrum for Markov chains and Markov processes:
  a generalization of Cheeger's inequality}.
\newblock {\em Transactions of the American Mathematical Society},
  309(2):557--580, 1988.

\bibitem{lee2014variance}
Anthony Lee and Krzysztof {\L}atuszy{\'n}ski.
\newblock {Variance bounding and geometric ergodicity of Markov chain Monte
  Carlo kernels for approximate Bayesian computation}.
\newblock {\em Biometrika}, 101(3):655--671, 08 2014.

\bibitem{liggett1991l_2}
Thomas~M Liggett.
\newblock {$L_2$ rates of convergence for attractive reversible nearest
  particle systems: the critical case}.
\newblock {\em The Annals of Probability}, 19(3):935--959, 1991.

\bibitem{maxwell2000central}
Michael Maxwell and Michael Woodroofe.
\newblock {Central limit theorems for additive functionals of Markov chains}.
\newblock {\em The Annals of Probability}, pages 713--724, 2000.

\bibitem{metropolis1953equation}
Nicholas Metropolis, Arianna~W. Rosenbluth, Marshall~N. Rosenbluth, Augusta~H.
  Teller, and Edward Teller.
\newblock {Equation of State Calculations by Fast Computing Machines}.
\newblock {\em The Journal of Chemical Physics}, 21(6):1087--1092, 1953.

\bibitem{meyn1993markov}
S.P. Meyn and R.L. Tweedie.
\newblock {\em {M}arkov chains and stochastic {s}tability}.
\newblock Cambridge University Press, 2 edition, 2009.

\bibitem{montenegro2006mathematical}
Ravi Montenegro and Prasad Tetali.
\newblock {Mathematical aspects of mixing times in Markov chains}.
\newblock {\em Foundations and Trends{\textregistered} in Theoretical Computer
  Science}, 1(3):237--354, 2006.

\bibitem{mousavi2023towards}
Alireza Mousavi-Hosseini, Tyler Farghly, Ye~He, Krishnakumar Balasubramanian,
  and Murat~A Erdogdu.
\newblock Towards a complete analysis of langevin monte carlo: Beyond
  poincar\'e inequality.
\newblock 3 2023.

\bibitem{roberts1997weak}
Gareth~O. Roberts, Andrew Gelman, and Walter~R. Gilks.
\newblock {Weak convergence and optimal scaling of random walk Metropolis
  algorithms}.
\newblock {\em The Annals of Applied Probability}, 7(1):110--120, 1997.

\bibitem{roberts2001optimal}
Gareth~O. Roberts and Jeffrey~S. Rosenthal.
\newblock {Optimal scaling for various Metropolis--Hastings algorithms}.
\newblock {\em Statistical Science}, 16(4):351--367, 2001.

\bibitem{Roberts2008}
Gareth~O. Roberts and Jeffrey~S. Rosenthal.
\newblock {Variance bounding Markov chains}.
\newblock {\em Ann. Appl. Probab.}, (3):1201--1214, 2008.

\bibitem{roberts2016complexity}
Gareth~O. Roberts and Jeffrey~S. Rosenthal.
\newblock Complexity bounds for markov chain monte carlo algorithms via
  diffusion limits.
\newblock {\em Journal of Applied Probability}, 53(2):410--420, 2016.

\bibitem{roberts1996geometric}
Gareth~O. Roberts and Richard~L. Tweedie.
\newblock {Geometric convergence and central limit theorems for
  multidimensional Hastings and Metropolis algorithms}.
\newblock {\em Biometrika}, 83(1):95--110, 1996.

\bibitem{roberts2001geometric}
Gareth~O Roberts and Richard~L Tweedie.
\newblock Geometric $l^2$ and $l^1$ convergence are equivalent for reversible
  markov chains.
\newblock {\em Journal of Applied Probability}, 38(A):37--41, 2001.

\bibitem{rockner2001weak}
Michael R{\"{o}}ckner and Feng-Yu Wang.
\newblock {Weak Poincar\'{e} inequalities and L2 convergence rates of Markov
  semigroups}.
\newblock {\em Journal of Functional Analysis}, 185:564--603, 2001.

\bibitem{tierney1998note}
Luke Tierney.
\newblock {A note on Metropolis--Hastings kernels for general state spaces}.
\newblock {\em The Annals of Applied Probability}, 8(1):1--9, 1998.

\bibitem{wang2014criteria}
Feng~Yu Wang.
\newblock {Criteria of spectral gap for Markov operators}.
\newblock {\em Journal of Functional Analysis}, 266(4):2137--2152, 2014.

\bibitem{zitt2008annealing}
Pierre-Andr{\'e} Zitt.
\newblock Annealing diffusions in a potential function with a slow growth.
\newblock {\em Stochastic Processes and their Applications}, 118(1):76--119,
  2008.

\end{thebibliography}

\end{document}